\documentclass{tac}
\pdfoutput=1
\usepackage{graphicx, hyperref}
\usepackage{amsmath}
\usepackage{amssymb}
\usepackage{slashed}
\usepackage{stmaryrd}
\usepackage{yfonts}
\usepackage[T1]{fontenc}
\usepackage[utf8]{inputenc}
\usepackage[english]{babel}
\usepackage{tikz-cd}
\usepackage{mathabx}
\usepackage{braket}
\usepackage{physics}
\usepackage[
backend=biber,
style=alphabetic,
sorting=none,
]{biblatex}
\usepackage{tikz-cd}
\usetikzlibrary{shapes}
\usepackage{xcolor, tcolorbox}

\theoremstyle{theorem}
\newtheorem{theorem}{Theorem}[section]
\newtheorem{lemma}{Lemma}[section]
\newtheorem{corollary}{Corollary}[section]
\newtheorem{proposition}{Proposition}[section]

\theoremstyle{remark}
\newtheorem{remark}{Remark}[section]

\theoremstyle{example}

\theoremstyle{definition}
\newtheorem{definition}{Definition}[section]

\usepackage{mathtools}

\usepackage{mathrsfs}

\numberwithin{equation}{section}

\newcommand{\beq}{\begin{eqnarray}}
\newcommand{\eeq}{\end{eqnarray}}

\hypersetup{
pdfstartview = {FitH},
}
\hypersetup{
	colorlinks=true,         
	linkcolor=brown,          
	citecolor=red,        
	urlcolor=blue            
}

\def\d{\mathfrak d}
\def\g{\mathfrak g}

\def\r{\mathfrak r}

\def\T{\mathfrak{T}}
\def\End{\mathfrak{End}}

\def\bbZ{\mathbb{Z}}

\def\bbC{\mathbb{C}}

\newcommand{\cC}{{\cal C}}
\newcommand{\cD}{{\cal D}}

\newcommand{\cT}{{\cal T}}
\newcommand{\cG}{{\cal G}}

\newcommand{\cH}{{\cal H}}

\newcommand{\cK}{{\cal K}}

\newcommand{\cN}{{\cal N}}

\newcommand{\cR}{{\cal R}}

\newcommand{\id}{\operatorname{id}}

\usepackage{authblk}
  \allowdisplaybreaks

\begin{document}
\title{Hopf 2-Algebras and Braided Monoidal 2-Categories}

\author[1]{{ \sf Hank Chen}}\thanks{hank.chen@uwaterloo.ca}


\author[1]{{\sf Florian Girelli}}\thanks{fgirelli@uwaterloo.ca}

\affil[1]{\small Department of Applied Mathematics, University of Waterloo, 200 University Avenue West, Waterloo, Ontario, Canada, N2L 3G1}



\maketitle

\begin{abstract}
    Following the theory of principal $\infty$-bundles of Niklaus-Schreiber-Steveson, we develop a homotopy categorification of Hopf algebras, which model quantum groups. We study their higher-representation theory in the setting of $\mathsf{2Vect}^{hBC}$, which is a homotopy refinement of the notion of 2-vector spaces due to Baez-Crans that allows for higher coherence data. We construct in particular the \textit{2-quantum double} as a homotopy double crossed product, and prove its duality and factorization properties. We also define and characterize "{2-$R$-matrices}", which can be seen as an extension of the usual notion of $R$-matrix in an ordinary Hopf algebra. We found that the 2-Yang-Baxter equations describe the braiding of extended defects in 4d, distinct from but not unlike the the Zamolodchikov tetrahedron equations. The main results we prove in this paper is that the 2-representation 2-category of a weak 2-bialgebra is braided monoidal if it is equipped with a universal 2-$R$-matrix, and that our homotopy quantization admits the theory of Lie 2-bialgebras as a semiclassical limit.
\end{abstract}

\bigskip

\noindent {\scriptsize AMS subject classification: 17B37, 81R50, 18N25.\\ Keywords: Quantum groups, representation theory, categorification, quantum field theory}

\tableofcontents

\section{Introduction}
It is well-known that the algebra of excitations in 3D  BF theory   (equivalent to a Chern-Simons theory \cite{Witten:1988hc, Osei:2013xra,chen:2022})  for a given gauge group  is described by its Drinfel'd double quantum group \cite{Delcamp:2016yix,Dupuis:2020ndx,Osei:2013xra}, which forms a quasitriangular Hopf algebra \cite{Majid:1990bt, Majid:1996kd}. The tensor category of its representations has equipped very important monoidal structures that capture the essential topological properties of the underlying topological field theory. In particular, the universal $R$-matrix equipped on a Hopf algebra $H$ defines the braiding map $b=\text{flip}\circ R$ on its representation category $\operatorname{Rep}(H)$. 
 
 The fact that the modules of such quantum group Hopf algebras are {braided} means that $\operatorname{Rep}(H)$ furnishes a representation of the Artin braid group \cite{Majid:1990bt}, which allows quantum invariants of knots and tangles to be developed from the Hopf algebra machinery. This was first pointed out by Witten \cite{WITTEN1990285} in which the Jones polynomial knot invariants were recovered from modules of the quantum $U_q(\mathfrak{su}(2))$ symmetry of $SU(2)$-Chern-Simons theory. Indeed, the deep connection between Hopf algebras and 3d geometry/topology is well-known \cite{Hopf:1941,Borel:1953}. 

 The well-known 2d toric code \cite{Kitaev1997} is another fruitful example of such construction. It  has been shown to be described by a 3d $\bbZ_2$-gauge BF theory with the underlying Drinfel'd double $D(\bbZ_2)$ symmetry. Moreover, the theory of non-degenerate fusion categories, modelling 3-dimensional gapped topological phases and their boundary excitations, are now very well-understood \cite{KitaevKong_2012,Wen:2010gda,Kong:2014qka}.

The characterization of quasitriangular Hopf algebras/quantum groups, as well as its $R$-matrix,  plays also a central role in the theory of quantum and classical integrable systems \cite{Tomei:2013,Adler:1978,book-integrable,Meusburger:2021cxe}. In particular for 1-dimensional spin chains such as the Toda lattice \cite{Adler:1978} or the XXX/XXZ/XYZ family of Ising spins \cite{Zhang:1991}, obtaining the quantum $R$-matrix is equivalent to solving the entire model, through the method known as \textit{quantum inverse scattering}. 

\medskip

One might wonder how these different examples extend to a dimension up. For example, one might wonder what is the right tool to characterize the algebra of excitations in a \textit{4d}  BF theory, or any higher-dimensional topological phases in general \cite{Bullivant:2016clk,Wen:2019,Kong:2020}, such as the 4d toric code \cite{Hamma:2004ud,Kong:2020wmn,Else:2017yqj}. In a similar way, given the well-studied integrable systems are typically in  1+1 dimension, one could wonder what would be the relevant structure for an integrable system in 2+1 dimension. 
 
According to the dimensional (or categorical) ladder proposal \cite{Crane:1994ty, Baez:1995xq, Mackaay:ek}, one way to obtain the relevant structure is through \textit{categorification}. In fact, several constructions already point  to the efficiency of using  2-categories (arising from for instance higher-representation theory \cite{neuchl1997representation,Baez:2012,Douglas:2018,Delcamp:2023kew}) to describe a 4-dimensional gapped topological phases, such as the 4d toric code \cite{Hamma:2004ud,Kong:2020wmn,Else:2017yqj}. 

The goal of this paper is to describe a categorification of the  rich and fruitful theory of quantum groups, starting from the theory of higher-dimensional $L_\infty$-algebras of \cite{Baez:2003fs,Baez:2004in,Nikolaus_2014,costello_gwilliam_2016}. Together with the quantization step, which we understand as taking a Lie algebra to a quantum Hopf algebras (cf. Drinfel'd-Jimbo quantization \cite{Majid:1996kd,Semenov1992,Drinfeld:1986in} and Kontsevich $\star$-quantization \cite{Kontsevich:1997vb}), we arrive at the following diagram,
\begin{equation}
        \begin{tikzcd}
            \text{Lie (or Poisson) algebra} \arrow[rr,"\text{categorify}"] \arrow[d,"\text{quantize}"] & & \text{(Poisson) } L_\infty\text{-algebras} \arrow[d,"\text{quantize}",dashed]\\
            \text{(Hopf) algebras} \arrow[rr,"\text{categorify}"] & & \text{(Hopf) } A_\infty\text{-algebras}
        \end{tikzcd}.\label{dimladder}
\end{equation}
As such, focusing on 4-dimensions, it follows that a {\bf categorical quantum group} living on the bottom-right corner of \eqref{dimladder} can be modelled as a "2-bialgebra" $\cG$ in the $A_\infty$ context. We shall construct such a structure starting from the strict/associative 2-algebras of \cite{Wagemann+2021}, then consider a weakening of the associativity up to homotopy expressed by a \textit{Hochschild 3-cocycle} $\cT$ on $\cG$. Moreover, we will substantiate the diagram \eqref{dimladder} in Section \ref{claslim}, where we prove that we recover the known notions of Lie 2-bialgebras \cite{Bai_2013,Chen:2012gz,Chen:2013} by taking an appropriate "classical limit". 

A structure of particular interest that we will also construct is the {\it 2-quantum double}, which is a "categorical Drinfel'd double" generalizing the works of Majid \cite{Majid:1990bt,Majid:1994nw,Majid:1996kd}. We will prove some crucial structural theorems about this 2-quantum double construction, such as quasitriangularity and self-duality. We explicitly demonstrate that these proofs run parallel to classical results for ordinary quantum doubles. 

\medskip

Next, we consider higher-representations of the weak 2-bialgebra $\cG$ by examining its action on 2-vector spaces. We emphasize that the 2-representation theory based on the Baez-Crans notion of 2-vector spaces \cite{Angulo:2018} $\mathsf{2Vect}^{BC}$, which is equivalently a category internal to $\mathsf{Vect}$ \cite{Baez:2003fs}, is insufficient for our purposes \cite{heredia2016representations2groupsbaezcrans2vector}. Instead, we shall base our "weak 2-representation theory" on a modified version of $\mathsf{2Vect}^{BC}$, which we denote simply by $\mathsf{2Vect}^{hBC}$. More precisely, if we think of $V\in\mathsf{2Vect}^{BC}$ as a $\mathsf{Vect}$-algebra in the bicategory $\mathsf{Cat}$ of (small) categories, then $V\in\mathsf{2Vect}^{hBC}$ can be thought of as as \textit{psuedo}-$\mathsf{Vect}$-algebra \cite{Fiore2004PseudoLB}.

The 1- and 2-morphisms in this 2-category $\mathsf{2Vect}^{hBC}$ are modelled respectively by cochain maps and cochain homotopies equipped with homotopy witnesses for associativity. In particular, the endormophism 1-categories of objects in $\mathsf{2Vect}^{hBC}$ by definition have the structure of a 2-term $A_\infty$-algebra. We shall show how the theory of $A_\infty$-algebras \cite{Stasheff:1963} provides the necessary commutativity of all diagrams in this 2-category $\mathsf{2Vect}^{hBC}$.

We denote the resulting 2-category of weak 2-representations by $\operatorname{2Rep}^\cT(\cG)$. Weakening the 2-representations is sufficient, and in fact \textit{necessary} \cite{heredia2016representations2groupsbaezcrans2vector}, in order for it to carry higher-homotopical data. These play very important roles in the following main result of this paper.
\begin{theorem}\label{mainthm}
The 2-representation 2-category $\operatorname{2Rep}^\cT(\cG)$ of a weak 2-bialgebra $(\cG,\cT,\Delta)$ is braided monoidal ({\'a} la Gurski \cite{GURSKI20114225}), with trivial left-/right-unitors, if $\cG$ is equipped with a universal 2-$R$-matrix $\cR$.
\end{theorem}
\noindent We will, in particular, universally characterize the quantum 2-$R$-matrices on $\cG$ which is responsible for the braiding, and provide  the categorified \textbf{2-Yang-Baxter equations} it satisfies. The theorem is then proven by translating the 2-algebraic properties to structures of the 2-representations, and explicitly checking all coherence diagrams \cite{GURSKI20114225,KongTianZhou:2020}. We will further prove in the appendix that the weak 2-representation theory we develop here do indeed host the necessary homotopy data as that studied in the literature \cite{Delcamp:2021szr,Delcamp:2023kew,Douglas:2018,Baez:2012}.  

\medskip

There is then naturally a forgetful functor $\operatorname{2Rep}(\cG)\rightarrow \mathsf{2Vect}^{hBC}$ into the 2-category of {\it weak} 2-vector spaces discussed briefly above, and {\it not} the Baez-Crans 2-category $\mathsf{2Vect}^{BC}$. However, $\mathsf{2Vect}^{hBC}$ does \textit{not} coincide with 2-vector spaces of the Karpanov-Voevodskey type $\mathsf{2Vect}^{KV}$ \cite{Kapranov:1994} --- this issue is currently under investigation by one of the authors. 




\subsubsection*{Outline}

Section \ref{strict2bialgebra} constructs a {\it strict} 2-bialgebra (where associativity is retained) following the definition of a strict 2-algebra/algebra crossed-module in \cite{Wagemann+2021}. In Section \ref{strict2bialg}, we begin by introducing a {\it graded} coproduct $\Delta$ on $\cG$, and define the appropriate notion for two 2-bialgebras to be {\it dually paired}. The 2-bialgebra axiom plays a central role in this notion of duality, just as in the case of the usual 1-bialgebra. We also provide a classic example of the function 2-bialgebra on a 2-group, and verify explicitly the 2-bialgebra axioms. 

\smallskip

Section \ref{str2qd}  delves into the construction of the 2-quantum double $D(\cG)$ associated to a pair of mutually dual 2-bialgebras $\cG,\cG^*$. We prove key structural theorems about $D(\cG)$, such as its factorizability and self-duality. Furthermore, motivated by Majid's construction of the generalized double \cite{Majid:1994nw}, based on the existence of a quantum  $R$-matrix, we introduce the notion of 2-$R$-matrix $\cR$ and derive  explicitly the "2-Yang-Baxter equations" that $\cR$ satisfies.  We demonstrate in particular how $D(\cG)$ is in fact naturally equipped with a universal 2-$R$-matrix.   

\smallskip

In Section \ref{weak2bialgebras}, we weaken our 2-bialgebra construction by introducing a {\it Hochschild 3-cocycle} that witnesses associativity. We note that we obtain precisely a 2-bialgebra in the $A_\infty$ context, which fits into the diagram \eqref{dimladder} as desired. We prove that there is once again a well-defined notion of duality in this case by extending the 2-bialgebra axioms to take the weakened associativity/coassociativity into account. 

\smallskip

In Section \ref{weakskeletalqd}, we generalize the 2-quantum double construction, as well as their structural theorems, to the weak case. We specialize to the skeletal case, where the structural $t$-map is trivial, in order to leverage the construction given in the strict case.

\smallskip

In Section \ref{2representations}, we discuss directly the notion of weak 2-representations $\operatorname{2Rep}(\cG)$ living in $\mathsf{2Vect}^{hBC}$, based on a modified notion of Baez-Crans 2-vector spaces. We discuss the different key aspects of the 2-category $\operatorname{2Rep}^\cT(\cG)$, namely the monoidal structure and the braiding. We will in particular prove the naturality of these struvtures form the underlying properties of the 2-bialgebra.

\smallskip

Finally, in Section \ref{braiding2rep}, we prove  the main theorem  by explicitly checking all the relevant coherence diagrams of a braided monoidal 2-category. We will moreover emphasize how the fusion associators/pentagonators and braiding hexagonators arise from the homotopy data attached to a weak quasitriangular 2-bialgebra $(\cG,\cT,\cR)$, as well as its weak 2-representation theory $\operatorname{2Rep}^\cT(\cG)$.

\medskip

We collect some supplementary but also important results that we have obtained in the Appendix. In Section \ref{2-hopf}, we provide an appropriate definition of an antipode in the strict case, and examine briefly its properties. 

In Section \ref{claslim}, we prove that under an appropriate "classical limit", the quantum 2-bialgebra and the universal 2-$R$-matrix that we have defined in fact recovers the Lie 2-bialgebra and the classical 2-$r$-matrix as studied in \cite{Bai_2013,Chen:2012gz,Chen:2013}.  

\smallskip

And finally, in Section \ref{weak2repthy}, we will demonstrate that the theory of weak 2-representations, as developed in this paper, has the same homotopy theory as the 2-representation theory of skeletal 2-groups as studied in the literature \cite{Douglas:2018,Baez:2012,Delcamp:2021szr}. The latter is typically understood to be a {\it symmetric} monoidal 2-category, but our theory allows for non-trivial braiding data for the 2-representations. 

\subsubsection*{Acknowledgements}
The authors would like to thank Matthew Yu, Theo Johnson-Freyd, David Green, Liang Kong, and Matt Hogencamp for valuable discussions and suggestions throughout the completion of this work.

\section{Strict 2-bialgebras}
\label{strict2bialgebra}
Quantum groups are Hopf algebras, hence we expect to define quantum 2-groups as "Hopf 2-algebras". Different  notions of  2-Hopf algebra have already been previously proposed in \cite{Majid:2012gy,Wagemann+2021,Pfeiffer2007,grosse2001suggestion}, but we shall work primarily in the context of \textit{Baez-Crans 2-vector spaces} \cite{Baez:2003fs}. 

A \textbf{Baez-Crans 2-vector space} $V\in \mathsf{2Vect}^{BC}$ is a category internal to the category $\mathsf{Vect}$ of $\bbC$-vector spaces. The following characterization result was obtained in \cite{Baez:2003fs}. 
\begin{proposition}\label{2ch}
    There is an equivalence $\mathsf{2Vect}^{BC} \simeq\mathsf{2Ch}(\mathsf{Vect})$ with the 2-category of 2-term $\bbC$-vector space chain complexes.
\end{proposition}
\noindent One side of the equivalence, given a $V=V_1\overset{s}{\underset{t}{\rightrightarrows}}V_0\in \mathsf{2Vect}^{BC}$, is
\begin{equation*}
    V_{-1} = \operatorname{ker}s,\qquad V_0=V_0,
\end{equation*}
whence the map $t:V_{-1}\rightarrow V_0$ defined by $$\mu_1(f) = x'-x,\qquad f:x\rightarrow x' \in V$$ determines a 2-term chain complex.

\subsection{Associative 2-algebras}
We begin with the following definition, then build up to the definition of an associative 2-algebra in \cite{Wagemann+2021}.
\definition\label{bimodule}
Let $\cG_0,\cG_{-1}$ denote a pair of associative algebras. We say that $\cG_{-1}$ is a \textbf{$\cG_0$-bimodule} if we have a  left and a right action\footnote{We will often omit  the subscript when there is no ambiguity.} $\cdot_l,\cdot_r$ of $\cG_0$ on $\cG_{-1}$ which commute. 
\begin{eqnarray}
(x'x)\cdot y = x'\cdot (x\cdot y), \qquad (x\cdot y)\cdot x' = x\cdot(y\cdot x'), \quad y\cdot (xx')= (y \cdot x)\cdot x'
     \label{actioncond}
\end{eqnarray}
for all $y\in \cG_{-1}$ and $x',x\in\cG_0$. 
\enddefinition
Equivalently we can demand that the following diagrams are commutative. We note $\mu_{i}$ the multiplication in $\cG_{i}$, $i=-1,0$. 

The associativity of the multiplication  is encoded in the usual diagram. 
\begin{equation}\label{diag:associativity}
\begin{tikzcd}
&  \cG_{i}\otimes \cG_{i}\otimes \cG_{i}  
\arrow[dl, "\id\otimes \mu_{i}"]  \arrow[dr, "\mu_{i}\otimes \id"] \\
\cG_i\otimes \cG_{i} \arrow[dr, "\mu_{i}"] & &\cG_{i}\otimes \cG_{i} \arrow[dl, "\mu_{i}"]\\
&
\cG_{i}
\end{tikzcd}.
\end{equation}
The bimodularity conditions \eqref{actioncond} read as commutative diagrams,
\begin{eqnarray}\label{diag:action0}
\begin{tikzcd}
&  \cG_{0}\otimes \cG_{0}\otimes \cG_{-1}  
\arrow[dl, "\mu_{0}\otimes \id"]  \arrow[dr,  "\id \otimes \cdot_l"] \\
\cG_{0}\otimes \cG_{-1} \arrow[dr, "\cdot_l"] & &\cG_{0}\otimes \cG_{-1} \arrow[dl, "\cdot_l"]\\
&
\cG_{-1} &
\end{tikzcd}\nonumber\\
\begin{tikzcd}
&  \cG_{-1}\otimes \cG_{0}\otimes \cG_{0}  
\arrow[dl, "\id \otimes \mu_{0}"]  \arrow[dr,  " \cdot_r\otimes\id"] \\
\cG_{-1}\otimes \cG_{0} \arrow[dr, "\cdot_r"] & &\cG_{-1}\otimes \cG_{0} \arrow[dl, "\cdot_r"]\\
&
\cG_{-1}
\end{tikzcd}
\end{eqnarray}
and the bimodule condition, the middle one of \eqref{actioncond} is 
\begin{eqnarray}\label{diag:bimo}
\begin{tikzcd}
&  \cG_{0}\otimes \cG_{-1}\otimes \cG_{0}  
\arrow[dl, "\cdot_l\otimes \id"]  \arrow[dr,  "\id \otimes \cdot_r"] \\
\cG_{-1}\otimes \cG_{0} \arrow[dr, "\cdot_r"] & &\cG_{0}\otimes \cG_{-1} \arrow[dl, "\cdot_l"]\\
&
\cG_{-1}
\end{tikzcd}
\end{eqnarray}

If we introduce a homomorphism $t$ between $\cG_{-1}$ and $\cG_{0}$, subject to some conditions, then    $\cG_{-1}$ and $\cG_0$  can be used to define a crossed module of algebras.

\definition 
An \textbf{associative 2-algebra} $\cG$ is an algebra object in $\mathsf{2Vect}^{BC}$.
\enddefinition
Through the characterization \textbf{Proposition \ref{2ch}}, we equivalently have the following.
\begin{definition}\label{assoc2alg}
An {\bf associative 2-algebra} is a crossed-module of (finite dimensional) associative algebras. More precisely, $\cG$ is the data of a pair of associative algebras $\cG_0, \,\cG_{-1}$ and an algebra homomorphism $t:\cG_{-1}\rightarrow\cG_0$ satisfying the following:
\begin{enumerate}
    \item $\cG_{-1}$ is a $\cG_0$-bimodule,
    \item $t$ is \textbf{two-sided $\cG_0$-equivariant},
    \begin{equation}
        t(x\cdot y) = x t(y),\qquad t(y\cdot x)=t(y)x \label{algpeif1}
    \end{equation}
    for all $y\in\cG_{-1},x\in\cG_0$, and
    \item the \textbf{Peiffer identity} is satisfied,
    \begin{equation}
        t(y)\cdot y' = yy' = y\cdot t(y'),\label{algpeif2}
    \end{equation}
    where $y,y'\in\cG_{-1}$.
\end{enumerate}
We call the latter two the {\it Peiffer conditions}. We denote an associative 2-algebra simply by $\cG$, or by $(\cG,\cdot)$ to emphasize the bimodule structure. Let $k$ denote the ground ring of the 2-vector space underlying $\cG$. We call $\cG$ {\bf unital} if there exists a unit map $\eta=(\eta_{-1},\eta_0):k \rightarrow \cG$ such that
\begin{equation}
  \eta_{-1} y = y \eta_{-1} = y,  \quad \eta_0  x = x \eta_0 = x, \label{def:unity}
\end{equation}
for all $y\in\cG_{-1},x\in\cG_0$. Moreover, $t$ should respect the units, ie. $t (\eta_{-1}) = \eta_0$. 
\end{definition}
Note that one may consider $\cG_{-1}$ first as a vector space and
    {\it define} its product with the Peiffer identity. This notion is how one may show the bijective correspondence between Lie algebra crossed-modules and 2-term $L_\infty$-algebras \cite{Bai_2013,Chen:2012gz}. However, in the skeletal case, since the Peiffer identity is empty, which forces the product on $\cG_{-1}$ to be trivial. 

\begin{remark}\label{bimodulecondition}
   If $t\neq 0$ were non-trivial then the Peiffer conditions, together with bimodularity, imply that
    \begin{equation}
        x\cdot(yy') = (x\cdot y)y',\qquad y(x\cdot y') = (y\cdot x)y',\qquad (yy')\cdot x=y(y'\cdot x) \nonumber
    \end{equation}
    for each $x\in \cG_0,y,y'\in\cG_{-1}$. This puts strong constraints on the algebra action $\cdot$, which is not necessarily imposed in the skeletal $t=0$ case. 
\end{remark}

\smallskip

Equivalently, we can encode the different conditions defining the 2-algebra in terms of commutative diagrams. The equivariance reads 
\begin{eqnarray}\label{diag:equivariance}
\begin{tikzcd}
&  \cG_0\otimes \cG_{-1}  
\arrow[dl, "\cdot_l"]  \arrow[dr, "\id \otimes t"] \\
 \cG_{-1} \arrow[dr, "t"] & &\cG_{0}\otimes \cG_{0} \arrow[dl, "\mu_{0}"]\\
&
\cG_{0}
\end{tikzcd}, \quad
\begin{tikzcd}
&  \cG_{-1}\otimes \cG_{0}  
\arrow[dl, "\cdot_r"]  \arrow[dr,  "t\otimes \id"] \\
\cG_{-1}\arrow[dr, "t"] & &\cG_{0}\otimes \cG_{0} \arrow[dl, "\mu_{0}"]\\
&
\cG_{0}
\end{tikzcd}
\end{eqnarray}
and the Peiffer identity is 
\begin{eqnarray}\label{diag:Pid}
\begin{tikzcd}
&  \cG_{-1}\otimes \cG_{-1}  
\arrow[dl, "t\otimes \id"]  \arrow[dd, "\mu_{-1}"] \arrow[dr, "\id\otimes t"] \\
\cG_0\otimes \cG_{-1} \arrow[dr, "\cdot_l"] & &\cG_{-1}\otimes \cG_{0} \arrow[dl, "\cdot_r"]\\
&
\cG_{-1}
\end{tikzcd}.
\end{eqnarray}
Finally the unit map is encoded in the following commutative diagrams.
\begin{eqnarray}
\centering
&
\begin{tikzcd}
&  \cG_{i}\otimes \cG_{i}  
 \arrow[d, "\mu_{i}"] \\
k\otimes \cG_i\arrow[ur, "\eta_i\otimes \id"]   \arrow[r, "\cong"'] &  \cG_i   & \arrow[l, "\cong"]  \arrow[ul, "\id\otimes \eta_i"'] \cG_{i}\otimes k
\end{tikzcd} &  \label{diag:unit}\\
\begin{tikzcd}
&  \cG_{0}\otimes \cG_{-1}  \arrow[d, "\cdot_r"]
  \\
k\otimes \cG_{-1}  \arrow[ur, "\eta_0\otimes \id"]   \arrow[r, "\cong"] &  \cG_{-1}  
\end{tikzcd}  &
\begin{tikzcd}
&   \cG_{-1} \otimes \cG_{0}  \arrow[d, "\cdot_l"]
  \\
 \cG_{-1}\otimes k   \arrow[ur, "\id \otimes \eta_0"]   \arrow[r, "\cong"] &  \cG_{-1}  
\end{tikzcd}
&
\begin{tikzcd}
&   \cG_{-1}  \arrow[d, "t"]
  \\
  k \arrow[ur, "\eta_{-1}"] \arrow[r, "\eta_0"]    &  \cG_{0}  
\end{tikzcd}\nonumber
\end{eqnarray}

\subsubsection{Classification of 2-groups and associative 2-algebras}
Recall a 2-group is a connected 2-groupoid $[G_{-1}],G_0,\text{pt}]$ \cite{Baez:2012,Baez:2008hz}, or equivalently its loop 1-groupoid $G_{-1}\rtimes G_0 \rightrightarrows G_0$ \cite{Douglas:2018}. These are equivalent to the following crossed-module description \cite{Chen:2012gz,Baez:2008hz}.
\begin{definition}
The crossed-module model of a {\bf 2-group} $G$ is a group homomorphism $t:G_{-1}\rightarrow G_0$ together with an action $\rhd$ of $G_0$ on $G_{-1}$ such that the following conditions
\begin{equation}
    t(x\rhd y) = xt(y)x^{-1},\qquad (ty)\rhd y'=yy'y^{-1}\label{pfeif2}
\end{equation}
are satisfied for each $x\in G_0$ and $y,y'\in G_{-1}$. The first and second conditions are known respectively as the {\bf equivariance} and the {\bf Peiffer identity}.
\end{definition}
\noindent 
A {\bf 2-group homomorphism}, is a graded map $\phi=(\phi_{-1},\phi_0):G\rightarrow G'$ such that
\begin{enumerate}
    \item $\phi_0:G_0\rightarrow G_0'$ and $\phi_{-1}:G_{-1}\rightarrow G_{-1}'$ are group homomorphisms,
    \item $\phi_{-1}(x\rhd y) = (\phi_0x)\rhd'(\phi_{-1}y)$ for each $x\in G_0,y\in G_{-1}$, and
    \item $\phi_0t = t'\phi_{-1}$.
\end{enumerate}
We say that two (strict) 2-groups $G,G'$ are elementary equivalent, or quasi-isomorphic, if there exists an invertible 2-group homomorphism between them. The fundamental classification result \cite{Zhu:2019,Kapustin:2013uxa} is that
\begin{theorem}
{\bf (Gerstenhaber, attr. Mac-Lane).} 2-groups are classified up to quasi-isomorphism by a degree-3 {\bf group cohomology class} $\tau\in H^3(N,V)$, where  $N=\operatorname{coker}t,V=\operatorname{ker}t$.
\end{theorem}
\noindent $\tau$ is also called the {\bf Postnikov class} in the literature \cite{Kapustin:2013uxa,Ang2018,chen:2022}. Note $V=\operatorname{ker}t$ must be Abelian due to the Peiffer identities. The tuple $(N,V,\tau)$ is known as {\it Ho{\'a}ng data} \cite{Ang2018,Nguyen2014CROSSEDMA}, which was proven by Ho{\'a}ng to classify "$\operatorname{Gr}$-categories". 
 
\smallskip

Similar to the case of 2-groups \cite{Baez:2003fs,Wagemann+2021}, a 2-algebra homomorphism $f=(f_{-1},f_0):\cG\rightarrow \cG'$ is a graded pair of algebra homomorphisms that respect the underlying bimodule structure; namely,
\begin{enumerate}
    \item $f_0:\cG_0\rightarrow \cG_0'$ and $f_{_1}:\cG_{-1}\rightarrow \cG_{-1}'$ are algebra homomorphisms,
    \item $f_{-1}(x\cdot y) = (f_0x)\cdot'(f_{-1}y)$ and $f_{-1}(y\cdot x) = (f_{-1}y)\cdot' (f_0x)$ for each $x\in\cG_0,y\in\cG_{-1}$, and
    \item $f_0t=t'f_{-1}$.
\end{enumerate}
We say that two 2-algebras are elementary equivalent, or quasi-isomorphic, if there exists an invertible 2-algebra homomorphism between them.
\begin{theorem}\label{2algclass}
{\bf (Gerstenhaber, attr. Wagemann \cite{Wagemann+2021}).} Associative 2-algebras are classified up to quasi-isomorphism by a degree-3 {\bf Hochschild cohomology} class $\cT\in HH^3(\cN,V)$, where $\cN = \operatorname{coker}t$ and $V=\operatorname{ker}t$. 
\end{theorem}
\noindent See \cite{Wagemann+2021} for a definition of Hochschild cohomology of an algebra. The Peiffer identity implies that $V\subset Z(\cG_{-1})$ is in the {\it nucleus} of $\cG_{-1}$; it is in fact a square-free ideal \cite{Wagemann+2021}. Note the nucleus is \textit{not} the same as the centre, which have commutative (but non-trivial) multiplication.

\subsubsection{Example: group 2-algebras from 2-groups}\label{sec:2gpalgb}
Let $k$ denote a field of characteristic zero. One example of 2-algebras comes from using a 2-group $G$. One way to construct a 2-algebra from $G$ is to take the group algebra functor $kG$ \cite{Pfeiffer2007}, and extend the $t$-map linearly such that $t: kG_{-1}\rightarrow kG_0$ is an algebra map. To form an associative 2-algebra, we need a $kG_0$-bimodule structure on $kG_{-1}$, which can be induced from the group action
\begin{equation}
    x\cdot y = x\rhd y,\qquad y\cdot x = x^{-1}\rhd y,\nonumber
\end{equation}
where $x\in  G_0,y\in G_{-1}$. 

However, the subtlety here is that $G_0$ acts by group automorphism, not algebra automorphism,
\begin{equation}
    x\cdot (yy') = (x\cdot y)(x\cdot y') \neq (x\cdot y)y',\label{grpactiontrouble} 
\end{equation}
which contradicts the condition listed in {\it Remark \ref{bimodulecondition}} if $t\neq 0$ or the group action $\rhd\neq 0$ were non-trivial. 

We are going to present three alternative ways to circumvent this issue, which will rely on different properties of the  2-group $G$. 
We will therefore obtain three different resulting 2-algebras $kG$ associated to $G$ (depending on the properties of $G$).   

\medskip

\paragraph{\textit{a) Wagemann's quotienting construction.} } The first  way to  guarantee the functoriality of the map $G\mapsto kG$ follows \cite{Wagemann+2021}. Essentially, this amounts to quotienting out certain terms such that we can recover the algebra automorphism. 

We begin with the group algebra $k(G_{-1}\times G_0)$ equipped with the {\bf balanced} algebra structure
\begin{equation}
    (y,x)\cdot (y',x') = (yy'+y\cdot x' + x\cdot y',xx'),\qquad x,x'\in kG_0,~y,y'\in kG_{-1},\nonumber
\end{equation}
and let $\pi_0$ denote the projection onto $kG_0$. Embed $kG_{-1}\hookrightarrow k(G_{-1}\rtimes G_0)$ by $y\mapsto (y,0)$ as a subalgebra, and define the ideal
\begin{equation}
    X = kG_{-1}\cdot D + D\cdot kG_{-1},\qquad \textrm{with }\, D \subset \{(y,x)\in k(G_{-1}\rtimes G_0)\mid t(y)=-x\}
    \nonumber
\end{equation}
where $D$ is the maximal subalgebra within the kernel subspace.

Explicitly, elements in the subspace $D\subset k(G_{-1}\rtimes kG_0)$ satisfy $t(y)=-x$, and hence consist of pairs $(y,x) = (y,-t(y))$ parameterized by $y\in kG_{-1}$. The ideal $X$ is then a direct sum of elements of the forms
\begin{gather}
    (y',0)\cdot(y,x) = (y'y-t(y^{-1})\rhd y',0) = (y'y-y^{-1}y'y,0),\nonumber\\
    (y,x)\cdot (y',0) = (yy'-t(y)\rhd y',0) = (yy'-yy'y^{-1},0)\nonumber
\end{gather}
using the 2-group properties, where $y'\in kG_{-1},x\in kG_0$ are arbitrary. 

We now form the algebra quotient $k(G_{-1}\rtimes G_0)/X$, and denote by $\tilde\pi_0:k(G_{-1}\rtimes G_0)/X\rightarrow kG_0$ the induced projection map. Explicitly, this quotient consist of pairs $(y,y')\in k(G_{-1})\otimes k(G_{-1})$ such that
\begin{equation}
    (y'-y^{-1}y')y = 0,\qquad y(y' - y'y^{-1}) =0,\qquad y,y'\in kG_{-1}.\label{cat1algcondition}
\end{equation}
Put $\tilde t = t|_{\widetilde{kG_{-1}}}$ as the restriction of the $t$-map onto the kernel $\widetilde{kG_{-1}}\equiv \operatorname{ker}\tilde{\pi}_0$ --- in other words, $\widetilde{kG_{-1}}$ consist of the $G_{-1}$-invariant elements of $kG_{-1}$ under both left- and right-multiplication. The induced $t$-map then induces the following 2-vector space 
\begin{equation}
    kG := \widetilde{kG_{-1}} \xrightarrow{\tilde t} kG_0\label{2grpalg}
\end{equation}
which has an associative 2-algebra structure. For a proof that this construction indeed yields an associative 2-algebra, see {Theorem 3.8.3} of \cite{Wagemann+2021}.

\medskip 

\paragraph{\textit{b) 2-group with adjoint action.}}
The condition  \eqref{cat1algcondition} is a very stringent requirement, especially in the $t=\id$ case. As an alternative, we can follow a different route when the 2-group action is given by the adjoint action. In this case, to  guarantee the functoriality of the map $G\mapsto kG$, one needs to   find bilinear maps $\cdot:kG_{-1}\rtimes kG_0\oplus kG_0\otimes kG_{-1}\rightarrow kG_{-1}$ satisfying
\begin{equation}
    x\rhd y = x\cdot y\cdot x^{-1},\qquad x\in G_0,~y\in G_{-1}.\label{2grpalg2}
\end{equation}
Provided $\cdot$ extends linearly to left- and right-regular algebra representations, one may use it directly to define the $kG_0$-bimodule structure of $kG_{-1}$. The equivariance condition and Peiffer identity 
\begin{gather}
    x t(y\cdot x^{-1}) = t(x\cdot y)x^{-1} = t(x\cdot y\cdot x^{-1})= t(x\rhd y) = xt(y)x^{-1},\nonumber\\
    yy'y^{-1} = t(y)\rhd y' = t(y)\cdot y'\cdot t(y^{-1}).\nonumber
\end{gather}
are directly verified. This construction is closely related to the invertibles functor $\cG\mapsto G=\cG^\times$, which is left-adjoint to the group algebra functor \cite{Wagemann+2021}.

\medskip 

If $G=(G_{-1}\xrightarrow{t=\id}G_0)$ is a \textit{trivial} 2-group, the $t$-map is the identity and there is no group kernel nor cokernel. Hence it is in the trivial class under elementary equivalence \cite{Wag:2006,Baez:2004in} of 2-groups --- this is why such 2-groups are called trivial. By definition, the group action $\rhd$ must be the conjugation action. Hence the second approach to associate a 2-algebra to $G$ seems well suited to this case.

The identity $t$-map extends $\operatorname{id}:kG_{-1}\rightarrow kG_0$ directly to the group algebras, whence we obtain an associative 2-algebra $kG$. Following  \eqref{2grpalg2}, the $kG_0$-bimodule structure of $kG_{-1}$ is given by group multiplication, and {\it not} conjugation. Similarly to the 2-group, there is no algebra kernel nor cokernel, whence $kG$ is in the trivial class under elementary equivalence of 2-algebras \cite{Wagemann+2021}. As such, we shall also call such 2-algebras $A\xrightarrow{\operatorname{id}}A$ trivial. 

This construction defines a functor $G_0\mapsto G$ assigning the group $G_0$ (resp. the algebra $A_0$) to the associated trivial 2-group $G$ (resp. the trivial associative 2-algebra $A$), which commutes with the group algebra functor \cite{Wagemann+2021}. We will show in Section \ref{2groupbialg} that  this functor, which defines an embedding of the category of associative algebras into that of associative 2-algebras, can be extended to the bialgebra/2-bialgebra context.

\medskip


\paragraph{\textit{c) The skeletal case.}}
The skeletal case is peculiar enough, that it deserves its own consideration. Recall for a \textbf{skeletal} 2-group,  the $t$-map $t=1:G_{-1}\rightarrow G_0$ is the constant map onto to identity $1\in G_0$. By the (2-group) Peiffer identity, $G_{-1}$ must be Abelian, and the group kernel $\operatorname{ker}t = G_{-1}$ is the entire group $G_{-1}$. If we take the group algebras and simply extend $t$ linearly, we obtain the augmentation $t:kG_{-1}\rightarrow k\cdot 1=k$, whose algebra kernel $\operatorname{ker}t = I$ is the augmentation ideal, which is in general distinct from the group algebra $kG_{-1}$. 

This case is in drastic contrast with the trivial $t=\id$ case. The Peiffer identity dictates that $kG_{-1}$ is {commutative}, but \textit{not} nuclear, and hence $kG$ would be non-skeletal. In other words, merely taking the group algebras of the graded components of a 2-group does not preserve skeletality, and hence does not preserve its elementary equivalence class. Since the linearized $t$-map is non-trivial $t\neq 0$, there are then two ways to see $kG$ as a 2-algebra:
\begin{enumerate}
    \item By following the quotient prescription  \eqref{2grpalg}, we must kill the entire degree-(-1) structure, yielding $\tilde t: 0\rightarrow kG_0$. Indeed, the only totally invariant subgroups of an Abelian group under group multiplication by itself is trivial, and $D \cong k(\operatorname{ker}t) = kG_{-1}$ is the entire group algebra. This is the {\it only} group 2-algebra one can construct associated to a skeletal 2-group $G$ following the prescription by Wagemann \cite{Wagemann+2021}.
    \item Alternatively, we may seek to solve  \eqref{2grpalg2} with the given group action $\rhd$. However, solutions do not exist unless $x\rhd\in\operatorname{Aut}(G_{-1})$ is \textit{inner} for all $x\in G_0$, which is certainly not always the case. This gives us a broader class of skeletal 2-groups with which we may form the 2-group algebra by linearly extending $t$.
\end{enumerate}

Yet another alternative construction (which we emphasize is specialized to the skeletal case) is to linearly extend the group action $\rhd$ but {\it not} the $t$-map. Instead, we take the skeletal 2-algebra $t = 0: kG_{-1}\rightarrow kG_0$, whose trivial $t$-map allows us to circumvent the issues posed by  \eqref{grpactiontrouble}. This construction $G\mapsto kG$ not only holds for any skeletal 2-group $G$, but for certain $G_{-1},G_0$ it also preserves the classification! This is due to a deep result in \cite{Siegel:1999}, which computes an explicit isomorphism between the Hochshild cohomology and group cohomology in certain cases. We will adopt this approach in an accompanying paper \cite{Chen2z:2023}, which applies the general framework of Hopf 2-algebras that we shall develop here to construct excitations in the 4d Kitaev model.

\medskip

With either  \eqref{2grpalg},   \eqref{2grpalg2}, or the skeletal case proposal understood, we shall neglect the tilde on the $t$-map and simply denote $kG= kG_{-1}\xrightarrow{t}kG_0$ in the following. Given a 2-group $G$, while we can construct in effect different 2-algebras $kG$ with the above two proposals, at the end of the day, we will obtain an associative 2-algebra, which will be the starting point of the 2-bialgebra definition.

\subsection{Associative 2-bialgebras}\label{strict2bialg}

\subsubsection{Definition}

\paragraph{Coassociative 2-coalgebra.}
We seek a dual notion of an associative 2-algebra {\bf Definition \ref{assoc2alg}}. The idea will be to reverse the arrows in the diagrams and swap the degree. Indeed, our duality structure will typically swap degrees. This is a consequence of how "dualization" is defined in homological algebra \cite{Bai_2013,chen:2022,Chen:2012gz,Chen:2013}.

\smallskip

Let us consider a pair of vector spaces, $\cG_{0}, \cG_{-1}$ with the map $t:\cG_{-1}\rightarrow \cG_0$.
In direct analogy with the 
2-cocycle $\delta=\delta_{-1}+\delta_0$ that were introduced to define a classical Lie 2-bialgebra \cite{Bai_2013,chen:2022}, we introduce the {\it coproduct} maps
\begin{equation}
    \Delta_{-1}:\cG_{-1}\rightarrow \cG_{-1}\otimes \cG_{-1},\qquad \Delta_0:\cG_0\rightarrow (\cG_{-1}\otimes \cG_0)\oplus(\cG_0\otimes\cG_{-1}). \label{grpdcoprod} 
\end{equation}
Note that $\Delta_0$ comes in two graded components  $\Delta_0= \Delta_0^l + \Delta_0^r$ with 
\begin{equation}
    \Delta_0^l:\cG_0\rightarrow\cG_{-1}\otimes \cG_0,\qquad \Delta_0^r:\cG_0\rightarrow\cG_0\otimes\cG_{-1}.\nonumber  
\end{equation}
In the following, we shall use extensively the conventional Sweedler notation 
\begin{equation}
    \Delta(y,x) \equiv \Delta_{-1}(y) + \Delta_0(x) = y_{(1)}\otimes y_{(2)} + (x_{(1)}^l\otimes x_{(2)}^l+x_{(1)}^r\otimes x_{(2)}^r)\label{sweed}
\end{equation}
where $x^l_{(2)},x^r_{(1)}\in \cG_{0}$ and $y_{(1)}, y_{(2)}, x_{(1)}^l,x^r_{(2)}\in \cG_{-1}$. 

\smallskip

Now let
\begin{equation}
    \Delta_0':\cG_0\rightarrow \cG_0\otimes \cG_0  \nonumber
\end{equation}
denote a coproduct in degree-0, such that $\Delta_{-1},\Delta_0'$ are subject to the following {\bf coassociativity} conditions
\begin{eqnarray}
    (\id \otimes \Delta_{-1})\circ \Delta_{-1} = (\Delta_{-1}\otimes \id)\circ\Delta_{-1},&\qquad& (\id \otimes \Delta_0')\circ \Delta_0'= (\Delta_0'\otimes \id)\circ\Delta_0',\label{cohgrpd00}
\end{eqnarray}
which can be obtained by reversing the arrows in \eqref{diag:associativity} for the products $\mu_{-1}$ and $\mu_{0}$. Hence  $(\cG_{-1},\Delta_{-1})$ and $(\cG_0,\Delta_0')$ are {\bf coassociative coalgebras} if  \eqref{cohgrpd00} is satisfied \cite{Majid:1996kd}. In the following, we shall use the Sweedler notation
\begin{equation}
    \Delta_0'(x) = \bar x_{(1)}\otimes \bar x_{(2)}\in\cG_0\otimes \cG_0.\label{deg0sweed}
\end{equation}

\begin{definition}
Let $(\cG_{-1},\Delta_{-1})$ and $(\cG_0,\Delta_0')$ denote a pair of coassociative coalgebras with the coactions $\Delta_0^l$ and $\Delta_0^r$. We say that $\cG_0$ forms a {\bf $\cG_{-1}$-cobimodule} if the following {\bf cobimodularity conditions}
\begin{eqnarray}
    (\Delta_{-1}\otimes \id)\circ \Delta_0^l &=& (\id \otimes\Delta_0^l)\circ\Delta_0^l,\nonumber\\
    (\id \otimes\Delta_{-1})\circ\Delta_0^r&=&(\Delta_0^r\otimes \id)\circ\Delta_0^r,\nonumber\\
    (\id\otimes \Delta_0^r)\circ\Delta_0^l &=& (\Delta_0^l\otimes\id)\circ\Delta_0^r\label{cohgrpd10}
\end{eqnarray}
are satisfied. In terms of commutative diagrams, we have
\begin{eqnarray}\nonumber
\begin{tikzcd}
&  \cG_{0}  
\arrow[dl, "\Delta_0^l"]  \arrow[dr, "\Delta_0^l"] \\
 \cG_{-1} \otimes \cG_0 \arrow[dr, "\id\otimes \Delta_0^l"] & &\cG_{-1} \otimes \cG_{0} \arrow[dl, "\Delta_{-1}\otimes \id"]\\
&\cG_{-1}\otimes \cG_{-1}\otimes \cG_{0} 
\end{tikzcd}, \nonumber\\
\begin{tikzcd}
&  \cG_{0}  
\arrow[dl, "\Delta_0^r"]  \arrow[dr, "\Delta_{0}^r"] \\
 \cG_{0} \otimes \cG_{-1} \arrow[dr, "\Delta_0^r\otimes\id"] & &\cG_{0} \otimes \cG_{-1} \arrow[dl, "\id\otimes \Delta_{-1}"]\\
&\cG_{0}\otimes \cG_{-1}\otimes \cG_{-1} 
\end{tikzcd}\nonumber \\
\begin{tikzcd}
&  \cG_{0}  
\arrow[dl, "\Delta_0^r"]  \arrow[dr, "\Delta^l_{0}"] \\
 \cG_{0} \otimes \cG_{-1} \arrow[dr, "\Delta_0^l\otimes\id"'] & &\cG_{-1} \otimes \cG_{0} \arrow[dl, "\id\otimes\Delta_0^r "]\\
&\cG_{-1}\otimes \cG_0\otimes \cG_{-1} 
\end{tikzcd}
\end{eqnarray}
\end{definition}

We emphasize again that, upon dualizing the commutative diagrams \eqref{diag:associativity}-\eqref{diag:unit}, we must also swap the grading: the action $\cdot_l: \cG_{0}\otimes \cG_{-1}\rightarrow\cG_{-1}$ is dualized to the coaction component $\Delta_0^l:\cG_{0} \rightarrow \cG_{-1}\otimes \cG_{0}$ of the coproduct.

\begin{definition}\label{coassoc2coalg}
A {\bf coassociative 2-coalgebra} $(\cG,\Delta)$ is a coalgebra homomorphism $t:\cG_{-1}\rightarrow\cG_0$ such that
\begin{enumerate}
    \item $\cG_0$ is a $\cG_{-1}$-cobimodule,
    \item $t$ is \textbf{coequivariant} 
    \begin{equation}
        D_t^+\circ\Delta_{-1}= \Delta_0\circ t, \label{cohgrpd+} 
    \end{equation}
    where we have introduced a convenient tensor notation for the induced $t$-map
    \begin{equation}
        D_t^\pm := t\otimes 1 \pm 1\otimes t \nonumber
    \end{equation}
in terms of the graded sum.  This condition is encoded by the following commutative diagram.
\begin{eqnarray}\label{diag:coequiv}
\begin{tikzcd}
&  \cG_{-1}  
\arrow[dl, "t"]  \arrow[dr, "\Delta_{-1}"] \\
 \cG_0 \arrow[dr, "\Delta_0"] & &\cG_{-1}\otimes \cG_{-1} \arrow[dl, "D_t^+"]\\
&
\cG_{0}\otimes \cG_{-1}\oplus \cG_{-1}\otimes \cG_0
\end{tikzcd}
\end{eqnarray}   
    \item the {\bf coPeiffer identity}
    \begin{equation}
        (t\otimes\id)\circ\Delta_0^l = \Delta_0' = (\id\otimes t)\circ\Delta_0^r, \label{cohgrpd-} 
    \end{equation}
     which in particular means that we must necessarily have
    \begin{equation}
        D_t^-\Delta_0 =(t\otimes\id)\circ\Delta_0^l - (\id\otimes t)\circ\Delta_0^r  =0.\nonumber
    \end{equation}    
It is encoded in the following commutative diagram    
    \begin{eqnarray}\label{diag:coPid1}
\begin{tikzcd}
&  \cG_{0}\otimes \cG_{0}  
  \\
\cG_0\otimes \cG_{-1}\arrow[ur, "\id\otimes t"]   & &\cG_{-1}\otimes \cG_{0} \arrow[ul, "t\otimes \id"]\\
&
\cG_{0}\arrow[ul, "\Delta_0^r"] \arrow[uu,"\Delta_0'"]  \arrow[ur, "\Delta_0^l"]
\end{tikzcd}.
\end{eqnarray}

\end{enumerate}
We call $(\cG,\Delta)$ {\bf counital} if there is a {counit} map $\epsilon=(\epsilon_{-1},\epsilon_0):\cG\rightarrow k $ such that 
\begin{eqnarray}
    \operatorname{id}=(\operatorname{id}\otimes \epsilon_{-1})\circ\Delta_{-1},&\quad& \operatorname{id} = (\epsilon_{-1}\otimes \operatorname{id})\circ\Delta_{-1},\nonumber\\
    \operatorname{id}=(\epsilon_{-1}\otimes \id)\circ\Delta_0^l, &\quad& \id =(\id\otimes \epsilon_{-1})\circ\Delta_0^r.\label{counit}
\end{eqnarray}
Moreover, $\epsilon$ should respect the $t$-map such that $\epsilon_{-1} = \epsilon_0\circ t $. 
\end{definition}

The counit conditions can be seen as reversing the arrows of the diagrams \eqref{diag:unit} (and also a {swap of the grading} since we are dualizing).
\begin{eqnarray}
\centering
& 
\begin{tikzcd}
&  \cG_{-1}\otimes \cG_{-1}  \arrow[dl, "\epsilon_{-1}\otimes \id"'] \arrow[dr, "\id\otimes \epsilon_{-1}"]
 \\
k\otimes \cG_{-1}   \arrow[r, "\cong"] &  \cG_{-1}  \arrow[u, "\Delta_{-1}"']  & \arrow[l, "\cong"]   \cG_{-1}\otimes k
\end{tikzcd} & \label{diag:counit}\\
\begin{tikzcd}
&  \cG_{0}\otimes \cG_{-1}  \arrow[dl, "\id\otimes\epsilon_{-1}"']
  \\
\cG_0 \otimes k    \arrow[r, "\cong"] &  \cG_0 \arrow[u, "\Delta_0^r"'] 
\end{tikzcd} & 
\begin{tikzcd}
&   \cG_{-1} \otimes \cG_{0}  \arrow[dl, "\epsilon_{-1}\otimes\id"']  
  \\
 k\otimes \cG_0     \arrow[r, "\cong"] &  \cG_0  \arrow[u, "\Delta_0^l"'] 
\end{tikzcd} 
 &
\begin{tikzcd}
&   \cG_{-1}  \arrow[d, "t"]\arrow[dl, "\epsilon_{-1}"'] 
  \\
  k    &  \arrow[l, "\epsilon_0"]  \cG_{0}  
\end{tikzcd}\nonumber
\end{eqnarray}

Note again that
 in {\bf Definition \ref{coassoc2coalg}}, the coequivariance and coPeiffer identity are treated as constraints between two coalgebras and the coalgebra homomorphism $t$ between them. With these constraints, we can deduce 
\begin{equation}
    \operatorname{id}=(\epsilon_0\otimes \id)\circ\Delta_0' = (\id\otimes\epsilon_0)\circ\Delta_0'\label{deg0counit}
\end{equation}
from \eqref{deg0sweed} and \eqref{counit}. In the skeletal $t=0$ case, the coproducts $\Delta_{-1},\Delta_0,\Delta_0'$ and the counits $\epsilon_{-1},\epsilon_0$ are independent, and this condition is separate from \eqref{counit}.

\begin{remark}\label{comp-const}
Similar to the 2-algebra case, if $t\neq0$ were not trivial, then we could have the following conditions
 \begin{eqnarray}
     (\id\otimes \Delta_0')\circ \Delta_0^l = (\Delta_0^l\otimes \id)\circ \Delta_0',\nonumber\\
     (\Delta_0'\otimes\id)\circ \Delta_0^r= (\id\otimes \Delta_0^r)\circ \Delta_0',\nonumber\\
     (\id\otimes\Delta_0^l)\circ \Delta_0' = (\Delta_0^r\otimes \id)\circ \Delta_0'\label{cobimodu}
 \end{eqnarray}
between the coproducts $\Delta_0$ and $\Delta_0'$. By making use of the Sweedler notation  \eqref{sweed}, \eqref{deg0sweed}, the coequivariance and the coPeiffer identities translate to
\begin{equation}\label{condcop}
\begin{cases}ty_{(1)} = (ty)_{(1)}^r \\ ty_{(2)} = (ty)_{(2)}^l\end{cases},\qquad \begin{cases} \bar x_{(1)} = t x_{(1)}^l = x_{(1)}^r \\ \bar x_{(2)} = x_{(2)}^l = tx_{(2)}^r\end{cases}.
\end{equation}
When combined, they give $\widebar{ty}_{(1)} = ty_{(1)}, \widebar{ty}_{(2)} = ty_{(2)}$ which will become important later. In the skeletal case, the constraints involving $t$ drop.
\end{remark}

\paragraph{2-bialgebras.}
Using the Sweedler notations  \eqref{sweed}, \eqref{deg0sweed}, we state the condition that the coproduct map $\Delta$ given in  \eqref{grpdcoprod} preserves the algebra/bimodule structure:
\begin{eqnarray}
    \Delta_{-1}(x\cdot y)=  \bar x_{(1)}\cdot y_{(1)}\otimes  \bar x_{(2)}\cdot y_{(2)},&\qquad& \Delta_{-1}(y\cdot x)= y_{(1)}\cdot \bar x_{(1)}\otimes y_{(2)}\cdot \bar x_{(2)},\nonumber\\
    \Delta_0^l(xx')=x_{(1)}^lx_{(1)}'^l\otimes x_{(2)}^lx_{(2)}'^l,&\qquad& \Delta_0^r(xx') = x_{(1)}^rx_{(1)}'^r\otimes x_{(2)}^rx_{(2)}'^r.\label{2algcoprod}
\end{eqnarray}
We call these conditions the {\bf 2-bialgebra axioms}. 

\smallskip

The bialgebra axioms in each degree,
\begin{equation}
    \Delta_{-1}(yy') = y_{(1)}y_{(1)}'\otimes y_{(2)}y_{(2)}',\qquad \Delta_0'(xx') = \bar x_{(1)}\bar x_{(1)}'\otimes \bar x_{(2)}\bar x_{(2)}', \nonumber
\end{equation}
follow directly from  \eqref{2algcoprod} and the coequivariance and coPeiffer identities  \eqref{cohgrpd+}, \eqref{cohgrpd-}; see {\it Remark \ref{comp-const}}.

\begin{definition}\label{assoc2bialg}
The tuple $(\cG,\cdot,\Delta)$ is an {\bf associative 2-bialgebra} iff $(\cG,\cdot)$ is an associative 2-algebra and $(\cG,\Delta)$ is a coassociative 2-coalgebra such that the two structures are mutually compatible; namely $\Delta$ satisfies \eqref{2algcoprod}.

We call $(\cG,\cdot,\eta,\Delta,\epsilon)$ {\bf unital} if $(\cG,\cdot,\eta)$ and $(\cG,\Delta,\epsilon)$ are respectively unital and counital.
\end{definition}

The  2-bialgebra axioms  are equivalently described in terms of the following commutative diagrams, where we use the swap $\sigma: \cG_0\otimes \cG_{-1}\rightarrow \cG_{-1}\otimes \cG_0$ and $\sigma': \cG_{-1}\otimes \cG_{0}\rightarrow \cG_{0}\otimes \cG_{-1}$, 
\begin{eqnarray}
\hspace*{-1cm}
\begin{tikzcd}
&  \cG_{0}\otimes \cG_{-1}  \arrow[dl, "\cdot_l"'] \arrow[dr, "\Delta_0'\otimes \Delta_{-1}"]
 \\
\cG_{-1}   \arrow[ddr, "\Delta_{-1}"] &  & \cG_0\otimes \cG_0\otimes \cG_{-1}\otimes \cG_{-1}  \arrow[d, "\id \otimes \sigma\otimes \id"]  \\
&& \cG_0\otimes \cG_{-1} \otimes \cG_0\otimes \cG_{-1} \arrow[dl, "\cdot_l\otimes \cdot_l"]\\
&\cG_{-1}\otimes \cG_{-1}
\end{tikzcd} \quad
\begin{tikzcd}
&  \cG_{-1} \otimes \cG_0 \arrow[dl, "\cdot_r"'] \arrow[dr, "\Delta_{-1}\otimes \Delta_{0}'"]
 \\
\cG_{-1}   \arrow[ddr, "\Delta_{-1}"] &  & \cG_{-1}\otimes \cG_{-1} \otimes \cG_0\otimes \cG_0   \arrow[d, "\id \otimes \sigma'\otimes \id"]  \\
&& \cG_{-1}\otimes \cG_{0} \otimes \cG_{-1}\otimes \cG_{0} \arrow[dl, "\cdot_r\otimes \cdot_r"]\\
&\cG_{-1}\otimes \cG_{-1}
\end{tikzcd}
\nonumber\\
\hspace*{-1cm}
\begin{tikzcd}
&  \cG_{0}\otimes \cG_{0}  \arrow[dl, "\mu_0"'] \arrow[dr, "\Delta_0^l\otimes \Delta_{0}^l"]
 \\
\cG_{0}   \arrow[ddr, "\Delta_0^l"] &  & \cG_{-1}\otimes \cG_0\otimes \cG_{-1}\otimes \cG_0  \arrow[d, "\id \otimes \sigma\otimes \id"']  \\
&& \cG_{-1}\otimes \cG_{-1} \otimes \cG_0\otimes \cG_0 \arrow[dl, "\mu_{-1}\otimes \mu_0"]\\
&\cG_{-1}\otimes \cG_0
\end{tikzcd}\quad
\begin{tikzcd}
&  \cG_{0}\otimes \cG_{0}  \arrow[dl, "\mu_0"'] \arrow[dr, "\Delta_0^r\otimes \Delta_{0}^r"]
 \\
\cG_{0}   \arrow[ddr, "\Delta_0^r"] &  & \cG_{0}\otimes \cG_{-1}\otimes \cG_0\otimes \cG_{-1}  \arrow[d, "\id \otimes \sigma'\otimes \id"]  \\
&&  \cG_0\otimes \cG_0 \otimes \cG_{-1}\otimes \cG_{-1}  \arrow[dl, "\mu_0\otimes \mu_{-1}"]\\
&\cG_{0}\otimes \cG_{-1}
\end{tikzcd}\nonumber
\end{eqnarray}

\subsubsection{Example: function 2-bialgebras on 2-groups}\label{2groupbialg}
Let $k$ denote a field of characteristic zero, and let $G$ denote a (finite) 2-group. 

\paragraph{Function 2-coalgebra on a (finite) 2-group.}
Consider the 2-group 2-algebra $kG$ constructed in section \ref{sec:2gpalgb}. We denote its linear dual by $k[G] = \operatorname{Hom}_k(kG,k)$, which consist of $k$-linear functions on $kG$. This space inherits the {\it dual} grading of the 2-group 2-algebra $kG$, in the sense that \textit{$k[G_{-1}]$ is in degree-0 while $k[G_0]$ is in degree-(-1)}. Each element $F  \in k[G]$, then admits a decomposition in terms of this grading, $F = \xi \oplus \zeta$, with $\xi\in k[G_0]$, $\zeta\in k[G_{-1}]$, for which the $t$-map is given by the pullback, 
\begin{equation}
    k[G] = k[G_0]\xrightarrow{t^*} k[G_{-1}],\qquad (t^*\xi)(y) = \xi(ty),\nonumber
\end{equation}
where $y\in kG_{-1}$. Note this $t$-map is the one on $kG$ defined in  \eqref{2grpalg}, which may differ from that of the underlying 2-group $G$.

\smallskip

The reversal of degree for the function 2-coalgebra allows us to define a grading-odd duality pairing given by the function evaluation
\begin{equation}
    \langle (y,x),(\xi,\zeta)\rangle = \operatorname{ev}((y,x)\otimes (\xi,\zeta)) = \xi(x) + \zeta(y),\nonumber
\end{equation}
with respect to which the coproduct $\Delta$ on $k[G]$ defined in  \eqref{function2alg} is dual to the 2-algebra structure on $kG$. Conversely, if there is a well-defined coassociative coproduct $\Delta$ on $kG$ (ie. satisfying  \eqref{cohgrpd+}-\eqref{cohgrpd-}), then the evaluation pairing dualizes it to a well-defined associative 2-algebra structure on $k[G]$. This is a guiding principle with which we shall construct the 2-quantum double in the following sections.

\smallskip

The natural coproduct $\Delta^*= \Delta_{-1}^*+ \Delta_0^*$ on $k[G]$, is induced by the 2-algebra structure on $kG$,
\begin{gather}
    \Delta_{-1}^*(\xi) =  \xi_{(1)} \otimes \xi_{(2)} \iff \xi(xx') =\xi_{(1)}(x)\xi_{(2)}(x'),\nonumber\\
    (\Delta_0^*)^l(\zeta) = \zeta^l_{(1)}\otimes \zeta^l_{(2)} \iff \zeta(x\cdot y) = \zeta^l_{(1)}(x)\zeta_{(2)}^l(y),\nonumber\\
    (\Delta_0^*)^r\zeta = \zeta^r_{(1)}\otimes\zeta^r_{(2)} \iff \zeta(y\cdot x) = \zeta^r_{(1)}(y)\zeta_{(2)}^r(x),\label{function2alg}
\end{gather}
where $x,x'\in kG_0$ and $y\in kG_{-1}$. The  superscript $^*$ encodes the fact we are considering the dual of $kG$ as we will emphasize soon. 
The induced coproduct $\Delta^*_0{}'$ in degree-0 is induced from the product $\mu_{-1}$,
\begin{equation}
    \Delta^*_0{}'(\zeta) = \bar\zeta_{(1)}\otimes\bar\zeta_{(2)}\iff \zeta(yy') = \bar\zeta_{(1)}(y)\bar\zeta_{(2)}(y'),\nonumber
\end{equation}
and is also related to the actions, thanks to the Peiffer identity, so that $\Delta^*_0{}' = \frac{1}{2}D_{t^*}^+\Delta_0$. In a similar way, all the 2-co-algebra axioms are satisfied as we are dualizing all the properties of the 2-algebra $kG$, hence we are reversing the arrows.

For example, the co-bimodularity is obtained from the bimodularity in $kG$. 
\begin{gather}
    (\Delta^*_0)^l(t^*\xi)(x,y) = \xi(t(x\cdot y)) = \xi(xt(y)) =  \xi_{(1)}(x) \xi_{(2)}(t(y)) = ((1\otimes t^*)\Delta_{-1}^*(\xi))(x,y),\nonumber\\
    (\Delta^*_0)^r(t^*\xi)(y,x) = \xi(t(y\cdot x))= \xi(t(y)x) = \xi_{(1)}(ty) \xi_{(2)}(x) = ((t^*\otimes 1)\Delta_{-1}^*(\xi))(y,x),\nonumber
\end{gather}
The co-Peiffer identity is obtained from  the Peiffer identity
\begin{equation}
    \zeta(yy') = \zeta(ty\cdot y') = \zeta(y\cdot ty')\implies (t^*\otimes 1)(\Delta_0^*)^l(\zeta) = (1\otimes t^*)(\Delta_0^*)^r(\zeta),\nonumber
\end{equation}
and so on and so forth.

\smallskip

Moreover, there is also  a counit $\epsilon:k[G]\rightarrow k$ given by evaluation at the identity, $\epsilon_{-1}(\xi) = \xi(1),\epsilon_0(\zeta) = \zeta(1)$. Since the $t$-map respects the identity, we have $\epsilon_{-1}(t^*\xi) = t^*\xi(1) = \xi(1) = \epsilon_0(\xi)$, as desired for a counit.

\medskip

In order to determine the 2-algebra structure of $k[G]$, we can use the coalgebra structure of $kG$. For this we are going to consider the specific case $t=\id$ as otherwise we cannot make the example as explicit as one would wish, due to the constraints \eqref{condcop}.

\paragraph{The trivial 2-algebra.} Recall a {\it trivial} 2-group $G$ has $G_{-1}=G_0$ and the $t$-map is the identity $t=\operatorname{id}$. This the simplest case one can consider, as its corresponding 2-group 2-algebra $kG$ is obtained by just extending every structure $k$-linearly. We shall now consider endowing a coproduct on $kG$ from that of $kG_0$.

Let $\Delta$ denote a coproduct on $kG$.  \eqref{cohgrpd+}, \eqref{cohgrpd-} dictates that all components $\Delta_{-1},\Delta_0,\Delta_0'$ of the coproduct are identical, and hence $(kG_0,\Delta_0')$ is a coalgebra iff $(kG,\Delta)$ is also one. Indeed, if $\Delta_0'(x) = x_{(1)}\otimes x_{(2)}$ then we must have
\begin{equation}
    \Delta_{-1}(y) = y_{(1)}\otimes y_{(2)},\qquad \Delta_0(x) = y_{(1)}\otimes x_{(2)} + x_{(1)}\otimes y_{(2)},\nonumber
\end{equation}
where $y_{(1)},y_{(2)}$ are identical to $x_{(1)},x_{(2)}\in kG_0$ as elements of the group algebra $kG_0$, but with degree-(-1); in other words, we have $t(y_{(1)}) = x_{(1)}$ under $t=\operatorname{id}$. Explicitly, we can choose the group like coproduct 
\begin{equation}
  \Delta_0'(x) = x\otimes x\quad   \Delta_{-1}(y) = y\otimes y,\quad \Delta_0(x) = y\otimes x + x\otimes y, \quad \textrm{with } ty=y=x\nonumber
\end{equation}

We recover then a 2-algebra structure on the graded function space $k[G]$, for which
\begin{gather}
    \zeta\zeta'(y) = \zeta(y_{(1)})\zeta'(y_{(1)})=\zeta(y)\zeta'(y),\qquad \xi\xi'(x) = \xi(x_{(1)})\xi'(x_{(2)})=\xi(x)\xi'(x),\nonumber\\
    (\zeta\cdot \xi)(x) = \zeta(y_{(1)})\xi(x_{(2)})=\zeta(y)\xi(x),\qquad (\xi\cdot \zeta)(x) = \xi(x_{(1)})\zeta(y_{(2)})=\xi(x)\zeta(y).\nonumber
\end{gather}
We have recovered the pointwise products on $k[G_0]$ and  $k[G_{-1}]$,  the bimodule structure in terms of these pointwise product (since $ty=y=x$).  
Moreover, one can check that we have the 2-bialgebra axioms  \eqref{2algcoprod},
\begin{eqnarray}
    (\zeta\zeta')(x\cdot y) &=& (\zeta\zeta')_{(1)}^l(x)(\zeta\zeta')_{(2)}^l(y) =  \zeta_{(1)}^l(x_{(1)})\zeta'^l_{(1)}(x_{(2)})\zeta_{(2)}^l(y_{(1)})\zeta'^l_{(2)}(y_{(2)})\nonumber\\
    &=&  (\zeta^l_{(1)}\zeta'^l_{(1)})(x)(\zeta^l_{(2)}\zeta'^l_{(2)})(y),\nonumber\\
    (\zeta\zeta')(y\cdot x) &=& (\zeta\zeta')_{(1)}^r(y)(\zeta\zeta')_{(2)}^r(x) =  \zeta_{(1)}^r(y_{(1)})\zeta'^r_{(1)}(y_{(2)})\zeta_{(2)}^r(x_{(1)})\zeta'^r_{(2)}(x_{(2)})\nonumber\\
    &=&  (\zeta^r_{(1)}\zeta'^r_{(1)})(x)(\zeta^r_{(2)}\zeta'^r_{(2)})(y) \nonumber\\
    (\zeta\cdot \xi)(xx')&=& (\zeta\cdot \xi)_{(1)}(x)(\zeta\cdot \xi)_{(2)}(x') = \zeta_{(1)}(y_{(1)})\xi_{(1)}(x_{(2)})\zeta_{(2)}(x'_{(1)}) \xi_{(2)}(x'_{(2)}) \nonumber\\
    &=& (\zeta_{(1)}\xi_{(1)})(x)(\zeta_{(2)}\xi_{(2)})(x');\nonumber
\end{eqnarray}
note $\bar \zeta = \zeta$ as the $t$-map is the identity. Therefore $(k[G],\Delta^*)$ and $(kG,\Delta)$ are (unital) associative 2-bialgebras that are mutually dual in a certain sense, with the former being commutative and the latter being cocommutative, when using the group-like coproducts, just as in the 1-group case \cite{Majid:1994nw}. We shall rigorously establish this general 2-bialgebra duality in {\bf Proposition \ref{2bialg}}.

\begin{remark}
In general, given an ordinary 1-bialgebra $G$, we can form its associated trivial 2-bialgebra $\cG = G\xrightarrow{\operatorname{id}}G$ in analogy with the above. Moreover, it is clear that any bialgebra homomorphism $f:G\rightarrow G'$ extends to a 2-bialgebra homomorphism $f\oplus f: \cG\rightarrow \cG'$, whose graded components are just copies of $f$. We therefore have an embedding
\begin{equation}
    \text{Bialg}_\text{ass}\hookrightarrow \text{2Bialg}_\text{ass}\label{2algembed}
\end{equation}
of the category of associative bialgebras into the category of associative 2-bialgebras. This is a "quantum" version of an analogous result for Lie 2-bialgebras stated in {\it Remark 3.11} of \cite{Bai:2007}.
\end{remark}

\section{Strict 2-quantum doubles and the 2-R-matrix}\label{str2qd}
In this section, we construct our main example of a strict 2-bialgebra given by the strict 2-quantum doubles which can be seen a categorification of the standard quantum double \cite{Majid:1996kd}, and the quantization of a classical 2-double \cite{Chen:2012gz, chen:2022} of Lie 2-algebras. 

The goal for studying (2-)quantum doubles is that, for the ordinary 1-bialgebra $H$, the {\it skew-pairing} involved in the construction of the quantum double $D(H,H)$ of Majid \cite{Majid:1994nw} provides a characterization of R-matrices on $H$. Moreover, this construction reveals that any R-matrix on $H$ can be derived in this way from $D(H,H)$. We wish to directly categorify Majid's construction, and derive a universal characterization of \textit{2-R-matrices} from our construction of a 2-quantum double.

\medskip

Our strategy will be as follows. Firstly, we consider a pair of dual associative 2-bialgebras. They are dual in the sense that the coalgebra sector is given by the algebra sector of its dual counterpart. We then define a notion of a canonical \textit{coadjoint action} of a 2-bialgebra on its dual. By requesting that the mutually-dual 2-bialgebras act on each other by such coadjoint actions, we are then able to form the 2-quantum double as a 2-bialgebra. We will then also prove a key factorization theorem for 2-quantum doubles.

\subsection{Matched pair of 2-(bi)algebras}

\paragraph{Dually paired 2-bialgebras}
Let $(\cG,\cdot,\Delta)$ denote a (finite dimensional) 2-bialgebra, and let $\cG^*$ denote its linear dual, defined with respect to the following duality evaluation/pairing map\footnote{We shall drop the subscripts on the pairing forms $\langle\cdot,\cdot\rangle$ when no confusion arises.}
\begin{equation}
    \langle (g,f),(y,x)\rangle = \langle f,y\rangle_{-1} + \langle g,x\rangle_0\label{pairing}
\end{equation}
for each $x\in \cG_0,y\in \cG_{-1},f\in \cG_{-1}^*,g\in \cG_0^*$. Note that the grading is flipped by dualizing the $t$-map: $\langle t^*\cdot,\cdot\rangle = \langle \cdot,t\cdot\rangle$, whence $t^*:\cG_0^*\rightarrow \cG_{-1}^*$ and $\cG^*$ is skeletal whenever $\cG$ is. In the following, we shall denote this pairing also by an evaluation $\operatorname{ev}$. 

So far, $\cG^*$ merely forms a 2-vector space. By leveraging the duality  \eqref{pairing}, we can induce algebraic structures on $\cG^*$ according to the {\it co}algebraic structures  \eqref{grpdcoprod}, \eqref{deg0sweed} on $\cG$ as follows:
\begin{eqnarray}
    \langle f\otimes f',\Delta_{-1}(y)\rangle= \langle ff',y\rangle,&\qquad&\langle g\otimes g', \Delta_0'(x)\rangle =\langle gg',x\rangle, \nonumber\\
    \langle f\otimes g,\Delta_0^l(x)\rangle = \langle f\cdot^* g,x\rangle,&\qquad& \langle g\otimes f,\Delta_0^r(x)\rangle = \langle g\cdot^* f,x\rangle,\nonumber\\
\langle \Delta^*_{0}{}' f,y \otimes y' \rangle= \langle f,yy'\rangle,&\qquad&
\langle \Delta^*_{-1}g, x\otimes x'\rangle =\langle g,xx'\rangle, \nonumber\\
    \langle \Delta^*_0{}^l f,x\otimes y \rangle = \langle f,x\cdot_l y \rangle,
    &\qquad& \langle \Delta^*_0{}^r f ,y\otimes x \rangle = \langle f,y\cdot_r x\rangle.\nonumber
\end{eqnarray}
The conditions  \eqref{cohgrpd+}, \eqref{cohgrpd-}, \eqref{cohgrpd00}, \eqref{cohgrpd10},  then ensure that $(\cG^*,\cdot^*)$ forms an associative 2-algebra. More is true, in fact, which we now prove in the following.

\begin{proposition}\label{2bialg}
Let $\cG,\cG^*$ be dually paired as in  \eqref{pairing}, then $(\cG,\cdot,\Delta)$ is an (unital) associative 2-bialgebra iff $(\cG^*,\cdot^*,\Delta^*)$ is an (unital) associative 2-bialgebra.
\end{proposition}
\begin{proof}
This is a straightforward computation using the pairing  \eqref{pairing}. In particular, the equivariance and Peiffer identity of $t^*$, as well as the fact that $\cG_0^*$ forms a $\cG_{-1}^*$-bimodule, follow directly from dualizing  \eqref{cohgrpd+}, \eqref{cohgrpd-}, \eqref{cohgrpd00}, \eqref{cohgrpd10}.

What is non-trivial is  \eqref{2algcoprod}. Define $\Delta^*_0$ by dualizing the bimodule structure $\cdot$ of $\cG$, then we have
\begin{eqnarray}
    \langle (\Delta^*_0)^l(ff'),x\otimes y\rangle = \langle f\otimes f',\Delta_{-1}(x\cdot y)\rangle,&\quad&  \langle (\Delta^*_0)^r(ff'),y\otimes x\rangle = \langle f\otimes f',\Delta_{-1}(y\cdot x)\rangle,\nonumber\\
    \langle (\Delta_{-1}^*)(f\cdot^*g),x\otimes x'\rangle = \langle f\otimes g,\Delta_0^l(xx')\rangle,&\quad&\langle (\Delta_{-1}^*)(g\cdot^*f),x\otimes x'\rangle = \langle f\otimes g,\Delta_0^r(xx')\rangle.\nonumber
\end{eqnarray}
We now compute using analogues of  \eqref{2algcoprod} for $\Delta^*$, that
\begin{eqnarray}
    \langle f_{(1)}^lf_{(1)}'^l \otimes f_{(2)}^lf_{(2)}'^l,x\otimes y\rangle &=& \langle (f_{(1)}^l\otimes f_{(1)}'^l) \otimes (f_{(2)}^l\otimes f_{(2)}'^l), (\bar x_{(1)}\otimes \bar x_{(2)}) \otimes (y_{(1)}\otimes y_{(2)})\rangle \nonumber\\
    &=&\langle(\Delta_0^*)^l(f)\otimes (\Delta_0^*)^l(f'),(\bar x_{(1)}\otimes y_{(1)})\otimes (\bar x_{(2)}\otimes y_{(2)})\rangle \nonumber\\
    &=& \langle f\otimes f',(\bar x_{(1)}\cdot y_{(1)})\otimes (\bar x_{(2)}\cdot y_{(2)})\rangle,\nonumber\\
    \langle f_{(1)}^rf_{(1)}'^r \otimes f_{(2)}^rf_{(2)}'^r,y\otimes x\rangle &=& \langle f\otimes f', (y_{(1)}\cdot \bar x_{(1)})\otimes (y_{(2)}\cdot \bar x_{(2)})\rangle,\nonumber
\end{eqnarray}
and similarly
\begin{eqnarray}
   \langle \bar f_{(1)}\cdot^* g_{(1)}\otimes \bar f_{(2)}\cdot g_{(2)}, x\otimes x'\rangle &=&  \langle (\bar f_{(1)} \otimes g_{(1)}) \otimes (\bar f_{(2)}\otimes g_{(2)}),(x_{(1)}^l \otimes x_{(2)}^l)\otimes (x_{(1)}'^l \otimes x_{(2)}'^l)\rangle \nonumber\\
   &=&\langle \Delta_0'^*(f)\otimes \Delta_{-1}^*(g),(x_{(1)}^l\otimes x_{(1)}'^l)\otimes (x_{(2)}^l\otimes x_{(2)}'^l)\rangle,\nonumber\\
   &=&\langle f \otimes g, x_{(1)}^lx_{(1)}'^l \otimes x_{(2)}^lx_{(2)}'^l\rangle ,\nonumber\\
   \langle g_{(1)}\cdot^*\bar f_{(1)}\otimes g_{(2)}\cdot^*\bar f_{(2)}, x\otimes x'\rangle &=& \langle g \otimes f, x_{(1)}^rx_{(1)}'^r \otimes x_{(2)}^rx_{(2)}'^r\rangle,\nonumber 
\end{eqnarray}
hence $\Delta$ also satisfies  \eqref{2algcoprod}. This proves that $(\cG^*,\cdot^*,\Delta^*)$ is an associative 2-bialgebra iff $(\cG,\cdot,\Delta)$ also is. 

Now consider the units and counits. Given
\begin{eqnarray}
    \langle g,\eta x\rangle = \langle (\eta^*\otimes\operatorname{id})\circ\Delta_{-1}^*(g),x\rangle,&\qquad&\langle g,x\eta\rangle = \langle (\operatorname{id}\otimes\eta^*)\circ\Delta_{-1}^*(g),x\rangle,\nonumber\\
    \langle f,\eta\cdot y\rangle = \langle (\eta^*\otimes\operatorname{id})\circ(\Delta_0^*)^l(f),y\rangle,&\qquad& \langle f,y\cdot \eta\rangle = \langle (\operatorname{id}\otimes \eta^*)\circ(\Delta_0^*)^r(f),y\rangle,\nonumber
\end{eqnarray}
we see that $\eta$ is a unit for $(\cG,\cdot)$ (ie. these quantities all vanish) iff $\eta^*$ is a counit for $(\cG^*,\Delta^*)$. Similarly, $\epsilon$ is a counit for $(\cG,\Delta)$ iff $\epsilon^*$ is a unit for $(\cG^*,\cdot^*)$. 
\end{proof}

\paragraph{Coadjoint action.}

\begin{definition} 
 The canonical \textbf{coadjoint action} of $\cG$ on $\cG^*$ is specified in terms of three components, $\bar\rhd=((\rhd_0,\rhd_{-1}),\Upsilon)$ given by 
\begin{eqnarray}
    \rhd_0:\cG_0\rightarrow \operatorname{End}\cG_0^*,&\qquad& \langle g,xx'\rangle = -\langle x\rhd_0 g,x'\rangle,\nonumber\\
    \rhd_{-1}:\cG_0\rightarrow \operatorname{End}\cG_{-1}^*,&\qquad& \langle f,x\cdot y\rangle = -\langle x\rhd_{-1} f,y\rangle,\nonumber\\
    \Upsilon:\cG_{-1}\rightarrow \operatorname{Hom}(\cG_{-1}^*,\cG_0^*),&\qquad& \langle f,y\cdot x\rangle = -\langle \Upsilon_yf, x\rangle.\label{coadj0}
\end{eqnarray}
\end{definition}
\noindent As we will see when discussing 2-representations in Section \ref{2representations}, the coadjoint action can also be interpreted as a weak 2-representation.

Analogously, we have the coadjoint back-action $\bar\lhd=((\lhd_0,\lhd_{-1}),\tilde\Upsilon)$ of $\cG^*$ on $\cG$, which we write from the right\footnote{This means that we have, for instance, $\langle g\cdot^* f,x\rangle = -\langle g,x\lhd_{-1}f\rangle$ and $\langle f\cdot^* g,x\rangle = -\langle f,x\tilde\Upsilon_{g}\rangle$.}. The "bar" notation is used to distinguish $\bar\rhd$ from the group action $\rhd$ in the case where $\cG=kG$ is defined through a 2-group $G$.

\paragraph{Matched pair.}
We now allow a given pair $(\cG, \cG^*)$ of strict 2-bialgebras to act upon each other by coadjoint actions $\bar\rhd$ and $\bar\lhd$. 
In analogy with \cite{Majid:1996kd}, we impose the following monstrous set of twelve compatibility conditions
\begin{eqnarray}
    x\rhd_{-1}(ff') &=& (tx^l_{(1)} \rhd_0 f_{(1)}^l) \cdot^*((x^l_{(2)}\lhd_{-1} f_{(2)}^l) \rhd_{-1}f')+(x^r_{(1)} \rhd_0 f_{(1)}^l)\cdot^*(\Upsilon_{x^r_{(2)}\lhd_0f_{(2)}^l} f')\nonumber\\
    &\qquad&+~ (\Upsilon_{x^l_{(1)}}f_{(1)}^r) \cdot^*(\Upsilon_{x^l_{(2)}\tilde\Upsilon_{f_{(2)}^r}}f') + (x_{(1)}^r\rhd_{-1}f_{(1)}^r)\cdot^*(\Upsilon_{x_{(2)}^r\lhd_0 (t^*f_{(2)}^r)}f'),\nonumber\\
    \Upsilon_y(ff') &=& ((ty_{(1)})\rhd_0 f^l_{(1)})\cdot^*( \Upsilon_{y_{(2)}\lhd_0 f^l_{(2)}} f')+ (\Upsilon_{y_{(1)}}f^r_{(1)})\cdot^*(\Upsilon_{y_{(2)}\lhd_0 (t^*f^r_{(2)})}f'),\nonumber\\
    x\rhd_0(f\cdot^* g) &=& (tx_{(1)}^l\rhd_0 f_{(1)}^l) \cdot^*((x^l_{(2)}\lhd_{-1} f^l_{(2)})\rhd_0 g)+(x_{(1)}^r\rhd_0f_{(1)}^l)\cdot^*(t(x_{(2)}^r\lhd_0 f_{(2)}^l)\rhd_0g)\nonumber\\
    &\qquad& +~ (\Upsilon_{x_{(1)}^l}f_{(1)}^r)\cdot^*(t(x_{(2)}^l\tilde\Upsilon_{f_{(2)}^r})\rhd_0g)+(x_{(1)}^r\rhd_{-1}f_{(1)}^r)\cdot^*(t(x_{(2)}^r\lhd_0(t^*f_{(2)}^r)) \rhd_0 g) ,\nonumber\\
    ty\rhd_0 (f\cdot^* g) &=& (ty_{(1)}\rhd_0f_{(1)}^l) \cdot^* (t(y_{(2)}\lhd_0f_{(2)}^l)\rhd_0 g) + (\Upsilon_{y_{(1)}}f_{(1)}^r)\cdot^*(t(y_{(2)}\lhd_0(t^*f_{(2)}^r))\rhd_0 g),\nonumber\\
    x\rhd_0(g\cdot^* f) &=& (tx_{(1)}^l\rhd_0 g_{(1)}) \cdot^*(\Upsilon_{x^l_{(2)}\tilde\Upsilon_{g_{(2)}}} f)+(x_{(1)}^r\rhd_0g_{(1)})\cdot^*(\Upsilon_{x_{(2)}^r\lhd_0(t^*g_{(2)})}f)\nonumber\\
    ty\rhd_0 (g\cdot^* f) &=& (ty_{(1)}\rhd_0g_{(1)}) \cdot^* (\Upsilon_{y_{(2)}\lhd_0(t^*g_{(2)})}f),\nonumber\\
    (xx') \lhd_{-1} f &=& (x \tilde\Upsilon_{tx'^l_{(1)} \rhd_0 f^l_{(1)}})\cdot(x'^l_{(2)}\lhd_{-1} f^l_{(2)}) + (x\tilde\Upsilon_{x'^r_{(1)}\rhd_0 f^l_{(1)}})\cdot (x'^r_{(2)} \lhd_0f^l_{(2)}) \nonumber\\
    &\qquad& +~ (x \tilde\Upsilon_{\Upsilon_{x'^l_{(1)}}f^r_{(1)}})\cdot (x'^l_{(2)} \tilde\Upsilon_{f^r_{(2)}}) + (x \lhd_{-1}(x'^r_{(1)}\rhd_{-1} f^r_{(1)}))\cdot(x'^r_{(2)} \lhd_0 (t^*f^r_{(2)})),\nonumber\\
    (xx')\tilde\Upsilon_g&=& (x \tilde\Upsilon_{tx'^l_{(1)} \rhd_0 g_{(1)}}) \cdot (x'^l_{(2)}\tilde\Upsilon_{g_{(2)}})+ (x\tilde\Upsilon_{x'^r_{(1)} \rhd_0g_{(1)}})\cdot(x'^r_{(2)} \lhd_0(t^*g_{(2)})),\nonumber\\
    (y\cdot x)\lhd_0 f &=& (y \lhd_0 t^*(tx_{(1)}^l\rhd_0 f_{(1)}^l))\cdot(x^l_{(2)} \lhd_{-1}f^l_{(2)}) + (y \lhd_0t^*(x_{(1)}^r\rhd_0 f_{(1)}^l))\cdot (x_{(2)}^r \lhd_0f_{(2)}^l) \nonumber\\
    &\qquad&+~(y \lhd_0t^*(\Upsilon_{x_{(1)}^l}f_{(1)}^r))\cdot(x_{(2)}^l \tilde\Upsilon_{f_{(2)}^r}) + (y \lhd_0(x_{(1)}^r \rhd_{-1}f_{(1)}^r))\cdot(x_{(2)}^r \lhd_0 (t^*f_{(2)}^r)),\nonumber\\
    (y\cdot x) \lhd_0 t^*g &=& (y \lhd_0t^*(tx_{(1)}^l \rhd_0g_{(1)}))\cdot(x_{(2)}^l\tilde\Upsilon_{g_{(2)}}) + (y\lhd_0 t^*(x_{(1)}^r \rhd_0g_{(1)}))\cdot(x_{(2)}^r \lhd_0t*^g_{(2)}),\nonumber\\
    (x\cdot y)\lhd_0 f &=& (x\tilde\Upsilon_{ty_{(1)}\rhd_0 f^l_{(1)}}) \cdot (y_{(2)} \lhd_0 f^l_{(2)}) + (x \tilde\Upsilon_{\Upsilon_{y_{(1)}}f_{(1)}^r})\cdot (y_{(2)} \lhd_0(t^*f_{(2)}^r)),\nonumber\\
    (x\cdot y) \lhd_0 t^*g &=& (x \tilde\Upsilon_{ty_{(1)}\rhd_0 g_{(1)}}) \cdot(y_{(2)}\lhd_0 t^*g_{(2)}),\nonumber
\end{eqnarray}
where we have made use of the Sweedler notation  \eqref{sweed}. 

\medskip 

We define a shorthand notation where $z=(y,x)\in\cG,h=(g,f)\in\cG^*$, such that the following
\begin{eqnarray}
    z\bar\rhd (h\cdot^* h') &=& (z_{(1)}\bar\rhd h_{(1)})\cdot^* ((z_{(2)}\bar\lhd h_{(2)}) \bar\rhd h') ,\label{comprel1}\\
    (z\cdot z')\bar\lhd h &=& (z \bar\lhd (z'_{(1)}\bar\rhd h_{(1)}))\cdot (z'_{(2)}\bar\lhd h_{(2)})\label{comprel2}
\end{eqnarray}
encode respectively the first six and last six of the above conditions. We also have the cross relations
\begin{equation}
        z_{(1)}\bar\lhd h_{(1)}\otimes z_{(2)}\bar\rhd h_{(2)} = z_{(2)}\bar\lhd h_{(2)}\otimes z_{(1)}\bar\rhd h_{(1)},\label{crossrel}
\end{equation}
as well as the unity axioms against the unit $\eta$ and counit $\epsilon$,
\begin{equation}
    z\bar\rhd \eta = \epsilon(z),\qquad \eta\bar\lhd h = \epsilon(h).\label{grpdunit}
\end{equation}

\begin{definition}\label{def:matchedpair}
We call a tuple $(\cG,\cG^*)$ of (finite dimensional) 2-bialgebras satisfying  \eqref{comprel1}-\eqref{grpdunit} a {\bf matched pair}.    
\end{definition}

\begin{remark}\label{skelqd}
Note that in the skeletal case $t,t^*=0$, the crossed relations  \eqref{comprel1}, \eqref{comprel2} reduce to just two non-trivial equations. These are given by
\begin{eqnarray}
    x\rhd_{-1}(ff') &=& (x^r_{(1)} \rhd_0 f_{(1)}^l)\cdot^*(\Upsilon_{x^r_{(2)}\lhd_0f_{(2)}^l} f')+ (\Upsilon_{x^l_{(1)}}f_{(1)}^r) \cdot^*(\Upsilon_{x^l_{(2)}\tilde\Upsilon_{f_{(2)}^r}}f')\nonumber\\
    &\equiv & (x_{(1)}\bar\rhd f_{(1)})\cdot^* ((x_{(2)}\bar\lhd f_{(2)})\bar\rhd f'),\nonumber\\
    (xx') \lhd_{-1} f &=&  (x\tilde\Upsilon_{x'^r_{(1)}\rhd_0 f^l_{(1)}})\cdot (x'^r_{(2)} \lhd_0f^l_{(2)}) + (x \tilde\Upsilon_{\Upsilon_{x'^l_{(1)}}f^r_{(1)}})\cdot (x'^l_{(2)} \tilde\Upsilon_{f^r_{(2)}})\nonumber\\
    &\equiv& (x\bar\lhd (x_{(1)}'\bar\rhd f_{(1)}))\cdot (x_{(2)}\bar\lhd f_{(2)}),\label{skelcomprel} 
\end{eqnarray}
where we have used a convenient notation for brevity. One may notice that these are precisely the usual crossed relations for a quantum double group (cf. \cite{Majid:1996kd}) of a semidirect product 2-bialgebra $\cG_{-1}\rtimes \cG_0$, where $\cG_{-1}$ is nuclear.
\end{remark}

\subsection{Construction of the strict 2-quantum double}\label{str2qdconstruct}
Inspired by the above, we now begin our construction of the general 2-quantum double given a matched pair $(\cG,\cG^*)$. We shall explicitly construct its 2-bialgebra structure such that its self-duality is manifest.

\paragraph{2-algebra structure.}
We consider $D(\cG)$ defined in terms of the  graded components given by
\begin{equation}
    D(\cG)_0\cong \cG_0\otimes \cG_{-1}^*\ni (x,f),\qquad D(\cG)_{-1} \cong \cG_{-1}\otimes \cG_0^*\ni (y,g),\nonumber
\end{equation}
for which we have a "right-moving" semidirect product $\overrightarrow{\times} = (\cdot,\bar\rhd)$, giving rise to $D(\cG)_{-1}\overrightarrow{\rtimes} D(\cG)_0$. Similarly, we also have a "left-moving" semidirect product $\overleftarrow{\times} = (\cdot^*,\bar\lhd)$ giving rise to $D(\cG)_{-1}\overleftarrow{\rtimes}D(\cG)_0$.  The combined $t$-map $T=t\otimes t^*$ is equivariant with respect to these semidirect products
\begin{equation}
    t^*(x\rhd_0 g) = x\rhd_{-1} t^*g,\qquad t(y\lhd_0f) = (ty)\lhd_{-1} f,\label{ddequivar}
\end{equation}
since the coadjoint action is a 2-representation,
while the definition of the adjoint $t^*$ implies
\begin{eqnarray}
    (ty)\rhd_0 g = \Upsilon_y(t^*g),&\qquad& y\lhd_0(t^*g) =  (ty)\tilde\Upsilon_g,\nonumber\\
    (ty)\rhd_{-1}f = t^*(\Upsilon_yf),&\qquad& x\lhd_{-1}(t^*g) = t(x\tilde\Upsilon_g).\label{genpeif}
\end{eqnarray}
These are in fact generalizations of the Peiffer identity.
\begin{proposition}\label{gen2alg}
If $\bar\rhd,\bar\lhd$ are given by the coadjoint representations (see  \eqref{coadj}), then  \eqref{genpeif} reproduces the Peiffer identity.
\end{proposition}
\begin{proof}
This is a direct computation. By the equality in the second row of  \eqref{genpeif}, we have
\begin{equation}
    \langle f, y\cdot ty'\rangle=-\langle t^*\Upsilon_yf,y'\rangle=-\langle (ty)\rhd_{-1}f,y'\rangle = \langle f,ty\cdot y'\rangle,\nonumber
\end{equation}
giving $ty\cdot y' = y\cdot ty'$. Now by the fact that $t$ is an algebra homomorphism, we have
\begin{eqnarray}
    \langle (ty)\rhd_0 g,ty'\rangle &=& -\langle g,(ty)(ty')\rangle = -\langle g,t(yy')\rangle,\nonumber\\
    \langle \Upsilon_y(t^*g),ty'\rangle &=& -\langle t^*g,y\cdot ty'\rangle=-\langle g,t(y\cdot ty')\rangle,\nonumber
\end{eqnarray}
for which the first row of  \eqref{genpeif} states $yy' = y\cdot ty'$. Altogether yields
\begin{equation}
    yy' = y\cdot t(y') = t(y)\cdot y'\nonumber
\end{equation}
for any $y,y'\in \cG_{-1}$, which is precisely the Peiffer identity on $\cG$. Similarly, if $\bar\lhd$ is the coadjoint representation then  \eqref{genpeif} reproduces the Peiffer identity on $\cG^*$. 
\end{proof}
\noindent In other words, the Peiffer identity in $D(\cG)$ is {\it by definition} given as in  \eqref{genpeif}. The multiplication between the sectors $\cG_{-1},\cG^*_{-1}$ is given by $yg=\Upsilon_y(t^*g)$ and $gy=(ty)\tilde\Upsilon_g$.

Now that we have defined the product of the graded components and the $t$-map associated to $D(\cG)$, we can identify the bimodule structure.

\medskip

We combine the right-moving $\overrightarrow{\times}=(\cdot,\bar\rhd)$ and left-moving  $\overleftarrow{\times}=(\cdot^*,\bar\lhd)$ multiplications on $D(\cG)$ to form $\hat\cdot = \overrightarrow{\times}+ \overleftarrow{\times}$,
\begin{equation}
    (z,h)\hat\cdot (z',h') = (z\cdot z' + z\bar\lhd h' + z'\bar\lhd h, h\cdot h' + z\bar\rhd h'+z'\bar\rhd h), \quad z,z'\in\cG, \, h,h'\in\cG^*.\label{2ddalgstructure}
\end{equation}
Since $\hat\cdot$ is a combination of the internal 2-algebra structures of $\cG,\cG^* $ and the 2-representations $\bar\rhd,\bar\lhd$, we have respectively the Peiffer conditions and associativity for $\cG,\cG^*$, as well as the 2-representation properties  \eqref{ddequivar}, \eqref{genpeif} and the matched pair conditions  \eqref{comprel1}, \eqref{comprel2}, \eqref{grpdunit}. These imply that the map $\hat\cdot$ \begin{itemize}
    \item[(i)] is associative, \item[(ii)] makes $D(\cG)_{-1}$ into a $D(\cG)_0$-bimodule, \item[(iii)] satisfies the Peiffer conditions under $T=t\otimes t^*$. 
\end{itemize}
Hence $(D(\cG),~\hat\cdot~)$ is a 2-algebra.

\paragraph{2-coalgebra structure.} We intend now to construct the coproduct $\Delta_D:D(\cG)\rightarrow D(\cG)^{2\otimes}$. We have to build the components 
\begin{eqnarray}    
{\Delta_D}_{-1}&:& D(\cG)_{-1}\rightarrow D(\cG)_{-1}\otimes D(\cG)_{-1} = (\cG_{-1}\otimes \cG^*_{0})\otimes (\cG_{-1}\otimes \cG^*_{0}) \nonumber \\
{\Delta_D}_{0}&:& D(\cG)_0\rightarrow (D(\cG)_{-1}\otimes D(\cG)_{0}) \oplus (D(\cG)_{0}\otimes D(\cG)_{-1})
\nonumber
\end{eqnarray}
We can directly infer some of the components ${\Delta_D}_{-1}$  from the coproducts $\Delta_{-1},\Delta^*_{-1}$ of $\cG,\cG^*$. Explicitly, it is defined as
\begin{eqnarray}
{\Delta_D}_{-1}^d&=&\Delta_{-1}\otimes \Delta^*_{-1}. 
\nonumber    %
 \end{eqnarray}   
 This coproduct  by construction encodes  the separate coproducts $\Delta=\Delta_D|_\cG,\Delta^*=\Delta_D|_{\cG^*}$ by restriction and it is consistent with the products of each 2-algebras. These components are diagonal in a sense and we need to introduce some off diagonal contributions, 
 \begin{eqnarray}
 \xi_{-1}:\cG_{-1}\rightarrow \cG^*_{0}\otimes \cG_{-1}
 ,\quad \zeta_{-1}: \cG_{0}^*\rightarrow 
 \cG_{-1}\otimes \cG_{0}^*,\nonumber
 \end{eqnarray}
 such that 
 \begin{equation}
   \Delta_D{}_{-1}=(\Delta_D)_{-1}^d +\xi_{-1} \otimes \zeta_{-1}.   \label{copD1}
 \end{equation} 
  $\xi_{-1}$ and $\zeta_{-1}$   can be interpreted as coactions and  are defined as dualized components of the coadjoint actions. Taking as usual $ (x,f)\in D(\cG)_0\cong \cG_0\otimes \cG_{-1}^*$ and $(y,g)\in D(\cG)_{-1}\cong \cG_{-1}\otimes \cG_{0}^*$ we have 
\begin{gather}
    \langle \xi_{-1}(y),x\otimes f\rangle := \langle y, x\rhd_{-1}f\rangle,\qquad
    \langle \zeta_{-1}(g),f\otimes x\rangle := \langle g,x\lhd_{-1}f\rangle.
    \label{2coact1}
\end{gather}
These coactions are 2-algebra maps by  \eqref{crossrel}, and hence $\Delta_D{}_{-1}$ satisfies  \eqref{2algcoprod} on $D(\cG)$.

\smallskip

In a similar way, ${\Delta_D}_{0}$ is also made of several components. We use the components $\Delta_0:\cG_0\rightarrow(\cG_{0}\otimes \cG_{-1})\oplus (\cG_{-1}\otimes \cG_{0})$ and $\Delta^*_0:\cG^*_{-1}\rightarrow(\cG^*_{0}\otimes \cG^*_{-1})\oplus (\cG_{-1}^*\otimes \cG_{0}^*) $ of $\cG$ and $\cG^*$ to define the "diagonal" contribution,
\begin{eqnarray}
     (\Delta_D^l)^d_0 &:=& \Delta^l_0 \otimes \Delta^{*l}_0 , \qquad (\Delta_D^r)_0:= \Delta^r_0 \otimes \Delta^{*r}_0. \nonumber
\end{eqnarray}
Once again, by restriction, one recovers the   separate coproducts  $\Delta^{r,l}_0$ and $\Delta^{*r,l}_0$ on respectively $\cG$ and $\cG^*$. 

We also have to recover the mixed terms. 
\begin{eqnarray}
    \xi_0^l: \cG_0 \rightarrow \cG^*_{0}\otimes \cG_0, \quad \xi_0^r: \cG_0 \rightarrow\cG_{-1}^*\otimes \cG_{-1} \nonumber\\
     \zeta_0^l: \cG^*_{-1}\rightarrow\cG_{-1}\otimes \cG^*_{-1}, \quad \zeta_0^r:\cG_{0} \rightarrow \cG_{0}\otimes \cG^*_{0}\nonumber,
\end{eqnarray}
such that 
\begin{equation}
  (\Delta_D^{r,l})_0 := (\Delta_D^l)^d_0 + \xi^{r,l}_{0}\otimes \zeta^{r,l}_{0}.  \label{copD2}
\end{equation}

These mixed terms  are again obtained  by dualizing the components of the coadjoint actions
\begin{gather}
    \langle \xi_0^l(x),x'\otimes g\rangle := \langle x,x'\rhd_0g\rangle,\qquad \langle \xi_0^r(x),y\otimes f\rangle := \langle x,\Upsilon_yf\rangle,\nonumber\\
    \langle \zeta_0^l(f),f'\otimes y\rangle := \langle f, y\lhd_{-1}f'\rangle,\qquad \langle \zeta_0^r(f),g\otimes x\rangle := \langle f, x\tilde\Upsilon_g\rangle.\label{2coact}
\end{gather}
Once again, these coactions are 2-algebra maps by  \eqref{crossrel}, and hence $\Delta_D$ satisfies  \eqref{2algcoprod} on $D(\cG)$. 

\smallskip 

We now need to show that it also satisfies  \eqref{cohgrpd+}, \eqref{cohgrpd-}. We do this by leveraging the self-duality $D(\cG)\cong D(\cG)^*$ under the natural non-degenerate self-pairing via  \eqref{pairing} (cf. \cite{Bai_2013}),
\begin{equation}
    \langle (z,h),(z',h')\rangle = \langle f,y'\rangle + \langle g,x'\rangle + \langle f',y\rangle+ \langle g',x\rangle.\label{selfpair}
\end{equation}
By {\bf Proposition \ref{2bialg}}, $(D(\cG),\cdot)$ is an associative 2-algebra iff $(D(\cG)^*\cong D(\cG),\Delta_D)$ is a coassociative 2-coalgebra, which implies  \eqref{cohgrpd+}-\eqref{cohgrpd-} for $\Delta_D$.

\begin{definition}\label{2dd}
    We call  the 2-bialgebra $$\cG \bar{\bowtie}\cG^*:=D(\cG) =  (D(\cG)_{-1}\xrightarrow{T} D(\cG)_0, \hat\cdot,\Delta_D)$$ built out of the the matched pair of strict 2-bialgebras  $(\cG,\cG^*)$  with the product, coproduct, and counit given respectively in  \eqref{2ddalgstructure}, \eqref{copD1} and \eqref{copD2}, \eqref{2coact},   the {\bf strict 2-quantum double} of $\cG$. 

\end{definition}

\subsection{Factorizability of 2-bialgebras}
Conversely, we can determine when a strict 2-bialgebra is actually a strict 2-quantum double, which is given by a {\it factorizability/splitting} condition. In fact, we prove that any 2-bialgebra that factorizes appropriately into 2-bialgebras will automatically determine a 2-quantum double.
\begin{theorem}\label{2qd}

Suppose a (unital) 2-bialgebra $(\cK=\cK_{-1}\xrightarrow{T}\cK_0,\hat\cdot)$ factorizes into two (unital) sub-2-bialgebras $\cG,\cH$, meaning that there is a span of inclusions,
\begin{equation}
    \cG\xhookrightarrow{\iota}\cK\xhookleftarrow{\jmath}\cH, \label{span}
\end{equation} 
such that $\hat\cdot\circ (\iota\otimes\jmath)$ is an isomorphism of 2-vector spaces and such that  the 2-sub-bialgebras $\cG,\cH$ are dually paired, with their $t$-maps satisfying $\langle t_\cG\cdot,\cdot\rangle = \langle \cdot,t_\cH\cdot\rangle$. Then $(\cG,\cH)$ is a matched pair and $\cK\cong \cG \bar{\bowtie} \cH$.
\end{theorem}

\begin{proof}
Let $\cK=\cK_{-1}\xrightarrow{T}\cK_0$ be a 2-bialgebra factorizing into two 2-subbialgebras $\cG,\cH$, with typical elements $w\in \cK_0$ and $e\in \cK_{-1}$. Its 2-algebra structure $\hat\cdot$ contains a multiplication $ww'$ in $\cK_0$ and a $\cK_0$-bimodule structure $w\hat\cdot e,e\hat\cdot w$ on $\cK_{-1}$, which are both associative. Since  \eqref{span} is a span of 2-vector spaces, we have
\begin{equation}
    T\circ (\iota_{-1}\otimes \jmath_{-1}) = (\iota_0\circ t_\cG)\otimes (\jmath_0\circ t_\cH) = (\iota_0\otimes \jmath_0)\circ (t_\cG\otimes t_\cH), \nonumber
\end{equation}
where $t_G,t_H$ are the $t$-maps in $\cG,\cH$ respectively, and $\iota_{-1},\iota_0$ are the graded components of the inclusion $\iota$; similarly for $\jmath$.

We now separate the bimodule structure $\hat\cdot$ into components according to the span  \eqref{span},
\begin{equation}
    \hat\rhd\equiv \hat\cdot|_{\operatorname{im}(\iota_0\otimes \jmath_{-1})},\qquad \hat\Upsilon\equiv \hat\cdot|_{\operatorname{im}(\iota_{-1}\otimes \jmath_0)},\nonumber 
\end{equation}
then for $e=\iota_{-1}(y),e'=\jmath_{-1}(g)$ where $y\in G_{-1},g\in H_{-1}$ we have
\begin{equation}
    (T\iota_{-1}(y))\hat\rhd \jmath_{-1}(g) = \iota_0(t_\cG y)\hat\rhd \jmath_{-1}(g).\nonumber
\end{equation}
By the Peiffer identity in $\cK$, this should read as a left-multiplication of $y$ on $g$. We lift this action along $t_\cH$ to create a map $\hat\Upsilon_y:\cH_0\rightarrow \cH_{-1}$, for which $\hat\Upsilon_y(t_\cH g)$ denotes the left-multiplicaion of $y$ by $g$. Similarly we have the lift $\hat\Upsilon_g:\cG_0\rightarrow \cG_{-1}$ of the right-multiplication of $g$ on $y$.

Provided we identify $\hat\Upsilon_{y\otimes g} = \Upsilon_y\otimes \tilde\Upsilon_g$, the Peiffer conditions in $\cK$ are then equivalent to the 2-representation properties  \eqref{ddequivar}, \eqref{genpeif}. In particular, the multiplication $y\cdot g=\Upsilon_y(t^*g) = (ty)\tilde\Upsilon_g$ is given by the generalized Peiffer identity as shown in {\bf Proposition \ref{gen2alg}}.

\smallskip

Now we prove that  \eqref{span} is in fact a span of 2-algebras. Due to the linear isomorphsm $\hat\cdot\circ(\iota\otimes\jmath)$, there exists a tuple of well-defined linear maps $\Psi = (\Psi_0,\Psi_{-1};\bar\Psi): \cG\otimes \cH\rightarrow \cH\otimes \cG$, called the {\bf braided transposition}, such that
\begin{eqnarray}
    \iota_0(x)\cdot\jmath_0(f) = \cdot \circ (\jmath_0\otimes\iota_0)\circ \Psi_0(x\otimes f),\nonumber\\
    \iota_0(x)\hat\rhd \jmath_{-1}(g) = \hat\Upsilon \circ(\jmath_{-1}\otimes\iota_0)\circ \Psi^r_{-1}(x\otimes g),\nonumber\\
    \iota_{-1}(y)\hat\Upsilon \jmath_0(f)= \hat\lhd \circ(\jmath_0\otimes\iota_{-1}) \circ \Psi^l_{-1}(y\otimes f),\nonumber\\
    \iota_{-1}(y)\cdot \jmath_{-1}(g) = \cdot \circ(\jmath_{-1}\otimes\iota_{-1})\circ\bar\Psi(y\otimes g),\nonumber
\end{eqnarray}
where $\Psi_{-1} = \Psi_{-1}^l+\Psi_{-1}^r$ and $x\in \cG_0,y\in \cG_{-1},f\in \cH_0,g\in \cH_{-1}$. Due to Peiffer conditions on $\cK$, these braiding maps are not independent and must satisfy
\begin{gather}
    (t_\cH\otimes 1)\circ \Psi_{-1}^r=\Psi_0\circ(1\otimes t_\cH),\qquad (1\otimes t_\cG)\circ \Psi_{-1}^l = \Psi_0\circ(t_\cG\otimes 1),\nonumber\\
    \Psi_{-1}^r\circ(t_\cG\otimes 1) = \bar\Psi =\Psi_{-1}^l\circ(1\otimes t_\cH).\nonumber
\end{gather}
By collecting all of the graded components of $\Psi$ in accordance with the shorthand notation $z=(y,x)\in \cG,h=(g,f) \in \cH$, the definition  of $\Psi$ can be concisely written as
\begin{equation}
    \iota(z)~\hat\cdot~\jmath(h) = \hat\cdot\circ (\jmath\otimes\iota) \circ \Psi(z\otimes h),\label{braid}
\end{equation} 
and the relations between its components is summarized as
\begin{equation}
    T'\circ\Psi_{-1} = \Psi_0\circ T,\qquad \bar\Psi = \Psi_{-1}\circ T,\label{braidcomprel}
\end{equation}
where $T' = t_\cH\otimes t_\cG$ is the $t$-map of the 2-bialgebra $\cK'\cong\cH\otimes\cG$ with $\cG,\cH$ swapped in the span  \eqref{span}.  \eqref{braidcomprel} then implies in particular that $\Psi:\cK\rightarrow \cK'$ is a 2-vector space homomorphism. 

\medskip

We now proceed formally as in the 1-bialgebra case \cite{Majid:1996kd,Majid:1994nw}. The associativity in $\cK$ is
\begin{eqnarray}
    (\iota(z)\hat\cdot\iota(z'))\hat\cdot\jmath(h)&=& \iota(z)\hat\cdot(\iota(z')\hat\cdot\jmath(h)),\nonumber\\
    (\iota(z)\hat\cdot\jmath(h))\hat\cdot \jmath(h')&=& \iota(z)\hat\cdot(\jmath(h)\hat\cdot \jmath(h')),\nonumber
\end{eqnarray}
which yields the {\bf 2-braiding relations}
\begin{eqnarray}
     \Psi \circ(\hat\cdot \otimes \operatorname{id}) &=& (\operatorname{id}\otimes \hat\cdot)\circ \Psi_{12}\circ \Psi_{23},\nonumber\\
     \Psi\circ(\operatorname{id}\otimes \hat\cdot) &=& (\hat\cdot\otimes \operatorname{id})\circ \Psi_{23}\circ \Psi_{12}.\label{2yb1}
\end{eqnarray}
This then allows us to define the actions
\begin{eqnarray}
    \bar\rhd &=& (\operatorname{id}\otimes \epsilon)\circ\Psi:\cG\otimes \cH\rightarrow \cH,\nonumber\\
    \bar\lhd &=& (\epsilon\otimes\operatorname{id})\circ\Psi:\cG\otimes \cH\rightarrow \cG,\nonumber
\end{eqnarray}
where $\epsilon$ denotes the counit map. Applying $\operatorname{id}\otimes \epsilon$ and $\epsilon\otimes\operatorname{id}$ respectively to the first and second equation of  \eqref{2yb1} implies that $\bar\rhd,\bar\lhd$ respect the semidirect product structures $\cG_{-1}\rtimes \cG_0,\cH_{-1}\rtimes \cH_0$, respectively. Together with our above result,  \eqref{span} is in fact a span of 2-algebras.

We now prove that  \eqref{span} is actually a span of 2-bialgebras, which proves the theorem. Applying $\epsilon\otimes\operatorname{id}$ and $\operatorname{id}\otimes \epsilon$ respectively to the first and second  \eqref{2yb1} yields
\begin{equation}
    (z\cdot z')\bar\lhd h = \hat\cdot (z\bar\lhd \Psi(z'\otimes h)),\qquad z\bar\rhd(h\cdot h') = \hat\cdot(\Psi(z\otimes h)\hat\rhd h').\label{2repbruh}
\end{equation}
We now take the coproduct $\Delta_K:\cK\rightarrow \cK^{2\otimes}$ on $\cK$, given in components and Sweedler notation (see  \eqref{grpdcoprod}, \eqref{sweed}) by
\begin{equation}
    (\Delta_{\cK})_{-1}(e) = e_{(1)}\otimes e_{(2)},\qquad (\Delta_{\cK})_0(w) = w_{(1)}^l\otimes w^l_{(2)}+ w_{(1)}^r\otimes w_{(2)}^r;\nonumber
\end{equation}
note $w_{(1)}^l,w_{(2)}^r\in \cK_{-1}$. With the span  \eqref{span}, we can write $w=\iota_0(x)\jmath_0(f),e=\iota_{-1}(y)\jmath_{-1}(g)$ for some appropriate elements $x,f,y,g$ such that
\begin{eqnarray}
     (\Delta_{\cK})_{-1}(y,g) &=& (y_{(1)}\otimes g_{(1)})\otimes (y_{(2)}\otimes g_{(2)}),\nonumber\\
     (\Delta_{\cK}^l)_0(x,f) &=& (x^l_{(1)}\otimes f^l_{(1)})\otimes (x^l_{(2)}\otimes f^l_{(2)}),\nonumber\\
     (\Delta_{\cK}^r)_0(x,f) &=& (x^r_{(1)}\otimes f^r_{(1)})\otimes (x^r_{(2)}\otimes f^r_{(2)}).\nonumber
\end{eqnarray}
This then allows us to define coproducts on $\cG,\cH$ by
\begin{eqnarray}
    (\Delta_{\cG})_{-1}(y) = y_{(1)}\otimes y_{(2)},&\qquad&(\Delta_{\cG})_0(x) = x^l_{(1)}\otimes x^l_{(2)}+x_{(1)}^r\otimes x_{(2)}^r,\nonumber\\
    (\Delta_{\cH})_{-1}(g) = g_{(1)}\otimes g_{(2)},&\qquad&(\Delta_{\cH})_0(f) = f^l_{(1)}\otimes f^l_{(2)}+f^r_{(1)}+f^r_{(2)},\nonumber
\end{eqnarray}
whence $\Delta_{\cK} = \Delta_{\cG\otimes \cH}$, which implies that $\hat\cdot\circ (\iota\otimes\jmath)$ and $\hat\cdot\circ(\jmath\otimes\iota)$ by construction respects the coproducts. 

As such, $\Psi$ is a 2-coalgebra map. In particular, we have
\begin{equation}
    \Delta_{\cK}\circ\Psi = (\Psi\otimes\Psi)\circ\Delta_{\cK'},\qquad (\epsilon\otimes \epsilon)\circ\Psi = \epsilon\otimes\epsilon\label{2repbruh'}
\end{equation}
where $\cK'$ is the 2-bialgebra with $\cG,\cH$ swapped in the span  \eqref{span}. An application of $\epsilon\otimes \operatorname{id}\otimes \epsilon\otimes\operatorname{id}$ and $\operatorname{id}\otimes \epsilon\otimes \operatorname{id}\otimes \epsilon$ to  \eqref{2repbruh'} gives
\begin{equation}
    \Delta_{\cG} \circ \bar\lhd = (\bar\lhd\otimes\bar\lhd) \circ \Delta_{\cK},\qquad \Delta_{\cH} \circ\bar\rhd = (\bar\rhd\otimes\bar\rhd)\circ\Delta_{\cK},\nonumber
\end{equation}
which ensures that $\bar\rhd,\bar\lhd$ are 2-coalgebra maps. 

Now applying $\epsilon\otimes\operatorname{id}\otimes\operatorname{id}\otimes \epsilon$ and $\operatorname{id}\otimes\epsilon\otimes\epsilon\otimes\operatorname{id}$ to  \eqref{2repbruh'} yields
\begin{eqnarray}
     z_{(1)}\bar\lhd h_{(1)} \otimes z_{(2)}\bar\rhd h_{(2)} &=& \tau\circ\Psi(z\otimes h),\nonumber\\
     z_{(1)}\bar\rhd h_{(1)} \otimes z_{(2)}\bar\lhd h_{(2)} &=& \Psi(z\otimes h).\nonumber
\end{eqnarray}
Using the second equation, together with  \eqref{2repbruh}, gives  \eqref{grpdunit} and
\begin{eqnarray}
     z\bar\rhd(h\hat\cdot|_\cH h') &=& \hat\cdot((z_{(1)}\bar\rhd h_{(1)}\otimes z_{(2)}\bar\lhd h_{(2)})\hat\rhd h')= (z_{(1)}\bar\rhd h_{(1)})\hat\cdot|_\cH((z_{(2)}\hat\lhd h_{(2)})\bar\rhd h'),\nonumber\\
     (z\hat\cdot|_\cG z')\bar\lhd h &=& \hat\cdot(z\bar\lhd (z'_{(1)}\bar\rhd h_{(1)}\otimes z'_{(2)}\bar\lhd h_{(2)}))= (z\bar\lhd (z'_{(1)}\bar\rhd h_{(1)}))\hat\cdot|_\cG(z'_{(2)}\hat\lhd h_{(2)}),\nonumber
\end{eqnarray}
which are precisely the matched pair conditions  \eqref{comprel1}, \eqref{comprel2} for $\hat\cdot|_\cG=\cdot,\hat\cdot|_{\cH}=\cdot^*$. On the other hand, using the first equation gives  \eqref{crossrel}. Thus  \eqref{span} is a span of 2-bialgebras and so $K \cong \cG \bar{\bowtie}\cH$.
\end{proof}

Note that the span  \eqref{span} factorizes the 2-algebra structure on $\cK$ into the right- $\overrightarrow{\times}=(\hat\cdot|_{\cG},\bar\lhd)$ and left-moving $\overleftarrow{\times}=(\hat\cdot|_{\cH}^\text{opp},\bar\rhd^\text{opp})$ 2-algebra structures. In other words, in order to identify $\cK$ with a 2-quantum double, we must have \cite{Majid:1994nw}
\begin{equation}
    \cK \cong \cG\bar\bowtie\cH \cong D(\cG,\cH^\text{opp}),\label{skewdouble}
\end{equation}
where $\cH^\text{opp}$ denotes the opposite 2-algebra; see Appendix \ref{2-hopf}.

This is because, as can be seen in  \eqref{2coact}, the back-action $\bar\lhd$ is written from right to left.

\subsection{Universal characterization of quantum 2-R-matrices}\label{quasitrihopf}
As we have mentioned in the beginning of this section, we wish to leverage the 2-quantum double construction we have given above in order to provide a notion of a quantum R-matrix on a 2-bialgebra $\cG$. More precisely, we shall use the skew-pairing on $\cG$ used in forming the 2-quantum double $D(\cG,\cG) = \cG\bar{\bowtie}\cG^\text{opp}$ in order to provide a definition of the 2-R-matrix on $\cG$ (and \textit{not} just on $D(\cG)$!). We shall show in Section \ref{braided2cat} that our definition of a 2-R-matrix indeed gives rise to a {\it braiding} on the 2-representations of $\cG$.

\paragraph{Review of the 1-bialgebra case.} We first recall the explicit construction of the R-matrix for the ordinary 1-bialgebra $H$. It was noted by Majid (see eg. \cite{Majid:1994nw,Majid:1996kd}) that, in forming the quantum double $D(H,H) =H\bowtie H^\text{opp}$ as a double crossed product, the (non-degenerate) skew-pairing which dualizes $H$ with itself satisfies 
\begin{gather}
    \langle xx',g\rangle_\text{sk} = \langle x\otimes x',\Delta(g)\rangle_\text{sk}=\langle x,g_{(1)}\rangle_\text{sk}\, \langle  x',g_{(2)}\rangle_\text{sk} ,\nonumber\\
    \langle x,gg'\rangle_\text{sk} = \langle \Delta(x),g'\otimes g\rangle_\text{sk} = \langle x_{(1)},g'\rangle_\text{sk} \, \langle x_{(2)},g\rangle_\text{sk},\nonumber
\end{gather}
where $x,x'\in H$ and $g,g'\in H^\text{opp}\cong H$. If we define this skew-pairing as a functional $\langle\cdot,\cdot\rangle_\text{sk} = R^*:H^{2\otimes}\rightarrow k$, then we see that the above conditions translate to the following operational condition
\begin{equation}
    R^*\circ (\mu\otimes \id) = R^*_{13}R^*_{23},\qquad R^*\circ (\id\otimes \mu) = R^*_{13}R^*_{12},\label{dualrmat}
\end{equation}
which is nothing but the defining properties of a \textit{dual} R-matrix.  Indeed, together with the property 
\begin{equation}
 x_{(1)} \, x'_{(1)}\,R^*(x'_{(2)},x_{(2)})= R^*(x'_{(1)},x_{(1)})\, x'_{(2)}\, x_{(2)}, \quad x,x'\in H
\end{equation}
we obtain the (dual) Yang-Baxter equations \cite{Majid:1994nw, Majid:1996kd}.

In summary, we see that the Drinfel'd double $D(H,H) = H\bowtie H^\text{opp}$ is canonically equipped with a quasitriangular $R$-matrix, which we call the \textit{universal $R$-matrix}. This then allows us to characterize the algebraic properties satisfied by $R$-matrices by studying Drinfel'd doubles and its skew self-pairing --- a (dual) quasitriangularity structure $R^*$ on a Hopf algebra $A$ should satisfy the operational condition \eqref{dualrmat} and the Yang-Baxter equations. If such a $R^*$ furthermore defines a self skew-paring $\langle\cdot,\cdot\rangle_\text{sk}$ on $A$, then $A=D(H)$ should itself form a Drinfel'd double. 

\medskip

\paragraph{(Dual) 2-$R$-matrix.}
We now follow an analogous treatment to characterize dual 2-$R$-matrices of a quasitriangular 2-bialgebra $\cG$. Take the 2-quantum double $D(\cG,\cG)$, whose underlying duality pairing  \eqref{selfpair} is given by a non-degenerate self-duality skew-pairing $\langle\cdot,\cdot\rangle_\text{sk}:\cG\otimes\cG\rightarrow k$. Explicitly, this pairing satisfies
\begin{gather}
    \langle x\cdot_l y,f\rangle_\text{sk} = \langle x\otimes y,\Delta_0^l(f)\rangle_\text{sk},\qquad \langle y\cdot_r x,f\rangle_\text{sk} = \langle y\otimes x,\Delta_0^r(f)\rangle_\text{sk},\nonumber\\
    \langle x,f\cdot_l g\rangle_\text{sk} = \langle \Delta_0^r(x),g\otimes f\rangle_\text{sk},\qquad \langle x,g\cdot f\rangle_\text{sk} = \langle \Delta_0^l(x),f\otimes g\rangle_\text{sk},\label{builtinR}
\end{gather}
and also in addition to the fact that it should respect the $t$-map $T = t\otimes t$ on $D(\cG,\cG)$,
\begin{equation}
    \langle ty,g\rangle_\text{sk} =  \langle y,tg\rangle_\text{sk},\nonumber
\end{equation}
where $x,f,f'\in \cG_0$ and $y,g\in\cG_{-1}$. Writing the skew-pairing in terms of a functional $\cR^*: \cG^{2\otimes}\rightarrow k$ by
\begin{equation}
    \cR^*_l(y,f)=\langle y,f\rangle_\text{sk},\qquad \cR^*_r(x,g) = \langle x,g\rangle_\text{sk} ,\nonumber
\end{equation}
we can rewrite \eqref{builtinR} as
\begin{gather}
    \cR^*_l\circ(\cdot_l \otimes \id) =  (\cR^*_r)_{13}(\cR^*_l)_{23},\qquad \cR^*_l\circ(\cdot_r\otimes \id) = (\cR^*_l)_{13}(\cR^*_r)_{23},\nonumber\\
    \cR^*_r\circ(\id\otimes\cdot_l) = (\cR^*_l)_{13}(\cR^*_r)_{12},\qquad \cR^*_r\circ(\id\otimes \cdot_r) = (\cR^*_r)_{13}(\cR^*_l)_{12},\nonumber
\end{gather}
where $\cdot_l,\cdot_r$ denotes respectively the left and right $\cG_0$-actions on $\cG_{-1}$. We also have the compatibility conditions with the $t$-map:
\begin{equation}
    \cR^*_r\circ (\id\otimes t) = \cR^*_l\circ(t\otimes \id)\in \cG_{-1}^{2\otimes}.\nonumber
\end{equation}
An intrinsic notion of higher-quasitriangularity can then be inferred from these properties.

\medskip

Inspired by the above functional $\cR^*$, we characterize the 2-R-matrix $\cR$ for a {single copy} of $\cG$.
\begin{definition}\label{2Rdef}
A quasitriangular \textbf{2-$R$-matrix} of a 2-bialgebra $(\cG,\cdot,\Delta)$ is an element $\cR\in\cG\otimes\cG$ \textit{of total degree -1}, consisting of the graded components
\begin{equation}
    \cR^l \in \cG_{-1}\otimes \cG_0,\qquad \cR^r\in\cG_0\otimes \cG_{-1},\nonumber
\end{equation}
such that the following identities are satisfied:
\begin{enumerate}
    \item the {\bf compatibility with the coproduct}
    \begin{gather}
    (\Delta_0^l \otimes \id) \cR^r = \cR^l_{13}\cdot_l\cR^r_{23},\qquad (\Delta_0^r\otimes \id) \cR^r = \cR^r_{13}\cdot_r\cR^l_{23},\nonumber\\
    (\id\otimes\Delta_0^l) \cR^l = \cR^l_{13}\cdot_r\cR^r_{12},\qquad (\id\otimes \Delta_0^r) \cR^l = \cR^r_{13}\cdot_l\cR^l_{12}, \label{def2R1}
    \end{gather}
    \item the {\bf coproduct permutation identity}
    \begin{equation}
      \cR^r\Delta_0^r(x)=(\sigma\circ \Delta_0^l(x))\cR^r,\qquad  \cR^l\Delta_0^l(x)=(\sigma\circ\Delta_0^r(x))\cR^l
    \label{def2R2} 
\end{equation}
for each $x\in\cG_0$, where $\sigma:\cG\otimes\cG\rightarrow \cG\otimes\cG$ is the permutation of tensor factors, and
\item the {\bf equivariance condition}
\begin{equation}
D^-_t\cR=0\iff    (t\,\otimes \id)\cR^l = (\id\otimes\, t)\cR^r\in\cG_0^{2\otimes}.\label{2ybequiv}
\end{equation}
\end{enumerate}
We call the tuple $\cR$ invertible iff $\cR^l,\cR^r$ are both invertible.
\end{definition}
\noindent  For the canonical, universal quasitriangular structure on the double $\cG=D(\cH,\cH)$, see \textit{Remark \ref{quasiR2double}}.

We now derive the categorified notion of the Yang-Baxter equations.
\begin{proposition}\label{2yangbaxter}
     The 2-$R$-matrix of a quasitriangular 2-bialgebra $(\cG,\cdot,\Delta,\cR)$ satisfies the \textbf{2-Yang-Baxter equations}
    \begin{align}
    \cR^r_{23}(\cR^r_{13}\cdot_l\cR_{12}^l) = (\cR^l_{12}\cdot_r\cR^r_{13})\cR^r_{23},\qquad 
    (\cR^l_{23}\cdot_l\cR^r_{13}) \cR^r_{12}= \cR^r_{12} (\cR^r_{13}\cdot_r\cR^l_{23} )  ,    \nonumber\\
      \cR^l_{23}(\cR^l_{13}\cdot_r\cR_{12}^r) = (\cR^r_{12}\cdot_l\cR^l_{13})\cR^l_{23},\qquad 
    (\cR^r_{23}\cdot_r\cR^l_{13}) \cR^l_{12}= \cR^l_{12} (\cR^l_{13}\cdot_l\cR^r_{23} )     .\label{2yangbax2}
\end{align}
\end{proposition}
\begin{proof}
    Recall that $\cR$ is quasitriangular iff $\cR^l,\cR^r$ are square and invertible. This pairs $\cG$ with itself and hence $\operatorname{dim}\cG_0 = \operatorname{dim}\cG_{-1}$.
We calculate $(\id\otimes \sigma\circ \Delta_0^l)\cR^l$ and     $(\sigma\circ \Delta_0^l\otimes \id )\cR^r$, as well as  $(\id\otimes \sigma\circ \Delta_0^r)\cR^l$ and     $(\sigma\circ \Delta_0^r\otimes \id )\cR^r$ in two ways.
First using  \eqref{def2R1}, we have
\begin{align}
(\id\otimes \sigma\circ \Delta_0^l)\cR^l &= (\id\otimes \sigma)   \cR^l_{13}\cdot_r\cR^r_{12} = \cR^l_{12}\cdot_r\cR^r_{13}, \nonumber\\
(\sigma\circ \Delta_0^l\otimes \id )\cR^r&= (\sigma\otimes \id ) \cR^l_{13}\cdot_l\cR^r_{23} = \cR^l_{23}\cdot_l\cR^r_{13}, \nonumber\\
(\id\otimes \sigma\circ \Delta_0^r)\cR^l &= (\id\otimes \sigma)\cR^r_{13}\cdot_l\cR^l_{12}= \cR^r_{12}\cdot_l\cR^l_{13},\nonumber\\
(\sigma\circ \Delta_0^r\otimes \id )\cR^r &= (\sigma\otimes \id ) \cR^r_{13}\cdot_r\cR^l_{23}=\cR^r_{23}\cdot_r\cR^l_{13}. \nonumber
\end{align}    
On the other hand from     \eqref{def2R2}, we have that,   
\begin{align}
(\id\otimes \sigma\circ \Delta_0^l)\cR^l &= \cR^r_{23}( (\id\otimes \Delta_0^r )\cR^l)  \cR^{r}{}^{-1}_{23}  = \cR^r_{23}( \cR^r_{13}\cdot_l\cR^l_{12})  \cR^{r}{}^{-1}_{23}  \nonumber\\
(\sigma\circ \Delta_0^l\otimes \id )\cR^r &= \cR^r_{12}( ( \Delta_0^r\otimes \id  )\cR^r)  \cR^{r}{}^{-1}_{12} =
\cR^r_{12} (\cR^r_{13}\cdot_r\cR^l_{23} )  \cR^{r}{}^{-1}_{12} , \nonumber\\
(\id\otimes \sigma\circ \Delta_0^r)\cR^l &=  \cR^l_{23}( (\id\otimes \Delta_0^l )\cR^l)  \cR^{l}{}^{-1}_{23}  = \cR^l_{23}( \cR^l_{13}\cdot_r\cR^r_{12})  \cR^{l}{}^{-1}_{23}\nonumber\\
(\sigma\circ \Delta_0^r\otimes \id )\cR^r &= \cR^l_{12}( ( \Delta_0^l\otimes \id  )\cR^r)  \cR^{l}{}^{-1}_{12} =
\cR^l_{12} (\cR^l_{13}\cdot_l\cR^r_{23} )  \cR^{l}{}^{-1}_{12}\nonumber
\end{align}
Putting each equation with its above counterpart  leads to   \eqref{2yangbax2}.
\end{proof}

To see how \textbf{Definition \ref{2Rdef}} reduces to an ordinary R-matrix at degree-0 $\cG_0$, we direct the reader to the beginning of \S \ref{braiding2rep}, as well as \S \ref{braided2cat} in the weakened case.

\begin{remark}\label{quasiR2double}
    As we have noted in the beginning of this section, the 2-quantum double $\cG = D(\cH)$ by construction has equipped a canonical invertible quasitriangularity structure, which splits into two copies of its (skew) self-pairing form  \eqref{selfpair}. Conversely, if a quasitriangular 2-bialgbera $(\cG,\cR)$ defines a skew self-pairing, then $\cG$ factorizes by \textbf{Theorem \ref{2qd}} and forms a 2-Drinfel'd double. See also the ensuing paragraph for more details.
\end{remark}

\paragraph{The (dual) 2-R-matrix from factorizability.} Due to the factorizability result {\bf Theorem \ref{2qd}}, we could have begun our characterization with a general associative 2-bialgebra $\cK$ which factorizes into two copies of $\cG$, instead of the 2-quantum double $D(\cG,\cG)$. This introduces the braided transposition $\Psi:\cG\otimes\cG \rightarrow \cG\otimes \cG$ given in  \eqref{braid} into the definition of the dual 2-R-matrix:
\begin{equation}
    \cR^*_l = \operatorname{ev}_l\circ\Psi_{-1}^l,\qquad \cR^*_r = \operatorname{ev}_r\circ\Psi_{-1}^r \nonumber
\end{equation}
where $\operatorname{ev} = \operatorname{ev}_{l} + \operatorname{ev}_r$ is precisely the skew-pairing $\langle\cdot,\cdot\rangle_\text{sk}$ that we have introduced previously. 

Dualizing this construction then gives 
\begin{equation}
    \cR^l = \Psi_{-1}^l\circ \operatorname{coev}_l,\qquad \cR^r = \Psi_{-1}^r\circ\operatorname{coev}_r, \label{2rmatreconstruct}
\end{equation}
where $\operatorname{coev}= \operatorname{coev}_l + \operatorname{coev}_r: k\rightarrow \cG\otimes\cG$ is the coevaluation. In other words, we are able to reconstruct the 2-$R$-matrix from the braided transposition $\Psi$ on the 2-quantum double $\cK\cong D(\cG,\cG)$. Indeed,  \eqref{braidcomprel} gives the equivariance  \eqref{2ybequiv}, and the relation  \eqref{2yb1} implies  \eqref{def2R1}.

The odd-degreeness of the self-duality on $\cK\cong D(\cH)$ explains why only $\Psi_{-1}$ appears in the reconstruction of the invertible 2-$R$-matrix in the factoriable case: the degree-0 component $\Psi_0$ dualizes to that in degree-(-2) $\bar\Psi^*$ for the dual $\cK^*\cong \cK$, which has the same $t$-map $T= t\otimes t$. As $\bar \Psi^*$ is determined by $(\Psi^*_{-1})^{l,r}=\Psi_{-1}^{r,l}$ per  \eqref{braidcomprel}, the component $\Psi_0$ is also completely determined by $\Psi_{-1}$.


\section{Weak 2-bialgebras}\label{weak2bialgebras}
We now begin our endeavour to weaken the associativity conditions in the above 2-quantum double construction. The idea of \textit{non-associative} 2-algebra has not been developed nearly as much as their associative counterpart, but we shall take inspiration from their Lie 2-algebra counterparts.

The motivation for this endeavour is twofold:
\begin{enumerate}
    \item Mathematically, it is known \cite{heredia2016representations2groupsbaezcrans2vector} that the 2-representations of associative 2-algebras do not carry non-trivial $k$-invariants --- namely, all coherences such as the associators, unitors, pentagonators etc. are identities.
    \item Physically, it is understood that the $L_\infty$-algebra of semiclassical observables in higher-dimensional field theories \cite{Budzik:2022mpd,costello_gwilliam_2016} will, in general, acquire higher homotopy products upon perturbative quantization.
\end{enumerate}
We will then introduce a \textit{homotopy-refinement} of Baez-Crans 2-vector spaces, forming a 2-category $\mathsf{2Vect}^{hBC}$, by considering \textit{pseudo} $\mathsf{Vect}$-algebras \cite{Fiore2004PseudoLB} in $\mathsf{Cat}$.

\subsection{Definition of weak 2-algebras}

Operationally, through the macrocosm principle \cite{Baez:1995xq}, algebra objects in $\mathsf{2Vect}^{hBC}$ are 2-term \textit{$A_\infty$-algebras} of Stasheff \cite{Stasheff:1963}. We can explicitly describe them by generalizing {\bf Definition \ref{assoc2alg}}.

\begin{definition}\label{weak2alg}
A 2-term $A_\infty$-algebra, or equivaelntly a {\bf weak 2-algebra} $(\cG,\cT)$, is a map $t:\cG_{-1}\rightarrow \cG_0$ between a pair of {\it not necessarily associative} algebras, together with an {invertible homotopy map} $\cT:\cG^{3\otimes}_0\rightarrow \cG_{-1}$ such that we have the conditions  \eqref{algpeif1}, \eqref{algpeif2}, as well as
\begin{enumerate}
    \item the {\bf weak 1-associativity}, 
    \begin{equation}
        (xx')x'' - x(x'x'') = t\cT(x,x',x''),\qquad (yy')y'' - y(y'y'') = \cT(ty,ty',ty'')  \nonumber
    \end{equation}
    and the {\bf weak bimodularity},
    \begin{gather}
        x\cdot (x'\cdot y) - (xx')\cdot y = \cT(x,x',ty)\qquad
        (x\cdot y)\cdot x' - x\cdot (y\cdot x') = \cT(x,ty,x'),\nonumber\\
        (y\cdot x)\cdot x' - y\cdot (xx') = \cT(ty,x,x'),
        \nonumber
    \end{gather}
    for each $x,x',x''\in\cG_0$ and $y,y',y''\in\cG_{-1}$,
    \item the {\bf Hochschild 3-cocycle condition},
    \begin{equation}
        x_1\cdot \cT(x_2,x_3,x_4)+\cT(x_1,x_2,x_3)\cdot x_4 = \cT(x_1x_2,x_3,x_4)- \cT(x_1,x_2x_3,x_4)+\cT(x_1,x_2,x_3x_4)\nonumber
    \end{equation}
    for each $x_1,\dots,x_4\in\cG_0$.
\end{enumerate}
We call $(\cG,\cT)$ a unital weak 2-algebra if we have a unit map $\eta:k\rightarrow \cG$
that satisfies the usual conditions \eqref{def:unity}, and such that $\cT$ is {\bf normalized} --- namely it vanishes whenever any of its arguments are $0$ or $\eta_0$.
\end{definition}
\noindent We note here that this structure is precisely the definition of a 2-term homotopy {\bf $A_\infty$-algebra} \cite{Stasheff:1963}, together with the Peiffer identity constraint \eqref{algpeif2}. The correspondence between the $n$-nary product $m_n\in \operatorname{Hom}^{n-2}(\cG^{n\otimes},\cG)$ and the weak 2-algebra structure is given by $$m_1(-)= t(-),\qquad m_2(-,-) = (--,-\cdot-),\qquad m_3(-,-,-)=\cT(-,-,-),$$ with $m_n=0$ trivial for $n\geq 4$. Nevertheless, we shall see that the Peiffer identity on $\cG$ shall play a very important role. 

Similar to {\it Remark \ref{bimodulecondition}}, the Peiffer identity implies the further constraints\begin{gather}
    (x\cdot y)y' - x\cdot(yy') = \cT(x,ty,ty'),\qquad (y\cdot x)y' - y(x\cdot y') = \cT(ty,x,ty'),\nonumber\\ 
    y(y'\cdot x) - (yy')\cdot x = \cT(ty,ty',x) \nonumber
\end{gather}
for $t\neq 0$, where $x\in\cG_0,y,y'\in\cG_{-1}$. 

\subsubsection{Weak 2-algebra homomorphisms}
We define a map between weak 2-algebras $(\cG,\cT)\rightarrow(\cG',\cT')$ as a cochain map $F=(F_1,F_0,F_{-1}):\cG\rightarrow \cG'$:
\begin{equation}
    F_1: \cG_0^{2\otimes}\rightarrow\cG_{-1}',\qquad F_0:\cG_0\rightarrow\cG_0',\qquad F_{-1}:\cG_{-1} \rightarrow\cG_{-1}' ,\nonumber
\end{equation}
such that $t'\circ F_{-1} = F_0\circ t$ and the following conditions are satisfied,
\begin{eqnarray}
    t'F_1(x,x') &=& F_0(xx') - F_0(x)F_0(x'),\nonumber\\
    F_1(x,ty) &=& F_{-1}(x\cdot y) - F_0(x)\cdot'F_{-1}(y),\nonumber\\
    F_1(ty,x) &=& F_{-1}(y\cdot x) - F_{-1}(y)\cdot'F_0(x)\nonumber\\
    \cT'(F_0(x),F_0(x'),F_0(x'')) &=& F_0(x)\cdot'F_1(x',x'')-F_1(xx',x'') \nonumber\\
     &\qquad&+~ F_1(x,x'x'') - F_1(x,x')\cdot'F_0(x'') \nonumber\\
     &\qquad&+~F_{-1}(\cT(x,x',x'')) .\label{weak2hom}
\end{eqnarray}
In other words, $F_1$ contributes as an "obstruction" for the other components $(F_0,F_{-1})$ to define a strict 2-algebra homomorphism, but only up to homotopy in the sense that $F_1$ by definition (see the last equation of  \eqref{weak2hom}) gives an explicit trivialization of the Hochschild cohomology class $[\cT'\circ F_0] - [F_{-1}\circ\cT]=0$.

It can then be deduced that quasi-isomorphism classes of weak 2-algebras --- where $\cG\sim \cG'$ are said to be quasi-isomorphic iff there exists a weakly inertible cochain map  \eqref{weak2hom} between them --- is still labeled by Hochschild cohomology classes $\cT\in HH^3(\cN,V)$, where $\mathcal{N} = \operatorname{coker}t$ and $V=\operatorname{ker}t$. In particular, $(\cG,\cT)$ is always quasi-isomorphic to its skeleton $(\cN\xrightarrow{0}V,[\cT])$, which is in fact associative.

\subsubsection{Example: weak 2-group algebras} 
Our definition of the weak 2-algebra is less natural in the context of groups, as weakening the associativity in a group $G$ reads
\begin{equation}
    (xx')x'' = \tau(x,x',x'')\cdot x(x'x''),\qquad x,x',x''\in G,\nonumber
\end{equation}
which does not reproduce our above notion of a weak 2-algebra when we pass to the group algebra $kG$. There is hence an inherent disconnect between a natural notion of a "weak 2-group" and that of a weak 2-algebra.

\medskip 

Consider a skeletal 2-group $G$ with Ho{\'a}ng data $(G_0,G_{-1},\tau)$ as a categorical group \cite{Cui_2017}. There is a copy of $G_{-1}$ over each object $x\in G_0$ as the space of endomorphisms on $x$. Notice here that $G_0$ is a genuine group with an associative product, but there are distinguished associator isomorphisms valued in $G_{-1}$,
\begin{equation}
    \tau(x,x',x''):(xx')x''\rightarrow x(x'x''),\qquad x,x',x''\in G_0,\nonumber
\end{equation}
that represents the Postnikov class $\tau\in H^3(G_0,G_{-1})$. The 3-cocycle condition for $\tau$ holds due to the pentagon relation. Note the group $G_{-1}$ is Abelian ($G$ being skeletal) and we use the addition for its product.

\medskip 

We wish to take the same point of view with weak 2-algebras. Suppose $\cG:\cG_{-1}\xrightarrow{0}\cG_0$ is a weak skeletal 2-algebra, then $\cG_0$ is in fact associative, and $\cG_{-1}$ is an associative $\cG_0$-bimodule. The difference with the strict case is that there are now distinguished associator isomorphisms
\begin{equation}
    \cT(x,x',x''): (xx')x''\rightarrow x(x'x''),\label{3cocyassociso}
\end{equation}
which is given by the data of the homotopy map $\cT$. Unfortunately, the construction of a 2-group algebra $kG$ described in  \eqref{2grpalg}, \eqref{2grpalg2} does {\it not} preserve the classifying 3-cocycles. This is because that the $t$-map on $kG$ is the augmentation, and hence $kG$ is classified by $HH^3(kG_0/k,\varepsilon)$ where $\operatorname{im}t \cong k$ while $\varepsilon=\operatorname{ker}t$ is the augmentation ideal.

It is, however, possible to construct a version of the group 2-algebra $kG$ that does give rise to a correspondence
\begin{equation}
    H^3(G_0,G_{-1})\rightarrow HH^3(kG_0,kG_{-1}),\qquad\tau\mapsto \cT,\nonumber
\end{equation}
as we have noted at the end of Section \ref{sec:2gpalgb}. Moreover, for certain skeletal 2-groups, one may even leverage the natural (ring!) isomorphism $HH^\ast(kN,kN)\cong H^\ast(N,kN)$ \cite{Siegel:1999} to produce a bijective correspondence between these classifying 3-cocycles. This fact is used in the accompanying work \cite{Chen2z:2023}.

\subsection{Weak 2-coalgebras}
We begin by defining the notion of a weak 2-coalgebra. Recall that the weakening in {\bf Definition \ref{weak2alg}} concerns only the associativity of the 2-algebra structure. Correspondingly, the weakening of a 2-coalgebra should only concern the coassociativity.

For brevity of notation later, we first rewrite the equations  \eqref{cohgrpd00}, \eqref{cohgrpd10} in a more concise way. Consider coassociativity  \eqref{cohgrpd00}; we naturally extend $\Delta_{-1}$ to act on tensor products (with alternating sign) such that 
\begin{equation}
    \Delta_{-1}\circ\Delta_{-1} \equiv (\id\otimes \Delta_{-1})\circ \Delta_{-1} - (\Delta_{-1}\otimes \id)\circ\Delta_{-1}.\nonumber
\end{equation}
Secondly, we recombine $\Delta_0 = \Delta_0^l + \Delta_0^r$ and extend it as well to tensor products, such that
\begin{eqnarray}
    (\Delta_{-1} + \Delta_0)\circ\Delta_0 &\equiv& \left[(\Delta_{-1}\otimes \id)\circ \Delta_0^l - (\id\otimes\Delta_0^l)\circ\Delta_0^l\right] \nonumber\\
    &\qquad&+~ \left[(\id\otimes\Delta_{-1})\circ\Delta_0^r -(\Delta_0^r\otimes \id)\circ\Delta_0^r\right]\nonumber
\end{eqnarray}
encodes two expressions in  \eqref{cohgrpd10}. We extend the $t$-map to the triple tensor product, $$D_t = \id\otimes \id\otimes t - \id\otimes t\otimes \id + t\otimes \id\otimes \id,$$ such that the equation
\begin{equation}
    D_t\circ\Delta_0\circ\Delta_0 = \Delta_0\circ D_t\circ\Delta_0 \nonumber
\end{equation}
encodes all three equations in  \eqref{cobimodu}. For convenience, we define also the map
\begin{equation}
    D_t[2] \equiv t\otimes t\otimes \id - t\otimes  \id\otimes t + \id\otimes t\otimes t,\nonumber
\end{equation}
which is an extension of two applications of $t$ to the $3$-fold tensor product.

\begin{definition}
Let $\Delta_1:\cG_0\rightarrow \cG_{-1}^{3\otimes}$ denote an {\it invertible} trilinear map. Together with the maps $(\Delta_{-1},\Delta_0)$ defined as in  \eqref{grpdcoprod}, we say that the tuple $(\cG,\Delta=(\Delta_{-1},\Delta_0,\Delta_1))$ is a {\bf weak 2-coalgebra} iff coequivariance  \eqref{cohgrpd+}, coPeiffer identity  \eqref{cohgrpd-}, \textbf{weak coassociativity} 
\begin{eqnarray}
    \Delta_{-1}\circ\Delta_{-1} &=& \Delta_1\circ t,\nonumber\\
    (\Delta_{-1} + \Delta_0)\circ\Delta_0 &=& D_t\circ\Delta_1,\label{weakcohgrpd1}
\end{eqnarray}
and \textbf{2-coassociativity}
\begin{equation}
    \Delta_1\circ\Delta_0 = \Delta_{-1}\circ\Delta_1 \label{weakcohgrpd2}
\end{equation}
are satisfied. In which case we call $\Delta_1$ the {\bf coassociator} of $\cG$. 

We call $(\cG,\Delta)$ \textbf{counital} if it is equipped with a counit $\epsilon:k\rightarrow\cG$ satisfying the usual conditions, and $\epsilon\circ\Delta_1=0$.
\end{definition}

Notice that, provided the coequivariance and the coPeiffer identity are satisfied, applying one more $t$-map to  \eqref{weakcohgrpd1} yields
\begin{equation}
    \Delta_0'\circ\Delta_0- \Delta_0\circ\Delta_0' = D_t[2]\circ \Delta_1,\label{weakcohgrpd3}
\end{equation}
which is a monoidal weakening of the condition  \eqref{cobimodu}. Similarly, applying the $t$-map yet once more gives a map $\Phi\equiv (t\otimes t\otimes t)\Delta_1:\cG_0\rightarrow \cG_0^{3\otimes}$ that lands only in $\cG_0$. We write this element multiplicatively such that
\begin{equation}
    (\Delta_0'\otimes \id)\circ\Delta_0' = \Phi \circ (\id\otimes \Delta_0')\circ\Delta_0'.\label{weakcohgrpd4}
\end{equation}
Recall that, in the skeletal case where $t=0$, the coproducts $\Delta_{-1},\Delta_0,\Delta_0'$ are independent and hence  \eqref{weakcohgrpd4} should also be imposed independently from  \eqref{weakcohgrpd1}.

\subsection{Weak 2-bialgebras}
Suppose now $(\cG,\cT)$ is a weak 2-algebra equipped with the tuple $\Delta=(\Delta_{-1},\Delta_0,\Delta_1)$ of linear maps. Recall the Sweedler notation  \eqref{deg0sweed} for $\Delta_0':\cG_0\rightarrow\cG_0^{2\otimes}$. We use it to state the condition that the coassociator $\Delta_1$ preserves the algebra structure on $\cG$,
\begin{eqnarray}
    (\Delta_{-1}\circ\cT)(x,x',x'') &=& \cT(\bar x_{(1)},\bar x_{(1)}',\bar x_{(1)}'') \otimes \cT(\bar x_{(2)},\bar x_{(2)}',\bar x_{(2)}''),\nonumber\\
    \Delta_1(xx') &=& x_{(1)}x_{(1)}'\otimes x_{(2)}x_{(2)}'\otimes x_{(3)}x_{(3)}',\label{coassocalg}
\end{eqnarray}
for $x,x',x''\in\cG_0$. Note that $\bar x_{(1)},\bar x_{(2)}\in\cG_0$ are not to be confused with the elements $x_{(1)}^{l,r}$ in  \eqref{sweed}.
 
\begin{definition}\label{wk2bialg}
The tuple $(\cG,\cT,\Delta)$ is a (unital) {\bf weak 2-bialgebra} iff $(\cG,\cT)$ is a weak 2-algebra and $(\cG,\Delta)$ is a (counital) weak 2-coalgebra. Equivalently, $(\cG,\cT,\Delta)$ is a weak 2-bialgebra iff the tuple $\Delta=(\Delta_1,\Delta_0,\Delta_{-1})$ satisfies  \eqref{cohgrpd-}, \eqref{cohgrpd+}, \eqref{weakcohgrpd1}-\eqref{weakcohgrpd3}, \eqref{2algcoprod} and \eqref{coassocalg}.

A weak 2-bialgebra $(\cG,\cT,\Delta)$ is called {\bf quasi-2-bialgebra} if $\cT=0$.
\end{definition}

Similar to what we have done for the strict case, we suppose $\cG$ is dually paired with its dual 2-algebra through  \eqref{pairing}. The coassociator $\Delta_1$ on $\cG$ induces a linear map $\cT^*:\cG_{-1}^*\rightarrow \cG_0^*$ by
\begin{equation}
    \langle f\otimes f'\otimes f'',\Delta_1(x)\rangle = \langle \cT^*(f,f',f''),x\rangle. \nonumber
\end{equation}
Similarly, the Hochschild 3-cocycle $\cT$ on $\cG$ induces a linear map $\Delta_1^*:\cG_{-1}^*\rightarrow (\cG_0^*)^{3\otimes}$. We form the tuple $\Delta^*=(\Delta_1^*,\Delta_0^*,\Delta_{-1}^*)$.

\begin{proposition}\label{weak2bialg}
Let $\cG,\cG^*$ be dually paired, then $(\cG,\cT,\Delta)$ is a (unital) weak 2-bialgebra iff $(\cG^*,\cT^*,\Delta^*)$ is a (unital) weak 2-bialgebra.
\end{proposition}
\begin{proof}
Let $(\cG^*,\cT^*,\Delta^*)$ be the dual weak 2-bialgebra. By a straightforward computation, it is obvious that the conditions  \eqref{cohgrpd+}, \eqref{cohgrpd-}, \eqref{weakcohgrpd1}, \eqref{weakcohgrpd3} on $(\cG^*,\Delta^*)$ hold iff the Peiffer conditions, weak 1-associativity and weak bimodularity hold for the 2-algebra $\cG$. The fact that $(\Delta_{-1}^*,\Delta_0^*)$ are 2-algebra maps, as well as the units/counits, are treated the same way as in the proof of {\bf Proposition \ref{2bialg}}.

The two non-trivial identities to check are  \eqref{weakcohgrpd2}, \eqref{coassocalg}. We will use the Sweedler notation for $\Delta_1$, $\Delta_1 x= x_{(1)}\otimes x_{(2)}\otimes x_{(3)}$. Note $x\in\cG_0$, but $x_{(i)}\in \cG_{-1}$ for $i=1,2,3$.

We begin first with the latter, and check that $\Delta_1$ is a 2-algebra map. We have
\begin{equation}
    \langle \Delta_1^*(ff'),x\otimes x'\otimes x''\rangle = \langle ff',\cT(x,x',x'')\rangle = \langle f\otimes f',\Delta_{-1}\cT(x,x',x'')\rangle,\nonumber
\end{equation}
while (recall $\Delta_0'$ dualizes to the multiplication on $\cG_0^*$, and $ f_{(i)}\in \cG^*_{-1}$)
\begin{eqnarray}
    \langle  f_{(1)} f_{(1)}'\otimes f_{(2)} f_{(2)}'\otimes  f_{(3)} f_{(3)}', x\otimes x'\otimes x''\rangle &=& \langle \bigotimes_{i=1}^3 f_{(i)}\otimes  f_{(i)}',\Delta_0'(x)\otimes\Delta_0'(x')\otimes\Delta_0'(x'')\rangle \nonumber\\
    &=&\langle ( f_{(1)}\otimes  f_{(1)}')\otimes ( f_{(2)}\otimes f_{(2)}')\otimes ( f_{(3)}\otimes  f_{(3)}'), \nonumber\\
    &\qquad&\quad (\bar x_{(1)}\otimes \bar x_{(2)}) \otimes (\bar x'_{(1)}\otimes \bar x'_{(2)})\otimes (\bar x''_{(1)}\otimes \bar x''_{(2)})\rangle \nonumber\\
    &=& \langle \Delta_1^*(f) \otimes \Delta_1^*(f'), \nonumber\\
    &\qquad& \quad(\bar x_{(1)}\otimes \bar x_{(1)}'\otimes \bar x_{(1)}'')\otimes (\bar x_{(2)}\otimes \bar x_{(2)}'\otimes \bar x_{(2)}'')\rangle \nonumber\\
    &=& \langle f\otimes f',\cT(\bar x_{(1)},\bar x_{(1)}',\bar x_{(1)}'')\otimes \cT(\bar x_{(2)},\bar x_{(2)}',\bar x_{(2)}'')\rangle,\nonumber
\end{eqnarray}
meaning that $\Delta_1$ is a 2-algebra map on $(\cG,\cT)$ iff $\Delta_1^*$ is also one on $(\cG^*,\cT^*)$.

We now check  \eqref{weakcohgrpd3}. Fix arbitrary elements $f_1,\dots,f_4\in\cG_{-1}^*$, we compute
\begin{eqnarray}
    \langle f_1\cdot\cT^*(f_2,f_3,f_4),x\rangle &=& \langle f_1\otimes \cT^*(f_2,f_3,f_4),\Delta_0^l(x)\rangle = \langle f_1\otimes\dots\otimes f_4,(1\otimes \Delta_1)\circ\Delta_0^l(x)\rangle,\nonumber\\
    \langle \cT^*(f_1,f_2,f_3)\cdot f_4,x\rangle &=& \langle \cT^*(f_1,f_2,f_3)\otimes f_4,\Delta_0^r(x)\rangle = \langle f_1\otimes\dots\otimes f_4,(\Delta_1\otimes 1)\circ\Delta_0^r(x)\rangle;\nonumber
\end{eqnarray}
on the other hand, we have 
\begin{eqnarray}
    \langle \cT^*(f_1f_2,f_3,f_4),x\rangle &=& \langle f_1f_2\otimes f_3\otimes f_4,\Delta_1(x)\rangle = \langle f_1\otimes\dots\otimes f_4,(\Delta_{-1}\otimes 1\otimes 1)\circ\Delta_1(x)\rangle,\nonumber\\
    \langle \cT^*(f_1,f_2,f_3f_4),x\rangle &=& \langle f_1\otimes f_2\otimes f_3 f_4,\Delta_1(x)\rangle = \langle f_1\otimes\dots\otimes f_4,(1\otimes 1\otimes \Delta_{-1})\circ\Delta_1(x)\rangle,\nonumber\\
    \langle \cT^*(f_1,f_2f_3,f_4),x\rangle &=& \langle f_1\otimes f_2f_3\otimes f_4,\Delta_1(x)\rangle = \langle f_1\otimes\dots\otimes f_4,(1\otimes \Delta_{-1}\otimes 1)\circ\Delta_1(x)\rangle.\nonumber
\end{eqnarray}
By summing these up, the Hochschild 3-cocycle condition for $\cT^*$ is equivalent to
\begin{equation}
    (1\otimes \Delta_1)\circ\Delta_0^l + (\Delta_1\otimes 1)\circ\Delta_0^r=\Delta_1\circ\Delta_0=\Delta_{-1}\circ\Delta_1.\nonumber
\end{equation}
which is exactly  \eqref{weakcohgrpd3}.
\end{proof}
\noindent Given $(\cG,\cG^*)$ are dually paired 2-bialgebras, we see that a quasi-2-bialgebra $(\cG,\cT=0,\Delta)$ encode the same data as a weak but coassociative 2-bialgebra $(\cG^*,\cT^*,\Delta^*)$, in which $\Delta_1^*=0$.

\section{Weak (skeletal) 2-quantum doubles}\label{weakskeletalqd}

Let $\cG,\cG^*$ be dually paired (weak) 2-bialgebras. To form its weak 2-quantum double, we require them to act on each other \textit{weakly}. This means, in particular, that the coadjoint actions $\bar\rhd,\bar\lhd$ now come with the additional components 
\begin{equation}
    \rhd_1: \cG_0^{2\otimes}\rightarrow \operatorname{Hom}(\cG_{-1}^*,\cG_0^*),\qquad \lhd_1: (\cG_{-1}^*)^{2\otimes}\rightarrow \operatorname{Hom}(\cG_0,\cG_{-1}).\nonumber
\end{equation}
This will be justified further in Section \ref{wk2rep} where we show that the coadjoint action can be interpreted weak representation. More specifically, just like the product and actions in \eqref{coadj0} contribute to defining dually some (crossed) relations, the cocycle $\cT$ should also contribute dually to the adjoint action. This is what $\rhd_1$ and $\lhd_1$ stand for, as we will see in   \eqref{assoc}. 

To construct non-skeletal weak 2-quantum doubles, one must explicitly keep track of how $\cT,\cT^*,\rhd_1,\lhd_1$ appear in the crossed-relations  \eqref{comprel1}, \eqref{comprel2}, \eqref{crossrel}. This is a Herculean task that we leave to the ambitious reader. We will restrict from now on to the skeletal case when defining the quantum double.

\subsection{Matched pair of skeletal weak  2-bialgebras}
Though the situation is drastically simplified in the skeletal case $t=0$, it is now important for us to keep track of the associators. We shall do this by using the notation of  \eqref{3cocyassociso}.

The non-trivial crossed relations  \eqref{skelcomprel}, in particular, are attached with the components $\rhd_1,\lhd_1$ of the coadjoint actions,
\begin{eqnarray}
    (x)\lhd_1^{f,f'}&:& x\rhd_{-1}(ff') \xrightarrow{\sim} (x_{(1)}\bar\rhd f_{(1)})\cdot^* ((x_{(2)}\bar\lhd f_{(2)})\bar\rhd f'),\nonumber\\
    \rhd_1^{x,x'}(f)&:&(xx') \lhd_{-1} f \xrightarrow{\sim} (x\bar\lhd (x_{(1)}'\bar\rhd f_{(1)}))\cdot (x_{(2)}\bar\lhd f_{(2)}),\nonumber
\end{eqnarray}
where we have made use of the shorthand notation defined in {\it Remark \ref{skelqd}}. These come together to allow us to define a Hochschild 3-cochain on the 2-quantum double $D(\cG)$,
\begin{equation}
    \cT_D:D(\cG)_0^{3\otimes}\rightarrow D(\cG)_{-1},\qquad \cT_D(w,w',w'') = \begin{cases}\cT(x,x',x'') \\ \rhd_1^{x,x'}(f'')  \\ (x)\lhd_1^{f',f''}  \\ \cT^*(f,f',f'') \end{cases},\label{assoc} 
\end{equation}
where $w=(x,f) \in D(\cG)_0$ is a degree-0 element, with $x\in\cG_0$ and $f\in\cG_{-1}^*$.

\begin{definition}
The pair $(\cG,\cG^*)$ of mutually paired weak skeletal 2-bialgebras forms a {\bf (skeletal) matched pair} iff, in addition to the compatibility conditions  \eqref{comprel1}-\eqref{grpdunit}, the 3-cochain $\cT_D$ defined in  \eqref{assoc} is a Hochschild 3-cocycle on $D(\cG) \cong \cG\otimes \cG^*$. 
\end{definition}
For arguments contained solely in $\cG_0$ or $\cG_{-1}^*$, this condition merely states the 3-cocycle conditions for $\cT,\cT^*$, respectively. The other ones mix non-trivially the different components of the 3-cocycle $\cT_D$,
\begin{align}
    x_1\rhd_0 (\rhd_1^{x_2,x_3}(f)) - \cT(x_1,x_2,x_3)\lhd_0 f &= \rhd_1^{x_1x_2,x_3}(f) - \rhd_1^{x_1,x_2x_3}(f) + \cT(x_1,x_2,x_3\lhd_{-1}f),\nonumber\\
    x_1 \cdot (x_2)\lhd_1^{f_1,f_2} - \rhd_1^{x_1,x_2}(f_1)\cdot^* f_2 &= (x_1x_2)\lhd_1^{f_1,f_2} - (x_1)\lhd_1^{x_2\rhd_{-1}f_1,f_2} + \rhd_1^{x_1,x_2}(f_1f_2),\nonumber\\
    x \rhd_0 \cT^*(f_1,f_2,f_3) - ((x)\lhd_1^{f_1,f_2})\lhd_0 f_3 &= \cT^*(x\rhd_{-1}f_1,f_2,f_3) - (x)\lhd_1^{f_1f_2,f_3} + (x)\lhd_1^{f_1,f_2f_3}.\label{2dd3cocyc}  
\end{align}
Then, we construct $D(\cG)$ as a 2-bialgebra as in Section \ref{str2qd}. 

Since we are in the skeletal case, it is easy to see from  \eqref{weak2hom} that the 2-quantum double is \textit{weakly} self-dual $D(\cG)\sim D(\cG)^*$, where we recall $\sim$ denotes equivalence of 2-algebras under the classification result {\bf Theorem \ref{2algclass}}. This means that the associated Hochschild 3-cocycles $\cT_D,\cT^*_D$ are cohomologous, where
\begin{equation}
    \cT^*_D:D(\cG)_0^{3\otimes}\rightarrow D(\cG)_{-1},\qquad {\cT}^*_D(w,w',w'') = \begin{cases}\mathring{\cT}(f,f',f'') \\ \mathring{\rhd}_1^{f,f'}(x'')  \\ (f)\mathring{\lhd}_1^{x',x''}  \\ \mathring{\cT}^*(x,x',x'') \end{cases},\nonumber
\end{equation}
denotes the dual of the 3-cocycle $\cT_D$. The "dual" version of  \eqref{2dd3cocyc} reads
\begin{align}
    f_1\lhd_0 (\mathring{\rhd}_1^{f_2,f_3}(x)) - \mathring{\cT}(f_1,f_2,f_3)\rhd_0 x &= \mathring{\rhd}_1^{f_1f_2,f_3}(x) - \mathring{\rhd}_1^{f_1,f_2f_3}(x) + \mathring{\cT}(f_1,f_2,f_3\lhd_{-1}x),\nonumber\\
    f_1 \cdot^* (f_2)\mathring{\lhd}_1^{x_1,x_2} - \mathring{\rhd}_1^{f_1,f_2}(x_1)\cdot x_2 &= (f_1f_2)\mathring{\lhd}_1^{x_1,x_2} - (f_1)\mathring{\lhd}_1^{f_2\lhd_{-1}x_1,x_2} + \mathring{\rhd}_1^{f_1,f_2}(x_1x_2),\nonumber\\
    f \lhd_0 \mathring{\cT}^*(x_1,x_2,x_3) - ((f)\mathring{\lhd}_1^{x_1,x_2})\rhd_0 x_3 &= \mathring{\cT}^*(f\rhd_{-1}x_1,x_2,x_3) - (f)\mathring{\lhd}_1^{x_1x_2,x_3} + (f)\mathring{\lhd}_1^{x_1,x_2x_3}.\label{2dd3cocycdual}
\end{align}
It is important to note that the components $\rhd_1,\lhd_1$ do {\it not} form Hochschild 3-cocycles by themselves, and similarly for the components $\mathring{\rhd}_1,\mathring{\lhd}_1$. 

\subsection{Factorizability of weak 2-bialgebras}
We now prove the analogue of {\bf Theorem \ref{2qd}}.
\begin{theorem}\label{weak2dd}
    Suppose $(\cK,\hat\cdot,\cT_K)$ is a weak 2-bialgebra that weakly factors into two skeletal weak sub-2-bialgebras $\cG,\cH$, namely the inclusions in the span  \eqref{span} are weak homomorphisms as defined in  \eqref{weak2hom}, then $\cK\sim D(\cG)$ are equivalent as 2-bialgebras. 
\end{theorem}
\noindent Recall two weak 2-bialgebras are equivalent when there exists an invertible weak 2-homomorphism  \eqref{weak2hom} between them.
\begin{proof}

    The fact that $\cK$ factors into skeletal 2-subalgebras means that it must also be skeletal itself. This allows us to leverage the proof of {\bf Theorem \ref{2qd}} to reconstruct the underlying 2-bialgebra structure of $\cK\cong D(\cG)$ as a 2-quantum double. 
    
    The subtlety here is that we must now keep track of the 3-cocycle $\cT_K: \cK_0^{3\otimes}\rightarrow \cK_{-1}$ in $\cK$ when we, in particular, invoke associativity in the form
    \begin{eqnarray}
        \cT_K(\iota_0(x),\iota_0(x'),\jmath_0(f))&\equiv& \rhd_1^{x,x'}(f) : (\iota_0(x)\hat\cdot\iota_0(x'))\hat\cdot\jmath_0(f)\xrightarrow{\sim} \iota_0(x)\hat\cdot(\iota_0(x')\hat\cdot\jmath_0(f)),\nonumber\\
        \cT_K(\iota_0(x),\jmath_0(f),\jmath_0(f'))&\equiv& (x)\lhd_1^{f,f'} :(\iota_0(x)\hat\cdot\jmath_0(f))\hat\cdot \jmath_0(f')\xrightarrow{\sim} \iota_0(x)\hat\cdot(\jmath_0(f)\hat\cdot \jmath_0(f')).\nonumber
    \end{eqnarray}
    Now in the skeletal case, the braiding map $\Psi = (\Psi_0,\Psi_{-1};\bar\Psi): \cG\otimes \cH\rightarrow \cH\otimes \cG$ is still defined as in  \eqref{braid}. However, the components $\rhd_1,\lhd_1$ now give rise to associators
    \begin{eqnarray}
        \rhd_1&:&\Psi \circ(\hat\cdot \otimes \operatorname{id}) \xRightarrow{\sim} (\operatorname{id}\otimes \hat\cdot)\circ \Psi_{12}\circ \Psi_{23},\nonumber\\
        \lhd_1&:&\Psi\circ(\operatorname{id}\otimes \hat\cdot) \xRightarrow{\sim} (\hat\cdot\otimes \operatorname{id})\circ \Psi_{23}\circ \Psi_{12}\label{2yb2}
    \end{eqnarray}
    that implement the braiding relations  \eqref{2yb1}. These braiding associators satisfy a set of algebraic conditions following from the 3-cocycle condition  \eqref{2dd3cocyc} for $\cT_K$.
    
    With the components $\rhd_1,\lhd_1$ as defined above, we now wish to reconstruct the Hochschild 3-cocycles $\cT_G,\cT_H$ of $\cG,\cH$ from $\cT_K$. Note this cannot be achieved by just restricting $\cT_K$ via the span  \eqref{span}, as this does not have the desired codomains. For instance, the restriction $\cT_K|_{\operatorname{im}\iota=\cG}: G_0^{3\otimes}\rightarrow \cK_{-1} \cong \cG_{-1}\otimes \cH_{-1}$ in general lands in the tensor product, for which only the $\cG_{-1}$-valued component gives the desired 3-cocycle $\cT_G$ on $\cG$. Nevertheless, with $\cT_G,\cT_H$ defined in this way, having the span  \eqref{span} means that the 3-cocycle condition for $\cT_K$ implies $(\cG,\cT_G),(\cH,\cT_H)$ form a matched pair of weak 2-bialgebras, as in  \eqref{2dd3cocyc}. 
    
    The "undesirable" piece $\tilde{\cT}_G$, namely the component of $\cT_K|_{\cG}$ valued in $\cH_{-1}$, is a Hochschild coboundary. This follows from the definition of the inclusion $\iota=(\iota_{-1},\iota_0,\iota_1):\cG\hookrightarrow \cK$ as a weak homomorphism. Indeed, by projecting the last of  \eqref{weak2hom} for $\iota_1$ to $\cH$, the first term $\iota_{-1}(\cT_\cG(x,x',x''))|_\cH=0$ vanishes whence
    \begin{eqnarray}
     \tilde{\cT}_G&\equiv &\cT_K(\iota_0(x),\iota_0(x'),\iota_0(x''))|_\cH \nonumber\\
     &=&
        \iota_0(x)\, \hat\cdot\, \iota_1(x',x'')|_\cH -~\iota_1(xx',x'')|_\cH+\iota_1(x,x'x'')|_\cH - \iota_1(x,x')|_\cH\hat\cdot \iota_0(x'')\nonumber\\
        &=& d_{HH}[\iota_1|_\cH](x,x',x''),\nonumber
    \end{eqnarray}
    where $d_{HH}$ is the Hochschild differential \cite{Wagemann+2021}.
    Similar arguments show that $\tilde{\cT}_H=d_{HH}[\jmath_1|_\cG]$ is a Hochschild coboundary as well. This establishes the weak equivalence $\cK\sim D(\cG)$.

    The same argument as above, but dualized, is applied to reconstruct $(\Delta_G)_1$ and $(\Delta_H)_1$ from the coassociator $(\Delta_K)_1$. The coassociator conditions  \eqref{weakcohgrpd1}-\eqref{weakcohgrpd3}, as well as  \eqref{coassocalg}, for them follow from those for $(\Delta_K)_1$. 
\end{proof}

Note the coadjoint actions $\rhd,\lhd$ only define genuine algebra representations when $\cT,\cT^*=0$ (as in {\bf Theorem \ref{2qd}}), or when $t,t^*=0$. Without skeletality, the braiding transposition $\Psi$ is no longer of the form given in  \eqref{braid}. Terms like $\rhd_1^{t\cdot,\cdot},\lhd_1^{t^*\cdot,\cdot}$ must now appear. This, of course, would modify  \eqref{2yb1} in a complicated and intricate manner. 

\begin{remark}
If the components $\iota_1,\jmath_1$ are not required as part of the data for the inclusions $\iota,\jmath$ in the span  \eqref{span}, then $\cK\not\sim D(\cG)$ in general. In particular, without the component $\iota_1$ trivializing $\tilde{\cT}_G$ by  \eqref{weak2hom}, its (possibly non-trivial) Hochschild class $[\tilde{\cT}_G]\in HH^3(\cK_0,\cK_{-1})$ is in fact an {\it extra} piece of data in $\cK$ that is not in $D(\cG)$, despite them sharing the same 2-bialgebra structure. Such a factorizable weak 2-bialgebra is still weakly self-dual $\cK\sim\cK^*$.
\end{remark}


In the following,  we shall shift gears a bit and study the 2-representation theory of quasitriangular 2-bialgebras, and finally work our way towards proving the main theorem.

\section{The monoidal 2-category of 2-representations}\label{2representations}
With the above algebraic machinery in place, we are now ready to discuss the 2-representations of a strict or weak 2-bialgebra $\cG$. In the following, we shall follow the Baez-Crans definition of a 2-vector space and the monoidal 2-category $\mathsf{2Vect}^{BC}$ they form \cite{Baez:2003fs,Baez:2010ya}.

Recall that a Baez-Crans 2-vector space is equivalent to a 2-term cochain complex of vector spaces from \textbf{Proposition \ref{2ch}}. Also equivalently, a 2-vector space is a {\it nuclear} 2-algebra \cite{Wagemann+2021}, or an Abelian Lie 2-algebra \cite{Bai_2013,Chen:2012gz}. 

2-vector spaces of this type form a 2-category $\mathsf{2Vect}^{BC}$ in which the 1-morphisms are cochain maps and 2-morphisms are cochain homotopies. Concretely, let $V=V_{-1}\xrightarrow{\partial}V_0,W=W_{-1}\xrightarrow{\partial'} W_0$ denote two 2-vector spaces. A cochain map $f:V\rightarrow W$ is a collection linear maps $f_{0,-1}:V_{0,-1}\rightarrow W_{0,-1}$ such that
\begin{equation}
    \partial'  f_{-1} = f_0 \partial.\nonumber
\end{equation}
Given two such cochain maps $f,g$, a cochain homotopy $q:f\Rightarrow g$ is a linear map $q:V_0\rightarrow W_{-1}$ such that 
\begin{equation}
    \partial q = f_0 - g_0,\qquad q\partial = f_{-1} - g_{-1}.\nonumber 
\end{equation}
We shall refine these notions to fit the definition of a 2-representation of $\cG$ in the following.

\subsection{Weak 2-representations}\label{wk2rep}
Recall that a representation of an ordinary algebra $A$ on the vector space $V$ is an algebra homomorphism $A\rightarrow \operatorname{End}(V)$. Morally, a 2-representation should therefore be a 2-algebra homomorphism between a 2-algebra $\cG$ and a "categorified" notion of the endomorphism algebra $\operatorname{End}(V)$. Correspondingly, a weak 2-representation should be a weak 2-homomorphism as in  \eqref{weak2hom} into a "weak endomorphism 2-algebra".

\subsubsection{Endomorphism 2-algebra on a 2-vector space}
In the strict case, the endomorphisms of a 2-vector space are naturally given in the setting of $\mathsf{2Vect}^{BC}$ --- namely $\operatorname{End}(V) = \operatorname{End}_{\mathsf{2Vect}^{BC}}(V)$, which forms an associative 2-algebra $\operatorname{End}(V)=\operatorname{End}(V)_{-1}\xrightarrow{\delta}\operatorname{End}(V)_0$ of linear transformations on a 2-term cochain complex $V$ \cite{Angulo:2018},
\begin{eqnarray}
     \operatorname{End}(V)_0 &=& \{(M,N)\in\operatorname{End}(V_{-1})\times \operatorname{End}(V_0) \mid \partial M = N\partial\},\nonumber\\
     \operatorname{End}(V)_{-1} &=&\{A\in\operatorname{Hom}(V_0,V_{-1})\mid (A\partial,\partial A) \in \operatorname{End}(V_{-1})\times\operatorname{End}(V_0)\},\nonumber
\end{eqnarray}
equipped with the 2-algebra structure (take $A\in\operatorname{End}(V)_{-1}, \,\, (M,N)\in \operatorname{End}(V)_{0}$)
\begin{equation}
     \delta: A\mapsto (A\partial,\partial A),\qquad (M,N)\cdot A = MA,\qquad A\cdot(M,N)=AN.\nonumber 
\end{equation}
The associativity of matrix multiplication implies that $\operatorname{End}(V)_{-1}$ is clearly a $\operatorname{End}(V)_0$-bimodule, Moreover, we have the Peiffer conditions (note $A,A'\in\operatorname{End}(V)_{-1}$)
\begin{eqnarray}
     \delta((M,N)\cdot A) &=& (MA\partial,\partial MA) = (MA\partial,N\partial A)= (M,N)\delta(A),\nonumber\\ 
     \delta(A\cdot(M,N)) &=& (AN\partial,\partial AN) = (A\partial M,\partial AN) = \delta(A)(M,N),\nonumber\\
     A\ast A'&\equiv& \delta(A)\cdot A' = A\partial A'= A\cdot\delta(A'),\nonumber
\end{eqnarray}
and hence $\operatorname{End}(V)$ is an associative 2-algebra. Note that none of the matrices here are required to be invertible.

\medskip 

As weak 2-algebras are no longer associative, the above presentation of $\operatorname{End}(V)$ in terms of matrices is no longer sufficient: we require a weaker version of $\operatorname{End}(V)$. Such a notion of the weak endomorphism 2-algebra $\End(V)$ would still have the same graded structure $\delta:\End(V)_{-1}\rightarrow\End(V)_0$ as in the strict case above, but its algebra structure should have its associativity controlled by a Hochschild 3-cocycle $\T$, in accordance with {\bf Definition \ref{weak2alg}}.

\medskip 

To begin, we extend the idea of \cite{Schafer:1955} to weak 2-algebras. In essence, we leverage the observation in the strict case that an algebra 2-homomorphism $\cG\rightarrow \End(V)$ is equivalent to a $\cG$-bimodule structure on $V$. We are going to provide a weak generalization of such a $\cG$-bimodule structure in \textbf{Definition} \ref{weakbimodule}.

Let $\text{2Alg}$ denote the category of weak 2-algebras $(\cG,\cT)$, which contains the full subcategory $\text{2Alg}_\text{ass}$ of {\it strict} 2-algebras. A 2-vector space $V\in\mathsf{2Vect}\subset \text{2Alg}_\text{ass}\subset\text{2Alg}$ fits as a strict 2-algebra with trivial multiplication. We consider $\cG$ as a weak 2-algebra (as defined in {\bf Definition \ref{weak2alg}}).  We then equip the direct sum $\cG\oplus V$ with a semidirect product structure,
\begin{eqnarray}
    (z + u) \cdot (z'+u') &=& yy' + x\cdot_l y' + y\cdot_r x' + xx'\nonumber\\
    &\qquad&+~ x\rhd w' + x\rhd v' + y\Yright w'+y\Yright v'\nonumber\\
    &\qquad& +~ w\lhd x' + v\lhd x' + w\Yleft y' + v\Yleft y',\nonumber
\end{eqnarray}
where we have used the shorthand notation $z= (y,x)\in\cG_{-1}\times \cG_0=\cG,\, u=(w,v)\in V_{-1}\times V_0= V$ and where 
\begin{eqnarray}
    \cdot_{l}: \cG_0\otimes \cG_{-1}\rightarrow\cG_{-1}, \qquad   \cdot_{r}: \cG_{-1}\otimes \cG_{0}\rightarrow\cG_{-1}, \nonumber\\
    \rhd: \cG_0\otimes V\rightarrow V, \qquad \lhd:V \otimes  \cG_0 \rightarrow V, \nonumber\\ 
   \Yright: \cG_{-1}\otimes V\rightarrow V, \qquad  \Yleft:V \otimes  \cG_{-1} \rightarrow V \nonumber
\end{eqnarray}
are all bilinear maps. 
\smallskip 

\begin{definition}\label{weakbimodule}
We say that $V$ is a \textbf{$\cG$-bimodule} if $(\cG\oplus V,\cdot) \in \text{2Alg}$ is 
a weak 2-algebra. In other words,
\begin{itemize}
    \item[(i)] $(\cG\oplus V)_{-1} \equiv \cG_{-1}\oplus V_{-1}$ is a weak $(\cG\oplus V)_0:=\cG_0\oplus V_0$-bimodule, 
\item[(ii)] the map $t\oplus \partial:\cG_{-1}\oplus V_{-1}\rightarrow \cG_0\oplus V_0$ is equivariant with respect to $\cdot$ and satisfies the Peiffer identity\footnote{The Peiffer identity states $y\Yright w = (ty)\rhd w = y\Yright(\partial w)$, and similarly $w\Yleft y = (\partial w)\Yleft y= w\lhd(ty)$. If we write $y\Yright v = \Upsilon_y v$, then we reproduce precisely the 2-representation properties  \eqref{genpeif}.}, 
\item[(iii)] there exists a well-defined trilinear invertible map $(\cG_0\oplus V_0)^{3\otimes}\rightarrow \cG_{-1}\oplus V_{-1}$ that satisfies the Hochschild 3-cocycle condition.
\end{itemize}
\end{definition}

By the macrocosm principle \cite{Baez:1995xq}, this puts us in the context of a {\it homotopy refinement} $\mathsf{2Vect}^{hBC}$ of the Baez-Crans 2-vector spaces. In this linear 2-category, each endomorphism category $=\End(V)$ can be thought of as the weak 2-algebra $\End(V)$, and a $\cG$-bimodule structure on $V\in\mathsf{2Vect}^{hBC}$ is equivalent to a 2-homomorphism $\cG\rightarrow \operatorname{End}_{\mathsf{2Vect}^{hBC}}(V) =\End(V)$. We call $\End(V)$ the {\bf weak endomorphism 2-algebra} on $V$, and denote by $\T:\End(V)_0^{3\otimes}\rightarrow \End(V)_{-1}$ its Hochschild 3-cocycle obtained from the third point of {\bf Definition} \ref{weakbimodule}. This motivates our following theory of weak 2-representations.

\begin{remark}\label{weak2end}
We emphasize here that the 2-category $\mathsf{2Vect}^{BC}$ of Baez-Crans 2-vector spaces is completely strict \cite{Baez:2003fs}, and hence its algebra objects (ie. associative 2-algebras/algebra crossed-modules) and its endomorphism categories $\operatorname{End}(V) = \operatorname{End}_{\mathsf{2Vect}^{BC}}(V)$ do not carry homotopy data. Weak 2-algebras/2-term $A_\infty$-algebras are therefore not part of the theory of the usual Baez-Crans 2-vector spaces. Instead, they constitute algebra objects in the {homotopy refinement} $\mathsf{2Vect}^{hBC}$ of $\mathsf{2Vect}^{BC}$. The difference between the setting $\mathsf{2Vect}^{hBC}$ and the Kapranov-Voevodsky setting $\mathsf{2Vect}^{KV}$ is currently under investigation by one of the authors. In Appendix \ref{weak2repthy}, we shall establish properties of weak 2-representations
and relate it to those in the literature \cite{Delcamp:2021szr,Delcamp:2023kew,Douglas:2018,Baez:2012}.
\end{remark}

\subsubsection{Weak 2-representations, weak 2-intertwiners and modifications}

\begin{definition}\label{weak2rep}
A {\bf weak 2-representation} $(\varrho,\rho):\cG\rightarrow \End(V)$ of $\cG$ on $V$ is a homomorphism between weak 2-algebras as in  \eqref{weak2hom}. In other words, $\rho=(\rho_0,\rho_1)$ is a chain map
\begin{equation}
    \begin{tikzcd}
\cG_{-1} \arrow[d, "\rho_1"] \arrow[r, "t"]     & \cG_0 \arrow[d, "\rho_0"]  \\
\End(V)_{-1} \arrow[r, "\delta"] & \End(V)_{0}
\end{tikzcd}\label{commsq}
\end{equation}
which preserves the 2-algebra structures {\it up to homotopy},
\begin{eqnarray}
    \delta\varrho(x,x') &=& \rho_0(xx') - \rho_0(x)\rho_0(x'),\nonumber\\
    \varrho(x,ty) &=& \rho_1(x\cdot y) - \rho_0(x)\cdot \rho_1(y),\nonumber\\
    \varrho(ty,x) &=& \rho_1(y\cdot x) - \rho_1(y)\cdot \rho_0(x),\label{weak2repalg}
\end{eqnarray}
and for which the Hochschild 3-cocycles $\cT,\T$ of respectively $\cG$ and $ \End(V)$  satisfy the following compatibility conditions
\begin{eqnarray}
    \rho_1(\cT(x,x',x'')) &=& \rho_0(x)\cdot \varrho(x',x'')- \varrho(xx',x'') \nonumber\\
     &\qquad&+~ \varrho(x,x'x'') - \varrho(x,x')\cdot \rho_0(x'')\nonumber\\
     &\qquad& +~\T(\rho_0(x),\rho_0(x'),\rho_0(x'')),\label{weak2rephom}
\end{eqnarray}
where $x,x',x''\in\cG_0$ and $y\in\cG_{-1}$. We require $\varrho$ to be invertible.

We call $\rho$ a {\bf strict 2-representation}  if $\varrho=0$ identically.
\end{definition}
\noindent As $\cT,\T$ are normalized, $\varrho$ by definition vanishes if any of its arguments are $0$ or the unit $\eta_0\in\cG_0$.

\begin{remark}\label{weak2repstrict}
Due to the classification {\bf Theorem \ref{2algclass}} of 2-algebras \cite{Wagemann+2021}, a non-trivial 2-algebra $\cG$ with $\cT\neq 0$ cannot admit a strict 2-representation. Conversely, however, 2-representations of a strict 2-algebra can still be weak, as  \eqref{weak2rephom} only states that the {\it cohomology class} of $\T$ is trivial, not that it is trivial as a 3-cocycle. However, if we further restrict to the case where $V$ is a strict $\cG$-bimodule (ie. the trilinear map in {\bf Definition \ref{weakbimodule}} vanishes), then $\T=0$ and $\End(V)$ is isomorphic to $\operatorname{End}(V)$.
\end{remark}

\paragraph{Example: weak coadjoint representation.}
A very natural example of a 2-representation is achieved by dualizing, using  \eqref{pairing}, the 2-representation  $\cG\rightarrow\operatorname{End}(\cG)$ given by the weak 2-algebra structure of $\cG$ on itself.

This gives rise to the {\it coadjoint representation} (cf. \cite{Bai_2013,chen:2022}) $\bar\rhd=(\rhd_1,(\rhd_0,\rhd_{-1}),\Upsilon):\cG\rightarrow\End(\cG^*)$ of $\cG$ on its dual $\cG^*$, given explicitly by
\begin{eqnarray}
    \rhd_0:\cG_0\rightarrow \End(\cG_0^*),&\qquad& \langle g,xx'\rangle = -\langle x\rhd_0 g,x'\rangle,\nonumber\\
    \rhd_{-1}:\cG_0\rightarrow \End(\cG_{-1}^*),&\qquad& \langle f,x\cdot y\rangle = -\langle x\rhd_{-1} f,y\rangle,\nonumber\\
    \Upsilon:\cG_{-1}\rightarrow \operatorname{Hom}(\cG_{-1}^*,\cG_0^*),&\qquad& \langle f,y\cdot x\rangle = -\langle \Upsilon_yf, x\rangle\label{coadj}
\end{eqnarray}
and 
\begin{equation}
    \rhd_1: \cG_0^{2\otimes}\rightarrow \End(\cG^*)_{-1}=\operatorname{Hom}(\cG_{-1}^*,\cG_0^*),\qquad \langle f,\cT(x,x',x'')\rangle = +\langle \rhd_1^{x,x'}(f),x''\rangle.\label{weakcoadj}
\end{equation}
Notice a plus sign occurs here, in contrast with the rest of the components defined in  \eqref{coadj}. This is because we have dualized two elements in $\cG$, instead of one.

Analogously, we have the coadjoint back-action $\bar\lhd=((\lhd_0,\lhd_{-1}),\tilde\Upsilon)$ of $\cG^*$ on $\cG$, which we write from the right\footnote{This means that we have, for instance, $\langle g\cdot^* f,x\rangle = -\langle g,x\lhd_{-1}f\rangle$ and $\langle f\cdot^* g,x\rangle = -\langle f,x\tilde\Upsilon_{g}\rangle$.}. The "bar" notation is used to distinguish $\bar\rhd$ from the crossed-module action $\rhd$ in the case where $\cG=kG$ is the 2-group algebra defined in  \eqref{2grpalg}.

Due to  \eqref{weak2repalg}, the components of a weak 2-representation are not genuine algebra representations in general, but only up to homotopy. We have in general that 
\begin{equation}
    (xx')\rhd_0 g = x\rhd_0(x'\rhd_0 g) + \rhd_1^{x,x'}(t^*g),\qquad (xx')\rhd_{-1}f = x\rhd_{-1}(x'\rhd_{-1}f) + t^*\rhd_1^{x,x'}(f), \nonumber
\end{equation}
where $t^*$ is the dual $t$-map on $\cG^*$, and 
\begin{equation}
    \Upsilon_{x\cdot y}f = x\rhd_0 (\Upsilon_yf) + \rhd_1^{x,ty}(f),\qquad \Upsilon_{y\cdot x}f = \Upsilon_y(x\rhd_{-1} f) + \rhd_1^{ty,x}(f).\nonumber
\end{equation}
Of course, these components reduce to genuine strict algebra representations if $\rhd_1=0$ or $t=0$, which simplifies the situation considerably.

\begin{remark}\label{2repgrpdalg}
Recall the 2-algebra $\cG=kG$ associated to the 2-group $G$ in Example \ref{sec:2gpalgb} In the skeletal case, we can induce a 2-representation $\rho$ of the 2-algebra from that of the 2-group, by extending it linearly. All 2-representations of $kG$ shall arise this way, in this case.
\end{remark}

\paragraph{1- and 2-morphisms on the weak 2-representation 2-category.}
With {\bf Definition \ref{weak2rep}} in hand, we are now ready to define the morphisms on the weak 2-representations. Let $\rho=(\varrho,\rho_0,\rho_1)$ and $\rho' = (\varrho',\rho_0',\rho_1')$ denote two weak 2-representations on $V,W\in \operatorname{2Rep}^\cT(\cG)$, respectively. 

\begin{definition}\label{weak2int}
A {\bf weak 2-intertwiner} $i = (I,i_1,i_0):V\rightarrow W$ consist of a 2-vector space homomorphism $(i_1,i_0): V\rightarrow W$ together with a collection of invertible cochain homotopies $I_{x,i}: V_0\rightarrow W_{-1}$ satisfying 
\begin{equation}
    \partial I_{x,i} = i_0\circ \rho_0^0(x) - \rho'^0_0(x)\circ i_0 ,\qquad I_{x,i}\partial = i_1\circ \rho_0^1(x) - \rho'^1_0(x)\circ i_1\nonumber
\end{equation}
for each $x\in \cG_0$, as well as 
\begin{equation}
    I_{ty,i} = i_1\circ\rho_1(y) - \rho_1'(y)\circ i_0\nonumber
\end{equation}
for each $y\in \cG_{-1}$. Moreover, $I_{\bullet,i}$ trivializes $\varrho-\varrho'$ as a Hochschild 2-cocycle, in the sense that for each $x,x'\in \cG_0$,
\begin{equation}
     \operatorname{id}_i\otimes\varrho(x,x') - \varrho'(x,x')\otimes \operatorname{id}_i = \operatorname{id}_{\rho_0(x)}\otimes I_{x',i} - I_{xx',i}+  I_{x,i} \otimes\operatorname{id}_{\rho_0(x')},\label{2inthomotopy}
\end{equation}
where $\operatorname{id}_i:i\Rightarrow i$ denotes the identity cochain homotopy on the intertwiner $i$.
\end{definition}
\noindent In other words, a weak 2-intertwiner $i: V\rightarrow W$ is such that the following diagrams
\begin{equation}
\begin{tikzcd}
V_{-1} \arrow[rd, "i_{1}"] \arrow[rr, "\partial"] \arrow[dd, "\rho^1_0"] &                                                                        & V_0 \arrow[rd, "i_0"] \arrow[dd, "\rho^0_0", bend left] &                             \\
                                                                   & W_{-1} \arrow[rr, "\partial'"', bend left] \arrow[dd, "\rho'^1_0"', bend left] &                                                          & W_0 \arrow[dd, "\rho'^0_0"] \\
V_{-1} \arrow[rr, "\partial", bend left] \arrow[rd, "i_{1}"]             &                                                                        & V_0 \arrow[rd, "i_0"]                                    &                             \\
                                                                   & W_{-1} \arrow[rr, "\partial'"]                                                &                                                          & W_0                        
\end{tikzcd},\qquad\qquad \begin{tikzcd}
V_0 \arrow[r, "\rho_1"] \arrow[d, "i_0"] & V_{-1} \arrow[d, "i_{1}"] \\
W_0 \arrow[r, "\rho_1'"]                 & W_{-1}                    
\end{tikzcd}\label{2int}
\end{equation}
commute {\it up to a natural invertible 2-morphism} given by $I_{\bullet,i}$. By definition, we have $I_{0,i} = I_{\eta_0,i} = 0$ where $\eta_0$ is the unit of $\cG_0$.

Now let $i,i':\rho\rightarrow\rho'$ denote two weak 2-intertwiners, we have the following.
\begin{definition}\label{weakmodif}
A {\bf modification} $\mu:i\Rightarrow i'$ between two weak 2-intertwiners is a $\cG$-equivariant cochain homotopy
\begin{equation}
    \begin{tikzcd}
V_{-1} \arrow[r, "\partial"] \arrow[d, "i_1-i_1'"'] & V_0 \arrow[ld, "\mu"'] \arrow[d, "i_0-i_0'"] \\
W_{-1} \arrow[r, "\partial'"]                      & W_0                                        
\end{tikzcd},\label{modif}
\end{equation}
where $\mu$ intertwines between $\rho_1(y),\rho_1'(y)$ for each $y\in \cG_{-1}$, as cochain homotopies. Moreover, $\mu$ trivializes $I_{\cdot,i}-I_{\cdot,i'}$ as a Hochschild 1-cocycle, in the sense that
\begin{equation}
    I_{x,i} - I_{x,i'} = \operatorname{id}_{\rho_0(x)}\otimes \mu-\mu   \label{modifhomotopy}
\end{equation}
for all $x\in \cG_0$, as a relation between cochain homotopies.
\end{definition}

We shall denote by $\operatorname{2Rep}^\cT(\cG)$ the 2-category of {weak 2-representations} of the weak 2-bialgebra $(\cG,\cT)$, consisting of weak 2-representation $(V,\rho)$ objects, weak 2-intertwiners $i$ as 1-morphisms and modifications $\mu$ as 2-morphisms. We will prove in Appendix \ref{weak2repthy} that our definition in fact coincides with the higher-representation theory developed in the literature \cite{Baez:2012bn,Douglas:2018,Delcamp:2021szr,Delcamp:2023kew}, in the case where $\cG=kG$ corresponding to a skeletal 2-group $G$. We devote the remainder of this section to proving that $\operatorname{2Rep}^\cT(\cG)$ forms a monoidal 2-category.

\subsection{Monoidal structure on the 2-representations}\label{fusion2rep2cat}
Recall that vector space cochain complexes come equipped with natural notions of direct sum $\oplus$, as well as tensor product $\otimes$, which satisfy the distributive law
\begin{equation}
    V\otimes(W\oplus U) = (V\otimes W)\oplus (V\otimes U), \nonumber
\end{equation} 
where $V,W,U$ are vector space cochains. For chain complexes, the direct sum is given simply by
\begin{equation}
    V\oplus W = V_{-1} \oplus W_{-1}\xrightarrow{\partial\oplus \partial'} V_0\oplus W_0,\nonumber
\end{equation}
while the tensor product is given by the following complex
\begin{equation}
    V\otimes W = \LaTeXunderbrace{V_{-1}\otimes W_{-1}}_{\text{deg}=-2}\xrightarrow{D^+}\LaTeXunderbrace{V_{-1}\otimes W_0 \oplus V_0\otimes W_{-1}}_{\text{deg}=-1}\xrightarrow{D^-} \LaTeXunderbrace{V_0\otimes W_0}_{\text{deg}=0},\label{grtensor}
\end{equation}
where $D^\pm = \pm1\otimes\partial'+\partial\otimes 1$ is the tensor extension of the differentials $\partial:V_{-1}\rightarrow V$ and $\partial':W_{-1}\rightarrow W_0$. 

We endow the direct sum and tensor product structure on 2-representations of $\cG$ in the same way as above. Note the direct double $\cG^{2\oplus}$ and the tensor square $\cG^{2\otimes}$ of a strict 2-algebra $\cG$ also have the same structure.

\paragraph{Direct sums.} For the direct sum 2-representation, this is simply accomplished by extending {\bf Definition \ref{weak2rep}} to a direct sum of 2-algebra homomorphisms
\begin{equation}
    (\varrho,\rho)\oplus(\varrho',\rho') = (\varrho\oplus\varrho',\rho\oplus \rho'): \cG\oplus\cG \rightarrow \End(V)\oplus\End(W).\nonumber
\end{equation}
In particular, the direct sum $V\oplus W$ of 2-representations of $\cG$ is given by the components
\begin{equation}
    (\rho\oplus\rho')_0^0 = \rho_0^0\oplus\rho'{}_0^0,\qquad (\rho\oplus\rho')_0^1 = \rho_0^1\oplus\rho'{}_0^1,\qquad (\rho\oplus\rho')_1 = \rho_1\oplus\rho_1' \nonumber
\end{equation}
such that the square  \eqref{commsq} commutes,
\begin{equation}
    (\rho\oplus\rho')_0\circ (t\oplus t) = (\delta \oplus\delta')\circ (\rho\oplus\rho')_1,\nonumber
\end{equation}
where $\delta,\delta'$ are the differentials of the two 2-algebras $\End(V),\End(W)$, respectively. The zero 2-representation under direct sum is of course the trivial complex $0 \rightarrow 0$.

\subsubsection{Tensor product}\label{sec:monoidal}
As in the 1-bialgebra case, the tensor product of 2-representations is accomplished by precomposing with the coproduct. However, the graded components of the coproduct $\Delta = \Delta_{-1} + \Delta_0$ in  \eqref{grpdcoprod}, as well as $\Delta_0'$ in  \eqref{deg0sweed}, allows us to define the tensor product between 2-representations $V\otimes W$
\begin{equation}
    \rho_{V\otimes W}(x) = \big((\rho_V)_0\otimes (\rho_W)_0\big)\circ\Delta_0'(x), \quad x\in\cG_0,\label{tensor2rep1}
\end{equation}
as well as its weak component (cf. {\bf Definition \ref{weak2rep}})
\begin{equation}
    \varrho_{V\otimes W}(x,x') = \varrho_V(\bar x_{(1)},\bar x_{(1)}')\otimes\varrho_W(\bar x_{(2)},\bar x_{(2)}'),\qquad x,x'\in\cG_0.\nonumber
\end{equation}
We also have the tensor product between a 2-intertwiner $i:V\rightarrow U$ and a 2-representation
\begin{align}
    \rho_{i\otimes W}(x) &= \big((\rho_U)_1\circ i \otimes (\rho_W)_0\big)\circ\Delta_0^l(x) + (-1)^\text{deg} \big(i\circ (\rho_V)_0 \otimes (\rho_W)_1\big)\circ\Delta_0^r(x),\nonumber\\
    \rho_{W\otimes i}(x) &= \big((\rho_W)_0\otimes (\rho_U)_1\circ i\big)\circ\Delta_0^r(x) +(-1)^\text{deg}\big((\rho_W)_1\otimes i\circ (\rho_V)_0\big)\circ\Delta_0^l(x)\label{tensor2rep2}
\end{align}
for each $x\in\cG_0$, where the sign depends on the degree of the components in  \eqref{grtensor}. Lastly, the tensor product between 2-intertwiners $i:V\rightarrow U,j:W\rightarrow T$ is given by just
\begin{equation}
    \rho_{i\otimes j}(y) = \big((\rho_U)_1\circ i\otimes(\rho_T)_1\circ j + (-1)^\text{deg}i\circ (\rho_V)_1\otimes j\circ (\rho_W)_1\big)\circ\Delta_{-1}(y)\label{tensor2rep3}
\end{equation}
for each $y\in\cG_{-1}$. This defines the invertible natural 2-morphism $I_{i\otimes j,\bullet}$ (cf. {\bf Definition \ref{weak2int}}).

The fact that  \eqref{tensor2rep1}, \eqref{tensor2rep2}, \eqref{tensor2rep3} define genuine 2-representations (up to the homotopy $\varrho$; cf. {\bf Definition \ref{weak2rep}} and  \eqref{weak2repalg}), for instance
\begin{equation}
    \delta\varrho_{V\otimes W}(x,x') = \rho_{V\otimes W}(xx') - \rho_{V\otimes W}(x)\rho_{V\otimes W}(x'),\nonumber
\end{equation}
requires the 2-bialgebra axioms  \eqref{2algcoprod}. 

\paragraph{Tensor unit.} Now if $\cG$ is a unital 2-bialgebra, then there is a tensor unit, denoted by $I\in\operatorname{2Rep}^\cT(\cG)$ given by the ground field complex $k\xrightarrow{1}k$, and a unit 2-intertwiner given by the identity $\id_I:1\rightarrow 1$, such that $\cG$ acts on them through multiplication of the counit $\epsilon$, 
\begin{equation*}
    \rho_I(x) = \epsilon_0(x),\qquad \rho_{\id_I}(y) = \epsilon_{-1}(y).
\end{equation*}
From \eqref{weak2repalg}, the corresponding component $\varrho=\id$ for the tensor unit $I$ is clearly the identity 2-morphism. In according with \eqref{tensor2rep1}, \eqref{tensor2rep2}, \eqref{tensor2rep3}, the condition \eqref{counit} then implies that the left- and right-unitor morphisms in $\operatorname{2Rep}^\cT(\cG)$ are all 1- and 2-isomorphisms. For instance, \eqref{deg0counit} implies
\begin{equation*}
    \rho_{V\otimes 1} = \rho_V = \rho_{1\otimes V},
\end{equation*}
whence $V\otimes 1,1\otimes V$ and $V$ coincides as 2-representations.

Due to this, all coherence diagrams in $\operatorname{2Rep}^\cT(\cG)$ concerning the unitors, such as the homotopy triangle and the zig-zag axioms \cite{Kong:2020,GURSKI20114225}, are trivially satisfied, and hence we will not directly prove them. The conditions \eqref{counit}, \eqref{deg0counit} can of course be easily relaxed to give non-trivial unitors, but we shall not consider this here.

\subsubsection{Naturality and the $\mathsf{Gray}$-property of the tensor product}
Recall the space $\End(V)_{-1}$ is modelled by cochain homotopies, which can be interpreted as "endomorphisms" on $\End(V)_0$. Using this perspective, we will prove the following key results.
\begin{lemma}\label{monoidal2rep2cat}
Let $i:V\rightarrow U$ denote a 2-intertwiner. We have the following diagrams
\begin{equation}
\begin{tikzcd}
V\otimes W \arrow[rr, "\rho_{V\otimes W}"] \arrow[dd, "i"'] &                                   & V\otimes W \arrow[dd, "i"] \\
                                                            & \xRightarrow{\rho_{i\otimes W}} &                            \\
U\otimes W \arrow[rr, "\rho_{U\otimes W}"]                  &                                   & U\otimes W
\end{tikzcd},\qquad \begin{tikzcd}
W\otimes V \arrow[rr, "\rho_{W\otimes V}"] \arrow[dd, "i"'] &                                   & W\otimes V \arrow[dd, "i"] \\
                                                            & \xRightarrow{\rho_{W\otimes i}} &                            \\
W\otimes U \arrow[rr, "\rho_{W\otimes U}"]                  &                                   & W\otimes U                
\end{tikzcd} \nonumber
\end{equation}
in $\operatorname{2Rep}^\cT(\cG)$. 
\end{lemma}
\begin{proof}
Let us focus first on the left diagram. The goal is to show that $\rho_{i\otimes W}$ defines a cochain homotopy which fits into the following diagram
\begin{equation}
    \begin{tikzcd}
V_{-1}\otimes W_{-1} \arrow[rr, "D^+"] \arrow[d] &  & V_{-1}\otimes W_0\oplus V_0\otimes W_{-1} \arrow[d] \arrow[rr, "D^-"] \arrow[lld, "\rho_{i\otimes W}"'] &  & V_0\otimes W_0 \arrow[d] \arrow[lld, "\rho_{i\otimes W}"'] \\
U_{-1}\otimes W_{-1} \arrow[rr, "D^+"']          &  & U_{-1}\otimes W_0\oplus U_0\otimes W_{-1} \arrow[rr, "D^-"']                                         &  & U_0\otimes W_0                                         
\end{tikzcd},\nonumber
\end{equation}
where the horizontal maps are the differentials given in  \eqref{grtensor}, and the vertical maps are various components of $\rho_{V\otimes W}\circ i - i\circ\rho_{V\otimes W}$. 

The key is the commutation relation  \eqref{commsq}, which allows us to write
\begin{equation}
    \delta(\rho_1(y)) = (\rho_1(y)\partial,\partial \rho_1(y)) = (\rho_0^1(ty),\rho_0^0(ty)) \nonumber
\end{equation}
for each $y\in\cG_{-1}$, as well as the definition  \eqref{deg0sweed} of $\Delta_0'$. Directly computing, we have for the rightmost triangle
\begin{eqnarray}
    D^-\rho_{i\otimes W} &=& \partial_U(\rho_U)_1(x_{(1)}^l) \circ i \otimes (\rho_W)_0^0(x_{(2)}^l) - (-1)^\text{deg} i\circ (\rho_V)_0(x_{(1)}^r) \otimes \partial_W(\rho_W)_0^0(x_{(2)}^r)\nonumber\\
    &=& (\rho_U)_0^0(tx_{(1)}^l) \circ i\otimes (\rho_W)_0^0(x_{(2)}^l) - i\circ (\rho_V)_0^0(x_{(1)}^r)\otimes (\rho_W)_0^0(tx_{(2)}^r) \nonumber\\
    &=& \rho_{U\otimes W} \circ i - i\circ\rho_{V\otimes W}\nonumber
\end{eqnarray}
as maps on $V_0\otimes W_0$ (with deg = 0), and similarly we have for the leftmost triangle
\begin{eqnarray}
    \rho_{i\otimes W}D^+ &=& (\rho_U)_1(x_{(1)}^l)\partial_U \circ i \otimes (\rho_W)_0^1(x_{(2)}^l) + (-1)^\text{deg} i\circ (\rho_V)_0(x_{(1)}^r) \otimes (\rho_W)_0^0(x_{(2)}^r)\partial_W\nonumber\\
    &=& (\rho_U)_0^1(tx_{(1)}^l) \circ i\otimes (\rho_W)_0^1(x_{(2)}^l) - i\circ (\rho_V)_0^1(x_{(1)}^r)\otimes (\rho_W)_0^1(tx_{(2)}^r) \nonumber\\
    &=& \rho_{U\otimes W} \circ i - i\circ\rho_{V\otimes W}\nonumber
\end{eqnarray}
as maps on $V_{-1}\otimes W_{-1}$ (with deg = -1). 

Now consider the middle section. We need to compute
\begin{gather}
    D^+\rho_{i\otimes W} = (\rho_U)_0^1(tx_{(1)}^l) \circ i\otimes (\rho_W)_0^0(x_{(2)}^l) - i\circ (\rho_V)_0^0(x_{(1)}^r)\otimes (\rho_W)_0^1(tx_{(2)}^r),\nonumber\\
    \rho_{i\otimes W}D^- = (\rho_U)_0^0(tx_{(1)}^l) \circ i\otimes (\rho_W)_0^1(x_{(2)}^l) - i\circ (\rho_V)_0^1(x_{(1)}^r)\otimes (\rho_W)_0^0(tx_{(2)}^r), \nonumber
\end{gather}
and sum them to find
\begin{eqnarray}
    D^+\rho_{i\otimes W}+ \rho_{i\otimes W}D^- &=& \left[(\rho_U)_0^1(tx_{(1)}^l)\otimes (\rho_W)_0^0(x_{(2)}^l) + (\rho_U)_0^0(tx_{(1)}^l)\otimes (\rho_W)_0^1(x_{(2)}^l)\right]\circ i \nonumber\\
    &\qquad& -~ i\circ \left[(\rho_V)_0^0(x_{(1)}^r)\otimes (\rho_W)_0^1(tx_{(2)}^r) + (\rho_V)_0^1(x_{(1)}^r)\otimes (\rho_W)_0^0(tx_{(2)}^r)\right] \nonumber\\
    &=& \rho_{U\otimes W} \circ i - i\circ\rho_{V\otimes W} \nonumber
\end{eqnarray}
as maps on $V_{-1}\otimes W_0\oplus V_0\otimes W_{-1}$. The other diagram is treated identically.
\end{proof}

We now show that  \eqref{tensor2rep3} is in fact {\it not} independently defined.
\begin{lemma}\label{2repdecomp}
If $j: W\rightarrow T$ is another 2-intertwiner, then $i\otimes j$  decomposes as $i\otimes j= i\otimes T \circ V\otimes j \cong U\otimes j \circ i\otimes W$. The homotopy $I_{i\otimes j,\bullet} = I_{i\otimes\id_W} \ast I_{\id_V\otimes j}$ also decomposes accordingly.
\end{lemma}
\begin{proof}
What we need to show is that $\rho_{i\otimes j} = (\rho_{i\otimes T}\ast\rho_{V\otimes j})\circ t = (\rho_{U\otimes j}\ast\rho_{i\otimes W})\circ t$ as 2-morphisms. Recall cochain homotopies $q:f\Rightarrow g,\, p:g\Rightarrow h$ in $\mathsf{2Vect}^{hBC}$ compose by $p\ast q = p\partial_U q :f\Rightarrow h$, where $U$ is the source 2-vector space of the cochain map $g$. Indeed, we have
\begin{gather}
    \partial_W(p\ast q) = (\partial_Wp)\circ (\partial_Uq) = (g_0 - h_0)\circ (f_0-g_0),\nonumber\\
    (p\ast q)\partial_V = (p\partial_U)\circ(q\partial_V) = (g_{-1}-h_{-1})\circ(f_{-1}-g_{-1})\nonumber
\end{gather}
as desired, where $W$ is the target of $h$ and $V$ is the source of $f$. Notice this is exactly how elements in $\End(V)_{-1}$ compose, $A\ast A' = A\delta A$.

The goal is to prove that $D^\pm\rho_{i\otimes j}(y)$ in fact decomposes as described above for each $y\in\cG_{-1}$. This follows from the coequivariance condition  \eqref{cohgrpd+}. By direct computation, precomposing  \eqref{tensor2rep2} yields (here we neglect the 2-vector space subscripts for brevity)
\begin{eqnarray}
    \rho_{i\otimes W}\circ t &=& (\rho_1i\otimes (\rho_0t) +(-1)^\text{deg} i(\rho_0t)\otimes \rho_1)\circ\Delta_{-1}\nonumber\\
    &=& (\rho_1i\otimes \partial\rho_1 + (-1)^\text{deg} i(\partial\rho_1)\otimes\rho_1)\circ\Delta_{-1},\nonumber\\
    \rho_{U\otimes j} \circ t &=& ((\rho_0t)\otimes\rho_1j + (-1)^\text{deg} \rho_1\otimes i(\rho_0t)) \circ\Delta_{-1} \nonumber\\
    &=& ((\rho_1\partial)\otimes \rho_1j + (-1)^\text{deg}\rho_1\otimes j( \rho_1\partial))\circ\Delta_{-1},\nonumber
\end{eqnarray}
where we have used  \eqref{commsq} to commute the $t$-map past the 2-representations to the differential $\partial$. Using the Sweeder notation  \eqref{sweed} for $\Delta_{-1}$, we 
compute their graded composition to be
\begin{eqnarray}
    (\rho_{U\otimes j})(ty)\ast(\rho_{i\otimes W})(ty)&=&(\rho_1(y_{(1)})\partial)\rho_1(y_{(1)})i\otimes \rho_1(y_{(2)})j(\partial\rho_1(y_{(2)})) \nonumber\\
    &\qquad&+~ (-1)^\text{deg}\rho_1(y_{(1)})(i\partial\rho_1(y_{(1)}))\otimes j(\rho_1(y_{(2)})\partial)\rho_1(y_{(2)}) \nonumber\\
    &=& (\rho_1(y_{(1)})\ast \rho_1(y_{(1)})) i \otimes j(\rho_1(y_{(2)})\ast \rho_1(y_{(2)})) \nonumber\\
    &\qquad&+ ~(-1)^\text{deg} (\rho_1(y_{(1)})\ast \rho_1(y_{(1)}))i\otimes j (\rho_1(y_{(2)})\ast\rho_1(y_{(2)}))\nonumber\\
    &=&(\rho_1 i \otimes \rho_1 j + (-1)^\text{deg}i\rho_1 \otimes j\rho_1)\circ\Delta_{-1}(y) = \rho_{i\otimes j}(y)\nonumber
\end{eqnarray}
as desired, where we have noted the property $i_{-1}(\rho_V)_1 = (\rho_U)_1i_0$ of the 2-intertwiners $i,j$ to permute them past the $\rho$'s. This proves that the 2-algebra homomorphisms $\rho_{i\otimes j} = \rho_{i\otimes T} \ast \rho_{V\otimes j}$ coincide. A similar argument shows that the 2-algebra homomorphisms $\rho_{i\otimes j} = \rho_{U\otimes j}\ast \rho_{i\otimes W}$ also coincide. 

This is not sufficient to imply that $i\otimes T \circ V\otimes j = U\otimes j \circ i\otimes W$, however. Indeed, the weak component $\varrho$ of the two decomposed 2-representations in general may differ. After some computations, one can show that we have
\begin{align}
    \varrho_{(i\otimes T)\circ (V\otimes j)}\circ\Delta_0(x) &= \varrho(tx_{(1)}^l,x_{(1)}^r)\otimes\varrho(x_{(2)}^l,tx_{(2)}^r)+(-1)^\text{deg}\varrho(x_{(1)}^r,tx_{(1)}^l)\otimes\varrho(tx_{(2)}^r,x_{(1)}^l),\nonumber\\
    \varrho_{(U\otimes j)\circ (i\otimes W)}\circ\Delta_0(x) &= \varrho(x_{(1)}^r,tx_{(1)}^l)\otimes\varrho(tx_{(2)}^r,x_{(1)}^l) + (-1)^\text{deg}\varrho(tx_{(1)}^l,x_{(1)}^r)\otimes\varrho(x_{(2)}^l,tx_{(2)}^r).\nonumber
\end{align}
The difference $\phi_{i,j}= \varrho_{(i\otimes T)\circ (V\otimes j)}\ast \varrho_{(U\otimes j)\circ (i\otimes W)}^{-1}$ between these 2-morphisms is precisely the 2-isomorphism $i\otimes T \circ V\otimes j \cong U\otimes j \circ i\otimes W$.
\end{proof}
\noindent This 2-isomorphism is precisely the \textit{interchanger}
\begin{equation}
    \phi_{i,j}: (i\otimes\id_T) \circ (\id_V\otimes j) \Rightarrow (\id_U\otimes j) \circ (i\otimes \id_W) \label{interchange}
\end{equation}
of two 1-morphisms $i:V\rightarrow U,j: W\rightarrow T$. A key property of monoidal 2-categories, or more generally {$\mathsf{Gray}$-enriched 3-categories}, is that most of its coherence data can be stritified away {\it aside} from this interchanger $\phi$ \cite{gurski2006algebraic,neuchl1997representation}. Conversely, given the coherent choice of the interchanger, the tensor product between 1-morphisms can be determined by the product between objects and 1-morphisms through naturality. We call this the "\textbf{$\mathsf{Gray}$-property}".

These lemmas are important, as its proof techniques will be used repeatedly in what follows.

\subsection{Monoidal associators}\label{fusionassoc}
In this section, we shall focus on the \textit{associator} morphisms attached to the 2-representations in $\operatorname{2Rep}^\cT(\cG)$, as they play a direct role in the main theorem. Other subtleties that show up are studied in detail in Appendix \ref{weak2repthy}.

Recall from Section \ref{sec:monoidal} that the tensor product on $\operatorname{2Rep}(\cG)$ is given by the coproduct $\Delta$. The associator morphisms $a$ are therefore given by the coasscociator $\Delta_1:\cG_0\rightarrow\cG_{-1}^{3\otimes}$ attached to the coproduct in $\cG$, and {\it not} the Hochschild 3-cocycle $\cT$. However, the data $\Delta_1,\cT$ are dual to each other by {\bf Proposition \ref{weak2bialg}}, hence if $\cG$ is self-dual (like the weak (skeletal) 2-quantum double as we  constructed in Section \ref{weakskeletalqd}), they in fact constitute the same data. As such we shall denote the weak 2-representation 2-category by $\operatorname{2Rep}^\cT(\cG)$. We shall neglect the tensor product notation $\otimes$ in the following.

\medskip

We begin by constructing the associator 2-morphism $a_{ijk}: (i\otimes j)\otimes k\Rightarrow i\otimes (j\otimes k)$ on the triple $i:V\rightarrow V',j:W\rightarrow W',k:U\rightarrow U'$ of 2-intertwiners. By definition in section \ref{weakcohgrpd1}, we see that the following quantity
\begin{equation}
    a_{ijk} = ((\rho_{V'})_1\circ i\otimes(\rho_{W'})_1\circ j\otimes(\rho_{U'})_1\circ  k +(-1)^\text{deg}i\circ (\rho_{V})_1\otimes j\circ (\rho_{W})_1\otimes k\circ (\rho_{U})_1)\circ (\Delta_1\circ t)\label{fusionassoc1}
\end{equation}
defines a cochain homotopy that fits into the following equation $\rho_{(ij)k} - \rho_{i(jk)} = a_{ijk}$, which induces a 2-morphism (also denoted by $a_{ijk}$) between the 2-intertwiners 
\begin{equation}
    a_{ijk}: (ij)k\Rightarrow i(jk).\nonumber
\end{equation}
Secondly,  \eqref{weakcohgrpd1} implies that the following quantities based on $D_t\Delta_1$,
\begin{eqnarray}
    a_{Vjk} &=& ((\rho_V)_0\otimes(\rho_{W'})_1\circ j\otimes(\rho_{U'})_1\circ  k +(-1)^\text{deg} (\rho_V)_0\otimes j\circ (\rho_{W})_1\otimes k\circ (\rho_{U})_1)\nonumber\\
    &\qquad&\circ (t\otimes 1\otimes 1)\Delta_1,\nonumber\\
    a_{iWk} &=& ((\rho_{V'})_1\circ i\otimes(\rho_{W})_0\otimes(\rho_{U'})_1\circ  k +(-1)^\text{deg} i\circ (\rho_{V})_1\otimes (\rho_{W})_0\otimes k\circ (\rho_{U})_1)\nonumber\\
    &\qquad& \circ (1\otimes t\otimes 1)\Delta_1,\nonumber\\
    a_{ijU} &=& ((\rho_{V'})_1\circ i\otimes(\rho_{W'})_1\circ j\otimes(\rho_{U})_0 +(-1)^\text{deg} i\circ (\rho_V)_1\otimes j\circ (\rho_{W})_1\otimes  (\rho_{U})_0)\nonumber\\
    &\qquad& \circ (1\otimes 1\otimes t)\Delta_1,\label{fusionassoc2}
\end{eqnarray}
give rise to the associators for the following tensor products,
\begin{equation}
    a_{Vjk}: (Vj)k\Rightarrow V(jk),\qquad a_{iWk}: (iW)k\Rightarrow i(WK),\qquad a_{ijU}: (ij)U\Rightarrow i(jU) \nonumber
\end{equation}
for the mixed tensor products defined by  \eqref{tensor2rep2}. Thirdly,  \eqref{weakcohgrpd2} implies that the following quantities based in $D_t[2]\Delta_1$, 
\begin{eqnarray}
    a_{VWk} &=& ((\rho_V)_0\otimes(\rho_{W})_0\otimes(\rho_{U'})_1\circ  k +(-1)^\text{deg} (\rho_V)_0\otimes (\rho_{W})_0\otimes k\circ (\rho_{U})_1)\nonumber\\
    &\qquad& \circ (t\otimes t\otimes 1)\Delta_1,\nonumber\\
    a_{iWU} &=& ((\rho_{V'})_1\circ i\otimes(\rho_{W})_0\otimes(\rho_{U})_0 +(-1)^\text{deg} i\circ (\rho_{V})_1\otimes (\rho_{W})_0\otimes (\rho_{U})_0)\nonumber\\
    &\qquad& \circ (1\otimes t\otimes t)\Delta_1,\nonumber\\
    a_{VjU} &=& ((\rho_{V})_0\otimes(\rho_{W'})_1\circ j\otimes(\rho_{U})_0 +(-1)^\text{deg} (\rho_V)_0\otimes j\circ (\rho_{W})_1\otimes  (\rho_{U})_0)\nonumber\\
    &\qquad& \circ (t\otimes 1\otimes t)\Delta_1,\label{fusionassoc3}
\end{eqnarray}
serve as the associators 
\begin{equation}
    a_{VWk}: (VW)k\Rightarrow V(Wk),\qquad a_{VjU}:(Vj)U\Rightarrow V(jU),\qquad a_{iWU}: (iW)\Rightarrow i(WU),\nonumber
\end{equation}
Notice that these quantities we have defined so far are all cochain homotopies/2-mophisms in $\operatorname{2Rep}^\cT(\cG)$, due to the appearance of $\rho_1$ in their tensor products. 

Lastly,  \eqref{weakcohgrpd3} allows us to define the associator {\it 1-morphism},
\begin{equation}
    a_{VWU} = ((\rho_V)_0\otimes (\rho_W)_0\otimes (\rho_U)_0)(\Phi), \label{fusionassoc4}
\end{equation}
with  $\Phi\equiv (t\otimes t\otimes t)\Delta_1:\cG_0\rightarrow \cG_0^{3\otimes}$, which induces an invertible 1-morphism
\begin{equation}
    a_{VWU}: (VW)U\rightarrow V(WU)\nonumber
\end{equation}
that intertwines between $\rho_{(V\otimes W)\otimes U}$ and $\rho_{V\otimes(W\otimes U)}$.

The adjoint associator 2-morphism $a^\dagger$ is implemented by minus the corresponding cochain homotopy. For  \eqref{fusionassoc4}, however, the adjoint morphism $a^\dagger_{VWU}$ is given by the inverse $\Phi^{-1}$.

\paragraph{The pentagon relation and naturality of the associator.}
We now prove the following.
\begin{lemma}
Suppose the 3-cocycle $\T=0$ is trivial for the moment. The pentagon relation for the associators $a$ arising from  \eqref{fusionassoc1}, \eqref{fusionassoc2}, \eqref{fusionassoc3}, \eqref{fusionassoc4} follows from the 2-coassociativity condition  \eqref{weakcohgrpd3} for $\Delta_1$.
\end{lemma}
\begin{proof}
Consider first  \eqref{fusionassoc1}. We {\it precompose}  \eqref{weakcohgrpd3} with $t$ and reconstruct the associators corresponding to each term according to the definition,
\begin{gather}
    (\id\otimes (\Delta_1\circ t))\circ\Delta_{-1}\leadsto \id_i \otimes a_{jkl},\qquad ((\Delta_1\circ t)\otimes\id )\circ\Delta_{-1}\leadsto a_{ijk}\otimes\id_l,\nonumber\\
    (1\otimes \Delta_{-1}\otimes 1)\circ\Delta_1\circ t \leadsto a_{i(jk)l},\qquad
    -(\Delta_{-1}\otimes 1\otimes 1)\circ\Delta_1\circ t\leadsto a^\dagger_{(ij)kl},\nonumber\\
    -(1\otimes 1\otimes \Delta_{-1})\circ\Delta_1\circ t\leadsto a^\dagger_{ij(kl)},\nonumber
\end{gather}
where $\operatorname{id}_i:i\Rightarrow i$ denotes the identity modification on the 2-intertwiner $i$. Now note that, by coequivariance  \eqref{cohgrpd+} $D_t\circ \Delta_{-1} = \Delta_0\circ t$, we have
\begin{equation}
    (\id\otimes (\Delta_1\circ t))\circ\Delta_{-1} = (\id\otimes \Delta_1)\circ \Delta_0^l\circ t,\qquad  ( (\Delta_1\circ t)\otimes\id)\circ\Delta_{-1} = (\Delta_1\otimes \id)\circ \Delta_0^r\circ t,\nonumber
\end{equation}
whence the pentagon relation
\begin{equation}
    \begin{tikzcd}
                                                        &                                                             & ((ij) k) l \arrow[rrd, "a_{ijk}\operatorname{id}_l "]  &                  &                                                        \\
(ij) (kl) \arrow[rru, "a^\dagger_{(ij)kl}"] &                                                             &                                                                                            &                  & (i (jk)) l \arrow[ldd, "{a_{i(jk)l}}"] \\
                                                        &                                                             &                                                                                            &                  &                                                        \\
                                                        & i (j(kl)) \arrow[luu,"a^\dagger_{ij(kl)}"'] &                                                                                            & i ((jk) l) \arrow[ll,"\operatorname{id}_ia_{jkl}"'] &                                                       
\end{tikzcd}\label{2ddpentagon}
\end{equation}
is equivalently expressed as 
\begin{eqnarray}
    0&=&(1\otimes \Delta_{-1}\otimes 1)\circ\Delta_1 \circ t- (\Delta_{-1}\otimes 1\otimes 1)\circ\Delta_1\circ t - (1\otimes 1\otimes\Delta_{-1})\circ\Delta_1\circ t\nonumber\\
    &\qquad&~+(\Delta_1\otimes 1)\circ\Delta_0^r\circ t + (1\otimes \Delta_1)\circ\Delta_0^l\circ t\nonumber\\
    &=& \left[-\Delta_{-1}\circ\Delta_1 + \Delta_1\circ\Delta_0\right] \circ t,\nonumber
\end{eqnarray}
which is nothing but the 2-coassociativity  \eqref{weakcohgrpd3} precomposed with $t$. Now by the coPeiffer identity $\Delta_0' = D_t\Delta_0$  \eqref{deg0sweed}, the same argument shows that the pentagon relations for the rest of the associator 2-morphisms  \eqref{fusionassoc2}, \eqref{fusionassoc3} are equivalent to applying the $t$-map $D_t,D_t[2]$ to  \eqref{weakcohgrpd3}.

Similarly, under the complete $t$-map $D_t[3] = t\otimes t\otimes t$, the 2-coasscociativity condition  \eqref{weakcohgrpd3} becomes
\begin{equation}
    \Delta_0'\circ \Phi = \Phi\circ \Delta_0',\label{2coassobj}
\end{equation}
which by  \eqref{tensor2rep1} implies the pentagon relation for the associator 1-morphism  \eqref{fusionassoc4}.
\end{proof}
\noindent We examine the case where $\T\neq 0$ is non-trivial in Appendix \ref{weak2repthy}. In particular, we show in {\bf Proposition \ref{pentag}} that $\T$ gives rise to the {\it pentagonator} 2-morphism $\pi$ in $\operatorname{2Rep}^\cT(\cG)$ implementing the pentagon axioms akin to \eqref{2ddpentagon}.

Recall from {\bf Proposition \ref{weak2bialg}} that, for a self-dual weak 2-bialgebra,  \eqref{weakcohgrpd3} follows from the 3-cocycle condition for the Hochschild 3-cocycle $\cT$. Thus the entirety of the 2-bialgebra (or 2-Hopf algebra) structure plays a central role, precisely as one would expect in Tannakian duality \cite{Majid:1990bt,Pfeiffer2007}.

\begin{lemma}
    The associator 2-morphism  \eqref{fusionassoc3} fits into diagrams of the form
\begin{equation}
\begin{tikzcd}
(VW)U \arrow[rr, "a_{VWU}"] \arrow[dd, "k"'] &                                   & V(WU) \arrow[dd, "k"] \\
                                                            & \xRightarrow{a_{VWk}} &                            \\
(VW)U' \arrow[rr, "a_{VWU'}"]                  &                                   & V(WU')                
\end{tikzcd}\label{assoc2morph}
\end{equation}
together with the associator morphism  \eqref{fusionassoc4}. Moreover, the associator 2-morphisms  \eqref{fusionassoc1}, \eqref{fusionassoc2} are completely determined by  \eqref{fusionassoc3}, \eqref{fusionassoc4}. 
\end{lemma}
\begin{proof}
    The first statement follows directly from the definitions, and by using the same argument as in the proofs of {\bf Lemma \ref{monoidal2rep2cat}}, and also later in {\bf Lemma \ref{braiding2rep2cat2}}. Similarly, by adapting the proof of {\bf Lemma \ref{2repdecomp}}, we see that  \eqref{fusionassoc1}, \eqref{fusionassoc2} admit the following decompositions
\begin{equation}
    a_{ijk} = (a_{V'W'k}\cdot a_{ijU})\circ t = \dots \text{etc.},\qquad D_\delta a_{ijU} = a_{V'jU}\cdot a_{iWU} = \dots \text{etc.},\nonumber
\end{equation}
where $D_\delta$ is the tensor triple of the $t$-map $\delta$ on $\End(V)$, and "etc." means permutations of the subscripts. This proves the second statement.
\end{proof}
\noindent This naturaliy property shall become very important later in Section \ref{braided2cat}.

\begin{remark}\label{quasihopf}
Suppose the endomorphism $\Phi$ in  \eqref{fusionassoc4} is {\it inner}, in the sense that it is given by conjugation with an element --- also denoted $\Phi$ --- of $\cG_0^{3\otimes}$, then the coassociativity condition becomes
\begin{equation}
    (\id\otimes\Delta_0')\circ\Delta'_0 = \Phi((\Delta_0'\otimes\id)\circ\Delta_0')\Phi^{-1},\nonumber
\end{equation}
and the 2-coassociativity condition  \eqref{2coassobj} becomes
\begin{equation}
    ((\id\otimes\id\otimes\Delta_0')\Phi)((\Delta_0'\otimes\id\otimes\id)\Phi) = (\Phi\otimes \eta_0)((\id\otimes\Delta_0'\otimes\id)\Phi)(\eta_0\otimes\Phi),\nonumber
\end{equation}
where $\eta_0$ is the unit of $\cG_0$. In other words, $(\cG_0,\Delta'_0,\Phi)$ in fact forms a {\it quasi-bialgebra} \cite{book-quasihopf}.
\end{remark}

We have established $\operatorname{2Rep}^\cT(\cG)$ as a monoidal (fusion) 2-category, or equivalently a monoidal bicategory. We now turn to the braiding structure in the following.

\section{The braided monoidal 2-category of 2-representations}\label{braiding2rep}
We now turn to the braiding structure on the weak 2-representations afforded by the 2-R-matrix $\cR$. We shall first examine some of the basic properties of the braiding map in Section \ref{sec:strictbraiding}. We will then study how such braiding maps interact with the weakened monoidal structures of the 2-representations 
in Section \ref{braided2cat}. 

Let $(\cG,\cdot,\Delta,\cR)$ denote a strict quasitriangular 2-bialgebra as defined in Section \ref{quasitrihopf}. Recall that a 2-R-matrix $\cR=\cR^l+\cR^r$ on the 2-bialgebra $\cG$ consist of the following components 
\begin{equation}
    \cR^l = \cR^l_{(1)}\otimes \cR^l_{(2)}\in  \cG_{-1}\otimes\cG_0,\qquad
    \cR^r=\cR_{(1)}^r\otimes \cR_{(2)}^r \in \cG_0\otimes\cG_{-1}\nonumber
\end{equation}
for which  \eqref{def2R1}, \eqref{def2R2}, \eqref{2ybequiv} are satisfied. The  equivariance condition,  \eqref{2ybequiv}, unambiguously defines an element
\begin{equation}
    R = \cR_{(1)}^r \otimes t\cR^r_{(2)} (\equiv R^r) = t\cR_{(1)}^l\otimes \cR^l_{(2)} (\equiv R^l) \in \cG_0\otimes\cG_0,\label{deg0rmat}
\end{equation}
where $t:\cG_{-1}\rightarrow \cG_0$ is the $t$-map on $\cG$. Notice by applying the $t$-map (at every leg in $\cG_{-1}$) to  \eqref{2yangbax2}, we obtain two identical expressions that are equivalent to the usual 1-Yang-Baxter equations
\begin{equation}
    R_{12}R_{13}R_{23} = R_{23}R_{13}R_{12} \nonumber
\end{equation}
for the degree-0 $R$-matrix  \eqref{deg0rmat}.

\subsection{The braiding maps and their naturality} \label{sec:strictbraiding}
We shall use these components to define the braiding $b$ on $\operatorname{2Rep}^\cT(\cG)$. Take two 2-representations $V,W$ of $\cG$; we define the braiding map between $V,W$ by
\begin{equation}
b_{VW}: V\otimes W \rightarrow W\otimes V, \quad 
    b_{VW}
    =\text{flip}\circ \rho_0(R)
    \label{2repbraid1}
\end{equation}
where $\rho_0 = (\rho_V)_0 \otimes (\rho_W)_0$ on $V\otimes W$, and $R\in\cG_0\otimes\cG_0$ is given in \eqref{deg0rmat}. By  \eqref{tensor2rep1}, the braiding between the tensor product 2-representations are then given by
\begin{equation}
    b_{V(W\otimes U)}= \text{flip}\circ \rho_0((1\otimes\Delta_0')R),\qquad b_{(V\otimes W)U}= \text{flip}\circ \rho_0((\Delta_0'\otimes 1)R).\nonumber
\end{equation}

If $W=V$ are the same 2-representations of $\cG$, then we have the {\it self-braiding} map $b_V=b_{VV}$. On the other hand, we define the {\it mixed braiding} map between a 1-morphism $i:V\rightarrow U$ and an object $W$ by
\begin{eqnarray}
    b_{iW} &=&\text{flip}  \circ \left[i\circ\rho_{10}(\cR^l) +(-1)^{\operatorname{deg}} \rho_{01}(\cR^r)\circ i\right], \nonumber\\
    b_{Wi} &=&\text{flip}  \circ \left[i\circ\rho_{01}(\cR^r) +(-1)^{\operatorname{deg}} \rho_{10}(\cR^l)\circ i\right],\label{2repbraid2}
\end{eqnarray}
where we have used the shorthand $\rho_{10} = (\rho_V)_1\otimes (\rho_W)_0$ and $\rho_{01} = (\rho_U)_0\otimes (\rho_W)_1$. The sign $(-1)^{\operatorname{deg}}$ depends on the degree of the complex $V\otimes W$; more explicitly, $b_{iW}$ gives two maps 
\begin{gather}
    b_{iW}^1: V_0\otimes W_0\rightarrow (W_{-1}\otimes U_0)\oplus (W_0\otimes U_{-1}),\nonumber\\
    b_{iW}^2: (V_{-1}\otimes W_0)\oplus (V_0\otimes W_{-1}) \rightarrow W_{-1}\otimes U_{-1}\nonumber
\end{gather}
on the tensor product $V\otimes W$, the latter of which carries a non-trivial sign $(-1)^\text{deg}=-1$; similarly for $b_{Wi}$. 

\begin{remark}
We shall {\bf define} the braiding maps $b_{ij}$ between two 1-morphisms $i,j$ by the decomposition formula
\begin{equation}
    \phi_{i,j} = b_{ij} = b_{jU}\cdot b_{Wi} = b_{Ti}\cdot b_{jV} ,\qquad \begin{cases} i: V\rightarrow U \\ j: W\rightarrow T\end{cases},\label{1morphbraid}
\end{equation}
and impose the condition that it be equal to the interchanger $\phi_{i,j}$ \eqref{interchange}. This bypasses the need for a $R$-matrix defined in degree-(-1), which is not encoded by the 2-$R$-matrix $\cR$ anyway.
\end{remark}

Let $i:V\rightarrow V',j:U\rightarrow U'$ denote any 2-intertwiner. The above definition  \eqref{2repbraid2}, together with  \eqref{tensor2rep2} then allows us to form
\begin{eqnarray}
    b_{(i\otimes W)j}&=& \text{flip}_{(V'\otimes U')\otimes W} \circ \left[(i\otimes j)\rho_{101}((\Delta_0^l\otimes1)\cR^r)+(-1)^\text{deg}\rho_{011}((\Delta_0^r\otimes 1)\cR^r)\circ (i\otimes j)\right],\nonumber\\
    b_{i(W\otimes j)}&=& \text{flip}_{W\otimes (V'\otimes U')} \circ \left[(i\otimes j)\rho_{101}((\id\otimes\Delta_0^r)\cR^l)+(-1)^\text{deg}\rho_{110}((1\otimes\Delta_0^l)\cR^l)\circ (i\otimes j)\right].
    \nonumber
\end{eqnarray}
By applying strict 2-representations to  \eqref{def2R1}, we obtain the following strict {\it higher hexagon relations},
\begin{equation}
    b_{(i\otimes W)j} = \id_i\otimes b_{Wj}\ast b_{Wi}\otimes\id_j,\qquad b_{i(W\otimes j)} =\id_i\otimes b_{jW}\ast  b_{iW}\otimes\id_j,\label{str2hexag}
\end{equation}
in which the associator isomorphisms $a$ have been suppressed. We will reinstate them later in Section \ref{braided2cat}.

\medskip

With the definitions  \eqref{2repbraid1}, \eqref{2repbraid2} in hand, we now need to prove some very important lemmas. 

\begin{lemma}\label{braidingmorphs}
The maps $b_{VW}$ and $b_{iW},b_{Wi}$ are respectively 2-intertwiners and modifications in $\operatorname{2Rep}(\cG)$ for all 2-representation $V,W$ and 2-intertwiner $i$ iff  \eqref{def2R2} is satisfied. 

\end{lemma}
\begin{proof}
Note for each 2-representation $\rho$, the flip map $\text{flip}:V\otimes W\rightarrow W\otimes V$ is a 2-intertwiner between $\rho$ and $\rho'=\rho\circ\sigma$. Moreover, we interpret the cochain homotopy defined by $(\rho_{V\otimes W})_0^1(x)$ for each $x\in \cG_0$ as a modification between the action $(\rho_{V\otimes W})_0^0(x)$ and itself, treated as a 2-intertwiner; similarly for $\rho'$. Therefore, in order for the mixed braiding map $b_{iW}$ to be a modification in $\operatorname{2Rep}(\cG)$, it must commute with the cochain homotopy $(\rho_{V\otimes W})_0^1(x)$ --- namely
    \begin{equation}
        b_{iW}\ast (\rho_{V\otimes W})_0^1(x) = (\rho'_{W\otimes V})_0^1(x)\ast b_{iW},\nonumber
    \end{equation}
    where $\ast$ denotes the composition of cochain homotopies. With $\rho_{W\otimes V}' = (\rho_W\otimes\rho_V) \circ \sigma\circ \Delta$, this is satisfied by definition  \eqref{2repbraid1} of $b_{iW}$ iff
    \begin{equation}
        \cR^r\Delta_0^r(x)=\sigma(\Delta_0^l(x))\cR^r,\qquad  \cR^l\Delta_0^l=\sigma(\Delta_0^r(x))\cR^l,\label{2rmatint2}
    \end{equation}
    which is precisely  \eqref{def2R2}.
    
    Similarly, in order for the braiding map $b_{VW}$ to be a 2-intertwiner, it must commute with the action $(\rho_{V\otimes W})_0^0(x)$ for each $x\in \cG_0$:
    \begin{equation}
        b_{VW} \circ (\rho_{V\otimes W})_0^0(x) = (\rho'_{W\otimes V})_0^0(x)\circ b_{VW},\nonumber
    \end{equation}
    where $\circ$ denotes the composition of 2-intertwiners.

    First if the 2-representation $\rho$ were strict, then this translates to the algebraic condition
    \begin{equation}
        \sigma\Delta_0'(x)R = R\Delta_0'(x), \nonumber
    \end{equation}
    which in fact follows also from  \eqref{def2R2}. To see this, we recall the definitions  \eqref{deg0rmat} of $R$ and  \eqref{deg0sweed} of the coproduct $\Delta_0'$, and simply apply $t\otimes 1$ and $1\otimes t$ respectively to  \eqref{def2R2}. The fact that $t$ is an algebra homomorphism and that $(t\otimes 1)\circ\sigma = \sigma\circ(1\otimes t)$ proves the statement.

    Second, if the 2-representation $\rho$ were weak, then in general the component $\varrho$ gives rise to a possibly non-trivial invertible natural 2-morphism 
    \begin{equation}
        \varrho(\sigma\Delta'_0(x),R) - \varrho(R,\Delta_0'(x)).\nonumber
    \end{equation}
    We will not need this 2-morphism in the following so we shall suppose $I_{b_{VW},\bullet}=\id$.

\end{proof}
\noindent Notice this lemma implies that $(\cG_0,\Delta_0',R)$ forms an ordinary quasitriangular 1-bialgebra in degree-0. We can then leverage the well-known result in the literature \cite{Majid:1996kd,Joyal:1993} that the Yang-Baxter equation for $R$ implies the hexagon relation for the braiding structure $b_{VW}$ at the level of the objects.

\medskip 

Next, we need to prove the naturality of $b$ with respect to the 2-intertwiners $i:V\rightarrow U$.
We shall do this via the same technique as {\bf Lemma \ref{monoidal2rep2cat}}.
\begin{lemma}\label{braiding2rep2cat2} Consider the intertwiners $i:V\rightarrow U$ and $j:U\rightarrow T$. 
The mixed braiding maps $b_{iW},b_{Wi}$ fit into the following diagrams
\begin{equation}
    \begin{tikzcd}
V\otimes W \arrow[rrr, "b_{VW}"] \arrow[dd, "i"'] &                                    &    & W\otimes V \arrow[dd, "i"] \\
                                                  & {} \arrow[r, "b_{iW}", Rightarrow] & {} &                            \\
U\otimes W \arrow[rrr, "b_{UW}"]                  &                                    &    & W\otimes U                
\end{tikzcd},\qquad     
\begin{tikzcd}
W\otimes V \arrow[rrr, "b_{WV}"] \arrow[dd, "i"'] &                                    &    & V\otimes W \arrow[dd, "i"] \\
                                                  & {} \arrow[r, "b_{Wi}", Rightarrow] & {} &                            \\
W\otimes U \arrow[rrr, "b_{WU}"]                  &                                    &    & U\otimes W                
\end{tikzcd}\nonumber
\end{equation}
in $\operatorname{2Rep}^\cT(\cG)$. Moreover, given a 2-intertwiner $j:U\rightarrow T$ composable with $i$, the corresponding braiding 2-morphisms compose as $b_{jW}\ast b_{iW} = b_{j\circ i,W}$.
\begin{equation}
    \begin{tikzcd}
V\otimes W \arrow[rrr, "b_{VW}"] \arrow[dd, "i"'] &                                    &    & W\otimes V \arrow[dd, "i"] \\
                                                  & {} \arrow[r, "b_{iW}", Rightarrow] & {} &                            \\
U\otimes W \arrow[rrr, "b_{UW}"]    \arrow[dd, "j"']               &                                    &    & W\otimes U  \arrow[dd, "j"']             
\\& {} \arrow[r, "b_{jW}", Rightarrow] & {} & \\ 
T\otimes W \arrow[rrr, "b_{WT}"]& & & W\otimes T 
\end{tikzcd}
=     
\begin{tikzcd}
V\otimes W \arrow[rrr, "b_{VW}"] \arrow[dd, "j\circ i"'] &                                    &    & W\otimes V \arrow[dd, "j\circ i"] \\
                                                  & {} \arrow[r, "b_{(j\circ  i)W}", Rightarrow] & {} &                            \\
T\otimes W \arrow[rrr, "b_{TW}"]                  &                                    &    & W\otimes T       
\end{tikzcd}\nonumber
\end{equation}
\end{lemma}
\begin{proof}
For brevity, we shall suppress the subscripts $V,U,W$ on the 2-representations. Recall the two equivalent ways $R^r,R^l$ to express $R$ in  \eqref{deg0rmat}. We can then write $$b_{UW} \circ i = \text{flip}\circ \rho_0(R^r) \circ i,\qquad i\circ b_{VW} = i\circ \text{flip}\circ \rho_0(R^l).$$ 

Consider the left diagram. As 2-morphisms in $\operatorname{2Rep}(\cG)$ are given by cochain homotopies, we need to show that the definition  \eqref{2repbraid2} of the mixed braiding map $b_{iW} = b_{iW}^1 + b_{iW}^2$ fits into the following diagram 
\begin{equation}
    \begin{tikzcd}
V_{-1}\otimes W_{-1} \arrow[r, "D^+"] \arrow[d] & V_{-1}\otimes W_0 \oplus V_0\otimes W_{-1} \arrow[d] \arrow[r, "D^-"] \arrow[ld, "b_{iW}^2"'] & V_0\otimes W_0 \arrow[d] \arrow[ld, "b_{iW}^1"] \\
W_{-1}\otimes U_{-1} \arrow[r, "D^+"']          & W_{-1}\otimes U_0 \oplus W_0\otimes U_{-1} \arrow[r, "D^-"']                                 & W_0\otimes U_0                                 
\end{tikzcd},\label{mixedbraidhomotopy}
\end{equation}
where the vertical arrows are the various graded components of $b_{UW}\circ i - i\circ b_{VW}$, and the horizontal arrows are the differentials on the three-term tensor product complex  \eqref{grtensor}; for instance, the ones at the top row are given by $D^\pm= 1\otimes\partial_W\pm \partial_V\otimes 1$.

As in {\bf Lemma \ref{monoidal2rep2cat}}, the key towards this is the commutative square  \eqref{commsq}, which states that for each $y\in \cG_{-1}$ we have
\begin{equation}
    (\rho_1(y)\partial,\partial\rho_1(y)) =  \delta(\rho_1)(y) = (\rho_0)(Ty) = (\rho_0^1(Ty),\rho_0^0(Ty)).\nonumber
\end{equation}
Let us examine first the commutative triangle on the ends of  \eqref{mixedbraidhomotopy}. First, for the right-most triangle, we compute in terms of the components $b_{iW}^{1,2}$ that
\begin{eqnarray}
    D^-b_{iW}^1 &=& (1\otimes\partial_V-\partial_W\otimes 1) \circ \text{flip} \circ \rho(\cR)\nonumber\\
    &=& \text{flip}\circ \left[\rho_0^0(\cR^r_{(1)})\otimes \partial_W (\rho_1(\cR^r_{(2)}))\circ i-i\circ i\partial_V(\rho_1(\cR^l_{(1)}))\otimes \rho_0^0(\cR^l_{(2)})\right] \nonumber\\
    &=& \text{flip}\circ \left[\rho_0^0(\cR^r_{(1)}\otimes t\cR_{(2)}^r)\circ  i+i\circ \rho_0^0(-t\cR^l_{(1)} \otimes \cR^l_{(2)})\right] \nonumber\\
    &=& b_{UW} \circ i - i \circ b_{VW}\nonumber
\end{eqnarray}
as maps on $V_0\otimes W_0$. Similarly for the left-most triangle, we have
\begin{eqnarray}
    b_{iW}^2D^+ &=& \text{flip} \circ \rho(\cR)\circ ( 1\otimes\partial_W+\partial_V\otimes 1) \nonumber\\
    &=& \text{flip}\circ \left[\rho_0^1(\cR^r_{(1)})\otimes (\rho_1(\cR^r_{(2)}))\partial_W\circ i-i\circ (\rho_1(\cR^l_{(1)}))\partial_V\otimes \rho_0^1(\cR^l_{(2)})\right] \nonumber\\
    &=& \text{flip}\circ \left[\rho_0^1(\cR^r_{(1)}\otimes t\cR_{(2)}^r)\circ i-i\circ \rho_0^1(t\cR^l_{(1)} \otimes \cR^l_{(2)})\right] \nonumber\\
    &=& b_{UW} \circ i - i \circ b_{VW}\nonumber
\end{eqnarray}
as maps $V_{-1}\otimes W_{-1}$. Note the sign $(-1)^\text{deg}$ in  \eqref{2repbraid2} is non-trivial here as $\cR$ acts on the degree-(-1) part of the tensor product $V\otimes W$.

We now turn to the middle section of  \eqref{mixedbraidhomotopy}. We are required to compute the following,
\begin{eqnarray}
    D^+b_{iW}^2 &=& (1\otimes\partial_V+\partial_W\otimes 1 )\circ\text{flip}\circ \rho(\cR)\nonumber\\
    &=& \text{flip}\circ \left[\rho_0^1(\cR_{(1)}^r)\otimes \partial_W(\rho_1(\cR_{(2)}^r))\circ i-i\circ \partial_V(\rho_1(\cR^l_{(1)}))\otimes \rho_0^1(\cR_{(2)}^l) \right] \nonumber\\
    &=& \text{flip}\circ \left[ \rho_0^1(\cR_{(1)}^r)\otimes \rho_0^0(t\cR_{(2)}^r)\circ i-i\circ\rho_0^0(t\cR^l_{(1)}) \otimes \rho_0^1(\cR_{(2)}^l)\right],\nonumber\\
    b_{iW}^1D^- &=& \text{flip}\circ \rho(\cR) \circ (1\otimes \partial_W-\partial_V\otimes 1 ) \nonumber\\
    &=& \text{flip} \circ \left[\rho_0^0(\cR^r_{(1)}) \otimes \rho_1(\cR_{(2)}^r)\partial_W\circ i -i\circ \rho_1(\cR_{(1)}^l)\partial_V\otimes \rho_0^0(\cR^l_{(2)})\right]\nonumber\\
    &=& \text{flip} \circ\left[\rho_0^0(\cR^r_{(1)}) \otimes \rho_0^1(t\cR_{(2)}^r)\circ i -i\circ \rho_0^1(t\cR_{(1)}^)\otimes \rho_0^0(\cR^l_{(2)})\right].\nonumber
\end{eqnarray}
Summing these and rearranging terms gives, as maps on $V_{-1}\otimes W_0\oplus V_0\otimes W_{-1}$,
\begin{gather}
    \text{flip}\circ \left[\rho_0^0(\cR^r_{(1)}) \otimes \rho_0^1(t\cR_{(2)}^r) + \rho_0^1(\cR_{(1)}^r)\otimes \rho_0^0(t\cR_{(2)}^r)\right]\circ i \\ 
    \qquad -~\text{flip}\circ i\circ \left[ \rho_0^1(t\cR_{(1)}^l)\otimes \rho_0^0(\cR^l_{(2)})+\rho_0^0(t\cR^l_{(1)}) \otimes \rho_0^1(\cR_{(2)}^l)\right] \nonumber\\
    = b_{UW}\circ i - i\circ b_{VW}. \nonumber
\end{gather}
The diagram on the right is treated identically, and this establishes the first statement. The second statement directly follows from the fact that $(j\circ i)\circ \rho_V = j\circ \rho_U \circ i = \rho_X\circ (j\circ i)$ for composable 2-intertwiners $i,j$.
\end{proof}

In particular, since {\bf Lemma \ref{braidingmorphs}} proves that $b_{VW}$ is a 1-morphism, we can {\it iterate} the braiding maps and define $b_{b_{VW}U}$ as a 2-morphism. {\bf Lemma \ref{braiding2rep2cat2}} then implies that this is a 2-morphism
\begin{equation}
    \begin{tikzcd}
(V\otimes W)\otimes U \arrow[rr, "b_{(VW)U}"] \arrow[dd, "b_{VW}"] &                           & U\otimes (V\otimes W) \arrow[dd, "b_{VW}"] \\
                                                                   & \xRightarrow{b_{b_{VW}U}} &                                            \\
(W\otimes V)\otimes U \arrow[rr, "b_{(WV)U}"]                      &                           &  U\otimes (W\otimes V)                    
\end{tikzcd}\label{iteratedbraid}
\end{equation}
on three 2-representations $V,W,U$, and similarly for $b_{Vb_{WU}}$. This will be important later in Section \ref{proof}.

\medskip

Recall the "higher-hexagon relations" \eqref{str2hexag} following directly from the identities  \eqref{def2R1}. We shall prove this in the weakened context in Section \ref{braided2cat}.

\subsection{Braided 2-quasi-bialgebras; the modified hexagon relations}\label{braided2cat}
We now wish to keep track of the interplay between the fusion associators $a$ and the braiding maps $b$  --- or, algebraically, the coassociator and the 2-$R$-matrix --- on $\operatorname{2Rep}^\cT(\cG)$. We shall do this by revisiting the fundamental characterization of 2-R-matrices in Section \ref{quasitrihopf}. In other words, we are prompted to study the {\it weak} 2-quantum double $D(\cG,\cG)$ and its braided transposition $\Psi$.

Fix the weak 2-bialgebra $\cG$. Despite the skeletal construction in Section \ref{weakskeletalqd}, we are able to form $D(\cG,\cG)$ here \textit{without} assuming skeletality, since we know exactly how $\cG$ acts on itself by weak 2-representations --- in the canonical way according to {\bf Definition \ref{weak2alg}}. This fact also allows us to identify $\cT_D$ as merely several copies of the 3-cocycle $\cT$ on $\cG$, and in particular the components $\rhd_1=\lhd_1 = \cT$ are equal. 

\medskip

To proceed, we recall in the factorizable case that \textit{associativity} of $\cK\cong D(\cG,\cG)$ is invoked to deduce the braiding relation  \eqref{2yb1} for the braided transposition $\Psi$. This associativity is, of course, witnessed by the 3-cocycle $\cT$. Combined with the understanding from Section \ref{quasitrihopf} that the braided transposition gives rise to the canonical \textit{dual} 2-$R$-matrix, we then deduce the following.
\begin{definition}
    A \textbf{2-$R$-matrix} $\cR$ for a \textit{weak} 2-bialgebra $(\cG,\Delta,\Delta_1)$ is defined as in \textbf{Definition \ref{2Rdef}}, except that \eqref{def2R1} is modified by the {\it dual} of $\cT$ --- ie. the coassociator $\Delta_1$ --- namely, we have
\begin{gather}
    D_t\Delta_1(x)_{231}\cdot( 1\otimes\Delta_0)\cR \cdot D_t\Delta_1(x)_{123} = \cR_{13}\cdot D_t\Delta_1(x)_{213}\cdot \cR_{12},\nonumber\\
    D_t\Delta_1(x)_{312}^{-1}\cdot (\Delta_0\otimes 1)\cR\cdot D_t\Delta_1(x)_{213}^{-1} = \cR_{13}\cdot D_t\Delta_1(x)_{132}^{-1}\cdot \cR_{23},\label{quasi2rmatrix}
\end{gather}
for each $x\in \cG_0$. 
\end{definition}
\noindent This condition bears a striking resemblance to the defining relations of a {\it braided quasi-bialgebra} \cite{book-quasihopf}; indeed, applying the double-$t$-map $D_t[2]$ to  \eqref{quasi2rmatrix} yields, by definition  \eqref{deg0rmat}, \eqref{fusionassoc4},
\begin{equation}
    \Phi_{231}(x) (1\otimes\Delta_0')R\Phi_{123}(x) = R_{13} \Phi_{213}(x) R_{12},\qquad \Phi_{312}^{-1}(x) (\Delta_0'\otimes 1)R  \Phi_{213}^{-1}(x) = R_{13} \Phi_{132}^{-1}(x)R_{23},\label{quasirmatrix}
\end{equation}
which is precisely a braided quasi-bialgebra structure at degree-0 $(\cG_0,\Delta_0',R,\Phi)$; see {\it Remark \ref{quasihopf}}. This motivates the following definition.

\begin{definition}
A {\bf braided 2-quasi-bialgebra}\footnote{Note  that a \textit{quasi} 2-bialgebra, as opposed to a 2-quasi-bialgebra here, refers to a weak 2-bialgebra with trivial 3-cocycle $\cT=0$ but non-trivial coassociator $\Delta_1$. } $(\cG,\Delta=(\Delta_1,\Delta_0, \Delta_{-1}),\cT,\cR)$ is a weak 2-bialgebra equipped with a 2-$R$-matrix $\cR$ and a coassociator $\Delta_1:\cG_0\rightarrow \cG_{-1}^{3\otimes}$ such that  \eqref{quasi2rmatrix}, \eqref{quasirmatrix}, \eqref{def2R2} and \eqref{2ybequiv} hold.
\end{definition}

Similar to  \eqref{str2hexag}, by applying {\it strict} 2-representations $\rho = (\rho_1,\rho_0)$ to  \eqref{quasi2rmatrix}, we obtain:
\begin{lemma}\label{braiddecomp}
For each $X\in\operatorname{2Rep}^\cT(\cG)$, we have the decompositions (the {\bf hexagon relations})
\begin{eqnarray}
    \begin{cases}b_{(VW)X} = a_{XVW} \circ b_{VX} \circ a_{VXW}^\dagger \circ b_{WX}\circ a_{VWX} \\ b_{V(WX)} = a_{WXV}^\dagger \circ b_{VX} \circ a_{WVX} \circ b_{VW} \circ a_{VWX}^\dagger \end{cases}\iff \text{ \eqref{quasirmatrix}},\label{hexag}\\
    \begin{cases}b_{(Vj)X}= a_{XVj} \ast \id_{b_{VX}} \ast a_{VXj}^\dagger \ast b_{jX}\ast a_{VjX} \\ b_{V(jX)} = a_{jXV}^\dagger \ast \id_{b_{VX}} \ast a_{jVX} \ast b_{Vj}\ast a_{VjX}^\dagger \end{cases} \iff \text{apply $D_t^+$ to  \eqref{quasi2rmatrix}},\nonumber\\
    \begin{cases}b_{(iW)k} = a_{kiW} \ast b_{ik} \ast a_{ikW}^\dagger \ast b_{Wk}\ast a_{iWk} \\ b_{i(Wk)} = a_{Wki}^\dagger\ast b_{ik} \ast a_{Wik} \ast b_{iW} \ast a_{iWk}^\dagger\end{cases} \iff \text{ \eqref{quasi2rmatrix}},\nonumber
\end{eqnarray}
as 1-/2-morphisms, and similarly for all the other possible braiding maps on tensor products.
\end{lemma}
\noindent The decomposition formula for $b_{ijk}$ follows from these, as well as the fact that $b_{ij},a_{ijk}$ are all determined by the mixed braiding/associators.

The 2-morphism $b_{(iW)X}$, for instance, can be expressed in terms of the following composition diagram
\begin{equation}
    \begin{tikzcd}
(VW)X \arrow[d, "i"'] \arrow[r, "a_{VWX}"] & V(WX) \arrow[d, "i"] \arrow[r, "b_{WX}"] & V(XW) \arrow[r, "a_{VXW}^\dagger"] \arrow[d, "i"] & (VX)W \arrow[r, "b_{VX}"] \arrow[d, "i"] & (XV)W \arrow[r, "a_{XVW}"] \arrow[d, "i"] & X(VW) \arrow[d, "i"] \\
(UW)X \arrow[r, "a_{UWX}"]                 & V(WX) \arrow[r, "b_{WX}"]                & U(XW) \arrow[r, "a_{UXW}^\dagger"]                & (UX)W \arrow[r, "b_{UX}"]                & (XU)W \arrow[r, "a_{UVW}"]                & X(UW)               
\end{tikzcd},\label{mixedhexag}
\end{equation}
which has also appeared in \cite{KongTianZhou:2020}. This establishes most of the structural properties of $\operatorname{2Rep}^\cT(\cG)$ as a braided 2-category, and the final ingredient to introduce is the {\it hexagonator}.

\subsection{The braiding hexagonator: weak 2-representations of a braided 2-quasi-bialgebra}\label{hexagona}
We obtained the decomposition {\bf Lemma \ref{braiddecomp}} by applying a strict 2-representation to  \eqref{quasi2rmatrix}. However, as we have noted previously in {\it Remark \ref{weak2repstrict}}, 2-representations of a weak 2-bialgebra $(\cG,\cT)$ {\it cannot} be strict, even when $\cG$ is skeletal. As such, we must take into account the additional component $\varrho: \cG_0^{2\otimes}\rightarrow\End(V)_{-1}$ when deriving the decompositions above (in particular  \eqref{hexag}).

For the rest of the paper, it suffices to consider the case $t=0$ or $t=\eta_0$, the constant map to the unit $\eta_0\in \cG_0$. Since $\varrho$ is normalized and the second and third equations in  \eqref{weak2repalg} involve {\it pre}-composing $\varrho$ with $t$, the only non-trivial relation is
\begin{equation}
    \rho_0(xx') - \rho_0(x)\rho_0(x') = \delta\varrho(x,x'),\qquad x,x'\in\cG_0,\nonumber
\end{equation}
where we recall that $\delta:\End(V)_{-1}\rightarrow\End(V)_0$ is the $t$-map on the weak endomorphism 2-algebra. Therefore, in order to obtain the decomposition of the form  \eqref{hexag} from  \eqref{quasirmatrix}, we must keep track of the terms involving $\varrho$ that appear. For instance, we have
\begin{equation}
    \rho_0^{3\otimes}(R_{13} \Phi_{213}) - \rho_0^{3\otimes}(R_{13})\rho_0^{3\otimes}(\Phi_{213}) = (\delta\varrho)^{3\otimes}(R_{13}, \Phi_{213}),\nonumber
\end{equation}
in which we notice that the second term on the left-hand side is the composition $b_{VU}\circ a_{WVU}$. 

More explicitly, translating  \eqref{quasirmatrix} to  \eqref{hexag} comes at a price given by a cochain homotopy 
\begin{eqnarray}
    \Omega_{V|WU}(x) &=& (\varrho_V\otimes\varrho_W\otimes\varrho_U)(\Phi_{231}(x) ,(1\otimes\Delta_0')R\Phi_{123}(x)) \nonumber\\
    &\qquad& -~ (\varrho_V\otimes\varrho_W\otimes\varrho_U)(R_{13} , \Phi_{213}(x)R_{12}) \nonumber\\
    &\qquad& +~ (\varrho_V\otimes\varrho_W\otimes\varrho_U)((1\otimes\Delta_0')R,\Phi_{123}(x)) \nonumber\\
    &\qquad& -~ (\varrho_V\otimes\varrho_W\otimes\varrho_U)(\Phi_{213}(x),R_{12})
    \label{hexhomotop}
\end{eqnarray}
between the two sides of  \eqref{quasirmatrix} for each $x\in\cG_0$, and similarly its adjoint $\Omega_{V|WU}^\dagger$. We thus have the following diagrams
\begin{eqnarray}
\begin{tikzcd}
(V_{-1}\otimes W_{-1}) \otimes U_{-1} \arrow[rr, "\partial^{3\otimes}"] \arrow[dd] &  & (V_0\otimes W_0) \otimes U_0 \arrow[dd] \arrow[lldd, "\Omega_{V|WU}"'] \\
                                                                                                                                                                                                  &  &                                                                                                                                                             \\
W_{-1}\otimes (U_{-1}\otimes V_{-1}) \arrow[rr, "\partial^{3\otimes}"']                                                                                                    &  & W_0\otimes(U_0\otimes V_0)                                                                                                                                  
\end{tikzcd},\nonumber \\
\begin{tikzcd}
      V_{-1}\otimes (W_{-1} \otimes U_{-1}) \arrow[rr, "\partial^{3\otimes}"] \arrow[dd] &  & V_0\otimes (W_0 \otimes U_0) \arrow[dd
      ] \arrow[lldd, "\Omega^\dagger_{V|WU}"'] \\
                                                                                                                                                                                                  &  &                                                                                                                                                             \\
(U_{-1}\otimes V_{-1})\otimes W_{-1} \arrow[rr, "\partial^{3\otimes}"']                                                                                                    &  & (U_0\otimes V_0) \otimes W_0                                                                                                                                 
\end{tikzcd},\nonumber
\end{eqnarray}
where the vertical arrows denote the decomposition  \eqref{hexag}. These diagrams cast $\Omega,\Omega^\dagger$ as the {\bf hexagonator} 2-morphisms in $\operatorname{2Rep}^\cT(\cG)$:
\begin{eqnarray}
\begin{tikzcd}
                                                                                  & V(W U) \arrow[rr, "b_{V(WU)}"] &  & (W U) V \arrow[rd, "a_{WUV}"] &                                  \\
(V W) U \arrow[ru, "a_{VWU}"] \arrow[rd, "b_{VW}"'] & {} \arrow[rr, "\Omega_{V|WU}", Rightarrow]                             &  & {}                                                          & W (U V) \\
                                                                                  & (W V) U \arrow[rr, "a_{WVU}"']    &  & W (V U) \arrow[ru, "b_{VU}"']      &                                 
\end{tikzcd},\nonumber \\
\begin{tikzcd}
                                                                                  & (V W) U \arrow[rr, "b_{(VW)U}"] &  & U (V W) \arrow[rd, "a^\dagger_{UVW}"] &                                  \\
V (W U) \arrow[ru, "a^\dagger_{VWU}"] \arrow[rd, "b_{WU}"'] & {} \arrow[rr, "\Omega^\dagger_{V|WU}", Rightarrow]                             &  & {}                                                          & (U V)  W\\
                                                                                  & V (U W) \arrow[rr, "a^\dagger_{VUW}"']    &  & (V U)  W \arrow[ru, "b_{VU}"']      &                                 
\end{tikzcd}.\nonumber
\end{eqnarray}
In other words, the quantities $\Omega_{V|WU},\Omega_{V|WU}^\dagger$ by definition is an invertible modification implementing the two sides of the decomposition  \eqref{hexag}. 

Now by the diagram  \eqref{mixedhexag}, the 2-intertwiners $i:V\rightarrow U$ and their associated mixed braiding maps $b_{iW}$ preserve these hexagon relations. This leads to the naturality of the hexagonator $\Omega_{V|WU}$ with respect to 2-intertwiners such that we have (cf. diagram (2.2) in \cite{KongTianZhou:2020})
\begin{equation}
\Omega_{V|WX}=\begin{tikzcd}
                                                                                     & V (W X) \arrow[rrr] \arrow[rdd, "i"']              &                                                      &                                             & (W X) V \arrow[rddd] \arrow[ldd, "i"]            &                                       \\
                                                                                     &                                                                  & {} \arrow[r, "b_{i(W X)}", Rightarrow]        & {}                                          &                                                                &                                       \\
                                                                                     &                                                                  & U (W X) \arrow[r, "b_{U(W X)}"] & (W X) U \arrow[rd, "a_{WXU}"] &                                                                &                                       \\
(V W) X \arrow[ruuu, "a_{VWX}"] \arrow[rddd, "b_{VW}"'] \arrow[r, "i"] & (U W) X \arrow[ru, "a_{UWX}"] \arrow[rd, "b_{UW}"] & {} \arrow[r, " \Omega_{U|WX}",Rightarrow]            & {}                                          & W (X U)                                          & W (X V) \arrow[l, "i"'] \\
                                                                                     & {} \arrow[d, "b_{iW}", Rightarrow]                               & (W U) X \arrow[r, "a_{WUX}"]           & W (U X) \arrow[ru, "b_{UX}"'] & {} \arrow[d, "b_{iX}", Rightarrow]                             &                                       \\
                                                                                     & {}                                                               &                                                      &                                             & {}                                                             &                                       \\
                                                                                     & (W V) X \arrow[rrr, "a_{WVX}"] \arrow[ruu, "i"]    &                                                      &                                             & W (V X) \arrow[ruuu, "b_{VX}"'] \arrow[luu, "i"] &                                      
\end{tikzcd},\nonumber
\end{equation}
and similarly for the adjoint diagrams with $\Omega^\dagger$. The tensor product $VX$ of 2-representations is equipped with the tensor product $\Omega_{VX|WU}$ hexagonator, which are by construction natural and invertible.

\begin{remark}\label{pentagonator}
Notice we did not define any associators for the {2-morphisms} $\mu$ in $\operatorname{2Rep}^\cT(\cG)$. This is because 2-morphisms in a 2-category the tensor product $\mu\otimes\nu = \mu\ast \nu$ given by composition is {\it strictly} associative; indeed, such an associator $a_{\mu\nu\lambda}:(\mu\nu)\lambda \Rrightarrow \mu(\nu\lambda)$ would have to be a 3-morphism. 

By the same token, the hexagon relations involving the mixed braiding maps (ie. the decompositions in {\bf Lemma \ref{braiddecomp}} aside from  \eqref{hexag}), as well as the pentagon relations for the associator 2-morphisms  \eqref{fusionassoc1}, \eqref{fusionassoc2}, \eqref{fusionassoc3}, must hold strictly on-the-nose. However, the fact that $a_{VWU}$ is a 1-morphism implies we can have a 2-morphism $\pi$, called the {\bf pentagonator}, that implements its pentagon relation. We will show in Appendix \ref{weak2repthy} that $\pi$ is given by the Hochschild 3-cocycle $\T$ attached to the weak endomorphism 2-algebra $\End(V)$.
\end{remark}

\subsection{Main theorem and its proof}\label{proof}
We are finally ready to state and prove the main theorem. As earlier, we will often omit the tensor products to lighten the notations.  
\begin{theorem}\label{mainthm2}
The 2-representation 2-category $\operatorname{2Rep}^\cT(\cG)$ of a weak quasitriangular 2-bialgebra $\cG$ is a braided monoidal 2-category with trivial left-/right-equivalences $l:1V \xrightarrow{\sim} V$, $r: V1 \xrightarrow{\sim} V$.
\end{theorem}
We will prove  this by using algebraic and diagrammatic manipulations that we have outlined throughout the paper, and reproduce all the coherence relations defining a braided monoidal 2-category in \cite{GURSKI20114225}. On the way, we shall also construct quantities that has also appeared in \cite{KongTianZhou:2020}. 

\medskip

Recall first that, from Section \ref{fusion2rep2cat}, we have trivial left- and right-unitors 
$l:1V\rightarrow V,r:V1\rightarrow V$, and hence all coherence relations involving them (ie. diagrams (2.5), (2.7)-(2.9) of \cite{KongTianZhou:2020}) are vacuously satisfied.

\paragraph{Braiding on the associator; the third Gurski axiom.} Let $V,W,U\in\operatorname{2Rep}^\cT(\cG)$ be four 2-representations. Consider the mixed braiding 2-morphism $b_{a_{VWU}X}$, which by {\bf Lemma \ref{braiding2rep2cat2}} fits into a diagram of the form
\begin{equation}
    \begin{tikzcd}
((VW)U)X \arrow[rr, "b_{((VW)U)X}"] \arrow[dd, "a_{VWU}"] &                            & X((VW)U) \arrow[dd, "a_{VWU}"] \\
                                                          & \xRightarrow{b_{a_{VWU}X}} &                                \\
(V(WU))X \arrow[rr, "b_{(V(WU))X}"]                       &                            & X(V(WU))                      
\end{tikzcd}.\nonumber
\end{equation}
{\bf Lemma \ref{braiddecomp}} states that we can in fact decompose the top and bottom 1-morphisms in this diagram, provided we keep in mind the hexagonator $\Omega,\Omega^\dagger$  \eqref{hexhomotop} that appears in doing so. We thus obtain a formula of the form
\begin{eqnarray}
    b_{((VW)U)X}&\xRightarrow{\Omega_{(VW)|UX}^\dagger}& a_{X(VW)U}\circ b_{(VW)X} \circ a_{(VW)XU}^\dagger \circ b_{UX}\circ a_{(VW)UX} \nonumber\\
    &\xRightarrow{\Omega_{V|WX}^\dagger}& a_{X(VW)U}\circ \left[a_{XVW}\circ b_{VX} \circ a_{VXW}^\dagger \circ b_{WX} \circ a_{VWX}\right] \nonumber\\
    &\qquad& \circ ~a_{(VW)XU}^\dagger \circ b_{UX}\circ a_{(VW)UX},\label{2.6top}
\end{eqnarray}
and similarly for the bottom 1-morphism $b_{(V(WU))X}$,
\begin{eqnarray}
    b_{(V(WU))X}&\xRightarrow{\Omega_{V|(WU)X}^\dagger}& a_{XV(WU)}\circ b_{VX} \circ a_{VX(WU)}^\dagger \circ b_{(WU)X}\circ a_{V(WU)X} \nonumber\\
    &\xRightarrow{\Omega_{W|UX}^\dagger}& a_{XV(WU)}\circ b_{VX}\circ a_{VX(WU)}^\dagger\nonumber\\
    &\qquad& \circ~ \left[a_{XWU}\circ b_{WX} \circ a_{WXU}^\dagger \circ b_{UX} \circ a_{WUX}\right] \circ a_{V(WU)X}.\label{2.6bot}
\end{eqnarray}
Now notice that there are three identical braiding maps that appear in both of these formulas, $b_{VX},b_{WX},b_{UX}$, but they act on objects that differ by an associator: we have $b_{UX}: (VW)(UX)\rightarrow (VW)(XU)$ from  \eqref{2.6top} and $b_{UX}: V(W(UX))\rightarrow V(W(XU))$ from  \eqref{2.6bot}, for instance. Such a square is precisely given by the diagram  \eqref{assoc2morph},
\begin{equation}
    \begin{tikzcd}
(VW)(UX) \arrow[rr, "a_{VW(UX)}"] \arrow[dd, "b_{UX}"'] &                                   & V(W(UX)) \arrow[dd, "b_{UX}"] \\
                                                            & \xRightarrow{a_{VWb_{UX}}} &                            \\
(VW)(XU) \arrow[rr, "a_{VW(XU)}"]                  &                                   & V(W(XU))                
\end{tikzcd},\nonumber
\end{equation}
and similarly for the other braiding maps that occur in both  \eqref{2.6bot}, \eqref{2.6top}.

Putting this all together, by successively decomposing the braiding maps, we achieve the following diagrammatic expression for $b_{a_{VWU}X}$ (here we only label the 2-morphisms for clarity):
\begin{equation}
    \hspace*{-1.5cm}
\begin{tikzcd}
	{((VW)U)X} &&&&& {X((VW)U)} \\
	{(VW)(UX)} & {(VW)(XU)} & {((VW)X)U} &&& {(X(VW))U} \\
	&& {(V(WX))U} & {(V(XW))U} & {((VX)W)U} & {((XV)W)U} \\
	{V(W(UX))} & {V(W(XU))} & {V((WX)U)} & {V((XW)U)} \\
	{V((WU)X)} &&& {V(X(WU))} & {(VX)(WU)} & {(XV)(WU)} \\
	{(V(UW))X} &&&&& {X(V(UW))}
	\arrow[""{name=0, anchor=center, inner sep=0}, from=1-1, to=1-6]
	\arrow[from=1-1, to=2-1]
	\arrow[from=2-1, to=2-2]
	\arrow[from=2-2, to=2-3]
	\arrow[""{name=1, anchor=center, inner sep=0}, from=2-3, to=2-6]
	\arrow[from=1-6, to=2-6]
	\arrow[from=2-3, to=3-3]
	\arrow[from=3-3, to=3-4]
	\arrow[""{name=2, anchor=center, inner sep=0}, from=3-4, to=3-5]
	\arrow[from=3-5, to=3-6]
	\arrow[from=2-6, to=3-6]
	\arrow[from=1-6, to=2-6]
	\arrow[""{name=3, anchor=center, inner sep=0}, from=2-1, to=4-1]
	\arrow[""{name=4, anchor=center, inner sep=0}, from=3-6, to=5-6]
	\arrow[from=5-6, to=6-6]
	\arrow[from=4-1, to=5-1]
	\arrow[from=5-1, to=6-1]
	\arrow[from=4-1, to=4-2]
	\arrow[""{name=5, anchor=center, inner sep=0}, from=4-2, to=4-3]
	\arrow[from=4-3, to=4-4]
	\arrow[from=4-4, to=5-4]
	\arrow[from=5-4, to=5-5]
	\arrow[from=5-5, to=5-6]
	\arrow[""{name=6, anchor=center, inner sep=0}, from=6-1, to=6-6]
	\arrow[""{name=7, anchor=center, inner sep=0}, from=3-3, to=4-3]
	\arrow[""{name=8, anchor=center, inner sep=0}, from=3-4, to=4-4]
	\arrow[""{name=9, anchor=center, inner sep=0}, from=2-2, to=4-2]
	\arrow[""{name=10, anchor=center, inner sep=0}, from=3-5, to=5-5]
	\arrow[""{name=11, anchor=center, inner sep=0}, from=5-1, to=5-4]
	\arrow["{a_{VWb_{UX}}}", shorten <=11pt, shorten >=11pt, Rightarrow, from=3, to=9]
	\arrow["{\Omega_{V|WX}^\dagger}"', shorten <=4pt, shorten >=4pt, Rightarrow, from=1, to=2]
	\arrow["{a_{Vb_{WX}U}}"{description}, shorten <=11pt, shorten >=11pt, Rightarrow, from=7, to=8]
	\arrow["{a_{b_{VX}WU}}", shorten <=11pt, shorten >=11pt, Rightarrow, from=10, to=4]
	\arrow["{\Omega_{W|UX}^\dagger}", shorten <=4pt, shorten >=4pt, Rightarrow, from=11, to=5]
	\arrow["{\Omega_{V|(WU)X}^\dagger}"{description}, shorten <=11pt, shorten >=11pt, Rightarrow, from=6, to=11]
	\arrow["{\Omega_{(VW)|UX}^\dagger}"{description}, shorten <=11pt, shorten >=11pt, Rightarrow, from=0, to=1]
\end{tikzcd}\nonumber
\end{equation}
This is precisely the third axiom in \cite{GURSKI20114225}; cf. diagram (2.6) in \cite{KongTianZhou:2020}.

\paragraph{The hexagonator $\Omega_{VX|WU}$; the third Gurski axiom.} We shall apply the same procedure as above to expand the defining diagram for $\Omega_{VX|WU}$,
\begin{equation*}
    \begin{tikzcd}
                                                                                  & (VX)(W U) \arrow[rr, "b_{(VX)(WU)}"] &  & (W U) (VX) \arrow[rd, "a_{WU(VX)}"] &                                  \\
((V X)W) U \arrow[ru, "a_{(VX)WU}"] \arrow[rd, "b_{(VX)W}"'] & {} \arrow[rr, "\Omega_{VX|WU}", Rightarrow]                             &  & {}                                                          & W (U (VX)) \\
                                                                                  & (W (VX)) U \arrow[rr, "a_{W(VX)U}"']    &  & W ((VX) U) \arrow[ru, "b_{(VX)U}"']      &                                 
\end{tikzcd}.
\end{equation*}
By rewriting each of the associator and braiding maps appearing here using \eqref{braiddecomp} and the pentagonator $\pi$ \eqref{pentag2rep} introduced in Appendix \ref{weak2repthy}, we obtain precisely diagram (2.4) in \cite{KongTianZhou:2020} as the $\Omega_{VX|WU}$. The third axiom of \cite{GURSKI20114225} then follows.

\paragraph{Iterating the braiding map; the fourth Gurski axiom.} Now consider the iterated braiding 2-morphism $b_{Vb_{UW}}$  \eqref{iteratedbraid}. By the same logic as above, we can use the decomposition  \eqref{hexag} once again on the top and bottom braiding morphisms that appear in the diagram,
\begin{eqnarray}
    b_{V(UW)} &\xRightarrow{\Omega_{V|UW}}& a_{UWV}^\dagger \circ b_{VW} \circ a_{UVW} \circ b_{VU}\circ a_{VUW}^\dagger,\nonumber\\
    b_{V(WU)} &\xRightarrow{\Omega_{V|WU}}& a_{WUV}^\dagger \circ b_{VU} \circ a_{WVU} \circ b_{VW}\circ a_{VWU}^\dagger.\nonumber
\end{eqnarray}

We can thus form the composition
\begin{equation}
    b_{\Omega_V|WU}\equiv \Omega_{V|WU}^{-1} \cdot b_{Vb_{UW}} \cdot \Omega_{V|UW}, \label{braidmorph}
\end{equation}
which fits into a diagram that "pastes" two hexagon diagrams together,
\begin{equation}
\hspace*{-1cm}
\begin{tikzcd}
W(VU) \arrow[dd, "b_{VU}"'] & (WV)U \arrow[l, "a_{WVU}"]  & (VW)U \arrow[r, "a_{VWU}"] \arrow[l, "b_{VW}"'] & V(WU) \arrow[ldd, "b_{V(WU)}"', bend right]     & V(UW) \arrow[ldd, "b_{V(UW)}", bend left] \arrow[l, "b_{UW}"'] &                            & (VU)W \arrow[dd, "b_{VU}"] \arrow[ll, "a_{VWU}"] \arrow[llllll, "b_{(VU)W}", bend right] \\
                            & \xRightarrow{\Omega_{V|WU}} &                                                 & \xLeftarrow{b_{Vb_{UW}}}                       &                                                                & \xLeftarrow{\Omega_{V|UW}} &                                                                                          \\
W(UV)                       &                             & (WU)V \arrow[ll, "a_{WUV}"']                    & (UW)V \arrow[r, "a_{UWV}"'] \arrow[l, "b_{UW}"] & U(WV)                                                          & U(VW) \arrow[l, "b_{VW}"]  & (UV)W \arrow[l, "a_{UVW}"'] \arrow[llllll, "b_{(UV)W}"', bend left]                      
\end{tikzcd}\nonumber
\end{equation}
Note that, by construction  \eqref{braidmorph}, the 2-morphisms $b_{\Omega_V|\bullet\bullet}$ are natural and {invertible}. Moreover, its definition is precisely (2.10) in \cite{KongTianZhou:2020}, and hence the fourth axiom of \cite{GURSKI20114225} follows.

\paragraph{Hochchild descent; the second Gurski axiom.} Let us now focus on  \eqref{weak2rephom}. Recall that it states, for $x_1,x_2,x_3\in\cG_0$, that
\begin{eqnarray}
    \rho_1(\cT(x_1,x_2,x_3)) -\T(\rho_0(x_1),\rho_0(x_2),\rho_0(x_3)) &=& \rho_0(x_1)\cdot \varrho(x_2,x_3) - \varrho(x_1x_2,x_3) \nonumber\\
    &\qquad& +~ \varrho(x_1,x_2x_3) - \varrho(x_1,x_2)\cdot\rho_0(x_3),\nonumber
\end{eqnarray}
where $\T$ is the Hochschild 3-cocycle on the weak endomorphism 2-algebra $\End(V)$ of a particularly chosen weak 2-vector space $V\in\mathsf{2Vect}^{hBC}$. We shall now specialize $x_1,\dots,x_3$ to the elements in $\cG_0$ of \eqref{quasirmatrix}, and let the equation act on $V$. 

By some computations, we see that the right-hand side translates to the composition of 2-morphisms
\begin{equation}
    \id_{\id_W}\Omega_{V|UX} \ast \Omega_{V|W(UX)} \ast \Omega_{V|(WU)X}^{-1} \ast (\Omega_{V|WU}\id_{\id_X})^{-1},\nonumber  
\end{equation}
while on the term $\rho\circ \cT$ on the left dualizes to terms of the form $(\rho_V\otimes\dots\otimes\rho_X)(\Delta_1\circ R - \cR\circ D_t^+ \Delta_1)$, which translates to
\begin{equation}
   a_{Wb_{VU}X}^\dagger \ast a_{b_{VW}UX} \ast ~a_{WUb_{VX}}^\dagger\ast  b_{Va_{WUX}}.\nonumber
\end{equation}
Now by leveraging the result {\bf Proposition \ref{pentag}} in Appendix \ref{weak2repthy}, the term $\T\circ\rho_0^{3\otimes}$ in fact defines the pentagonators $\pi$ on $\operatorname{2Rep}^\tau(\cG)$. The left-hand side then acquires also the contribution
\begin{equation}
    \pi_{WVUX} \ast \pi_{WUVX} \ast \pi_{VWUX}^\dagger \ast \pi_{WUXV}^\dagger, \nonumber
\end{equation}
where $\pi_{WUXV}(x) = \T((\rho_W)_0(x),(\rho_U)_0(x),(\rho_X)_0(x))(V)$; see  \eqref{pentag2rep}.

Altogether, this gives rise to the equation
\begin{eqnarray}
    \pi_{WVUX} \ast \pi_{WUVX} \ast \id_{\id_W}\Omega_{V|UX} \ast \Omega_{V|W(UX)} \ast a_{b_{VW}UX}^\dagger \ast b_{Va_{WUX}}^\dagger = \nonumber\\
    \pi_{VWUX} \ast \pi_{WUXV}\ast \Omega_{V|WU}\id_{\id_X} \ast \Omega_{V|(WU)X} \ast a_{Wb_{VU}X}^\dagger \ast a_{WUb_{VX}}^\dagger \nonumber
\end{eqnarray}
for $V,W,U,X\in\operatorname{2Rep}^\cT(\cG)$, which is precisely the second axiom in \cite{GURSKI20114225} (or equivalently axiom (2.1) in \cite{KongTianZhou:2020}). In \cite{KongTianZhou:2020}, this axiom was also captured in a cohomological manner in (3.2) there, but where the adjoint equivalences $a_{iVW},a_{iVW}^\dagger$ are omitted.

\medskip

In summary, we find that $\operatorname{2Rep}^\cT(\cG)$ has the following ingredients:
\begin{center}
        \begin{tabular}{|c|c|c|}
    \hline
         objects & 1-morphisms & 2-morphisms \\
         \hline 
         2-representations & 2-intertwiners & $\substack{\text{equivariant} \\ \text{cochain homotopies}}$ \\ 
         $(V,b_{V\bullet},\Omega_{V|\bullet\bullet})$ & $(i,b_{i\bullet})$ & $\mu$  \\
         \hline
         
    \end{tabular}
\end{center}
\noindent This establishes {\bf Theorem \ref{mainthm2}} and concludes our paper.

\section{Conclusion}
In this paper, we have given a construction of a categorified notion of a quantum double suited for 2-groups/2-algebras \cite{Wagemann+2021}. This was accomplished by naturally lifting Majid's quantum double construction \cite{Majid:1994nw,Majid:1996kd} to internal categories, which makes the structural theorems manifest. In particular, we have provided explicit algebraic computations that demonstrate concretely the notion of duality between 2-bialgebras, and how two 2-bialgebras can be "pasted together" through the notion of 2-representations \cite{Angulo:2018}. We have also given examples which demonstrate that the category of bialgebras (and their quantum doubles) embeds into the category of 2-bialgebras (and their 2-quantum doubles), generalizing an analogous statement proven in \cite{Bai_2013} for the classical case. 

By endowing Baez-Crans 2-vector spaces $\mathsf{2Vect}^{BC}$ with higher homotopical data, we described the {\it weak} 2-representation theory for 2-bialgebras $\cG$. We give a concrete description of the 2-category $\operatorname{2Rep}^\cT(\cG)$, which by definition comes with a forgetful 2-functor $\operatorname{2Rep}(\cG)\rightarrow \mathsf{2Vect}^{hBC}$ into the $k$-linear 2-category $\mathsf{2Vect}^{hBC}$ of such "deformed" Baez-Crans 2-vector spaces. Though we had provided an operational description of $\mathsf{2Vect}^{hBC}$ --- in particular the property that its endomorphism categories and algebra objects are modelled by 2-term $A_\infty$-algebras --- its precise construction shall appear soon in a later work. 

As in the case of 1-Hopf algebras \cite{Majid:1996kd}, the monoidal structure on $\operatorname{2Rep}^\cT(\cG)$ is controlled by the coproduct and the coassociator, whence the naturality and the $\mathsf{Gray}$-property of the tensor product follow from the coequivariance and coPeiffer identities ({\bf Lemmas \ref{monoidal2rep2cat}, \ref{2repdecomp}}). We have also introduced the 2-$R$-matrix $\cR$ of $\cG$, which was defined naturally from the properties of the 2-quantum double $D(\cG,\cG)$. We show that the resulting braiding on $\operatorname{2Rep}^\cT(\cG)$ is in fact natural and coherent ({\bf Lemmas \ref{braidingmorphs}, \ref{braiding2rep2cat2}}). 

Then, together with the structure of weak 2-representations, we identified the pentagonator and hexagonator 2-morphisms and exhibited all the necessary coherence diagrams to prove {\bf Theorem \ref{mainthm2}}. This gives a direct correspondence of the ingredients of a braided monoidal 2-category \cite{GURSKI20114225,KongTianZhou:2020} with those of an underlying weak quasitriangular 2-bialgebra $\cG$.

\medskip

We also note that {\bf Theorem \ref{mainthm2}} hints towards a (braided) {\it higher Tannakian reconstruction} (cf. \cite{Pfeiffer2007}). It should state that, morally, given a sufficiently "nice" braided 2-category (such as the Drinfel'd centre $Z_1(\mathcal{D})$ of a monoidal 2-category $\mathcal{D}$ \cite{KongTianZhou:2020} or a \textit{fusion} 2-category \cite{Douglas:2018}) and a fibre 2-functor $F:\cC\rightarrow\mathsf{2Vect}^{hBC}$, there is a braided equivalence $\mathcal{C} \simeq \operatorname{2Rep}^\cT(\cG)$ such that the diagram
\begin{equation}
    \begin{tikzcd}
\cC \arrow[rd, "\rotatebox{-20}{$\sim$}"'] \arrow[rr, "F"] &                                         & \mathsf{2Vect}^{hBC} \\
                                                         & \operatorname{2Rep}^\cT(\cG) \arrow[ru] &                     
\end{tikzcd}\nonumber
\end{equation}
commutes. The authors are aware that efforts towards such a result in the semisimple setting are currently being undertaken. Together with this work, it allows to distill all coherences in $\cC$ into a sequence of Hochschild cohomological descent equations between the weak 2-bialgebra $(\cG,\cT,\cR)$ and the weak 2-endormophism 2-algebra $\End(V)$.

Furthermore, as mentioned in the Introduction, such a braided equivalence would also allow us to reconstruct the topological field theory associated to a 4D gapped topological phase, described by a braided 2-category \cite{Johnson-Freyd:2020}. The work towards this goal, in the case of the 3D toric code as well as its spin-$\bbZ_2$ variant, is in \cite{Chen2z:2023}.
 
\medskip
 
We will show in Appendix \ref{weak2repthy} that our framework is able to produce the $k$-invariants of the 2-representation theory for a skeletal 2-group $G$ \cite{Delcamp:2023kew,Douglas:2018}, which demonstrates that this work does in fact remedy the issues mentioned at the beginning of Section \ref{weak2bialgebras} from \cite{heredia2016representations2groupsbaezcrans2vector}. 

\paragraph{Relations with existing quantum group categorifications.} Following the main text, we introduce the notion of (strict) Hopf 2-algebras in Appendix \ref{2-hopf}. There have been numerous proposals for the notion of "categorified Hopf algebras" in the literature, such as the trialgebra proposal of Pfeiffer \cite{Pfeiffer2007}, the quantum 2-group of Majid \cite{Majid:2012gy}, or the quantum groupoid of Lu \cite{Lu:1996}. It is well-known that group crossed-modules are equivalent to 2-groups and group groupoids \cite{Chen:2012gz}, and the $\text{cat}^1$-Hopf algebras of Wagemann \cite{Wagemann+2021} coincide with trialgebras in the cocommutative case \cite{Pfeiffer2007}. Hence our definition of \textit{strict} 2-bialgebras are very closely related to all of these alternative formulations. 

On the other hand, the content of {\bf Theorem \ref{mainthm}} has been explored in the context of Hopf categories of Gurski \cite{gurski2006algebraic} by Neuchl \cite{neuchl1997representation}. Hence, we posit that the relationship between our \textit{weak} 2-bialgebras and these Hopf categories should be understood through a monoidal "model change" from 2-vector spaces of the "homotopy/weak Baez-Crans" sense $\mathsf{2Vect}^{hBC}$ to the Kapranov-Voevodsky sense $\mathsf{2Vect}^{KV}$ \cite{Kapranov:1994}.\footnote{We emphasize that this model change would not be possible in the strict Baez-Crans sense, as $\mathsf{2Vect}^{BC}$ is known to be completely strict, while $\mathsf{2Vect}^{KV}$ cannot be strictified.} This problem is beign tackled currently by one of the authors.

\paragraph{Ribbon tensor 2-categories and modular invariants of 4-manifolds.} Recall that, in the 1-Hopf algebra case $H$, the representation category $\operatorname{Rep}(H)$ can be modelled with {\it ribbon diagrams} \cite{Majid:1996kd}, which are pictorial presentations for the computations one can do in $\operatorname{Rep}(H)$. This allowed one to deduce quantum invariants of 3-manifolds from the underlying pivotal braided tensor category: such as the Turaev-Viro state-sum invariants \cite{Turaev:1992hq,Barrett1993}, the multifusion invariants of Cui \cite{Cui_2017}, and the modular invariants of Reshetikhin-Turaev \cite{Reshetikhin:1991tc}.

The above Turaev-Viro state sum model has seen a recent generalization to 4-dimensions, known as the {\it Douglas-Reutter model} \cite{Douglas:2018}. This is a state sum model which takes as input a fusion spherical 2-category, and outputs a 4-manifold invariant. Similar constructions had also been studied by Mackaay \cite{Mackaay:ek,Mackaay:hc,Mackaay:uo}, but the {\it 2-ribbon calculus} underlying the Douglas-Reutter model served tantamount importance.

Nevertheless, these invariants are \textit{not} modular, and a 4-dimensional analogue of the Reshetikhin-Turaev invariant is still unknown. For this, one requires the notion of a \textbf{ribbon tensor 2-category}, in which both the braiding and the {\it twist} operations of the underlying tensor 2-category play a central role. Algebraically, we expect such higher ribbon structures to be captured by {\bf ribbon Hopf 2-algebras}. We shall pursue this line of thinking in a followup work, in an effort to construct a 4-dimensional analogue of the Reshetikhin-Turaev modular functor.



\newpage

\appendix

\section{Hopf 2-algebras}\label{2-hopf}

Recall in the usual 1-algebra case that a Hopf algebra $H$ is by definition a bialgebra equipped with an {\it antipode} $S:H\rightarrow H$ which is an anti-algebra and anti-coalgebra map, satisfying
\begin{equation}
    \cdot \circ (\operatorname{id}\otimes S)\circ\Delta=\cdot\circ(S\otimes\operatorname{id})\circ\Delta=\eta\circ\epsilon,\nonumber
\end{equation}
where $\cdot,\Delta$ are the product/coproduct on $H$ and $\eta,\epsilon$ are the unit/counit maps. Correspondingly, we shall define a 2-Hopf algebra $\cG$ as a 2-bialgebra equipped with an appropriate notion of an antipode $S:\cG\rightarrow \cG$.

\paragraph{Opposite 2-algebras.}
Recall that any algebra $A$ comes with an "opposite" algebra $A^\text{opp}$, for which the algebra structure is written "backwards"; the multiplication is given by
\begin{equation}
    A\otimes A\rightarrow A,\qquad x\otimes x' \mapsto x'x.\nonumber
\end{equation}
Similarly, an {\bf opposite 2-algebra} $\cG^\text{opp}$ of a (strict) 2-algebra $\cG$ consist of opposites $\cG_0^\text{opp},\cG_{-1}^\text{opp}$ of the graded components of $\cG$, and a "swapped" $\cG_0$-bimodule structure on $\cG_{-1}$:
\begin{equation}
  \begin{array}{lccl}
    \cdot^{opp}_l:&   \cG_0\otimes\cG_{-1}&\mapsto  &\cG_{-1}\\
       &x\otimes y &\mapsto &y\cdot x = x\cdot^{opp}_l y 
  \end{array}  
  ,\qquad
   \begin{array}{lccl}
    \cdot^{opp}_r:&   \cG_{-1}\otimes\cG_0&\mapsto  &\cG_{-1}\\
       &y\otimes x &\mapsto &x\cdot y = y  \cdot^{opp}_r x.
  \end{array}  
\end{equation}
In other words, the left $\cG_0$-module structure of $\cG_{-1}$ is swapped with the right one. It is easy to see that the equivariance property and Peiffer identity for $t$ still take the form  \eqref{algpeif1}, \eqref{algpeif2}.

\paragraph{Antipode and Hopf 2-algebras.} In analogy with the 1-algebra case, an {\bf anti-2-algebra map} on $\cG$ is equivalent to a 2-algebra map $\cG\rightarrow \cG^\text{opp}$ into the opposite 2-algebra $\cG^\text{opp}$ as defined above. More explicitly, an anti-2-algebra map $\phi: \cG\rightarrow \cG$ has graded components $\phi_{0,-1}: \cG_{0,-1}\rightarrow \cG_{0,-1}$ as anti-algebra maps, such that $\phi_0t = t\phi_{-1}$ and 
\begin{eqnarray}
    \phi_{-1}(x\cdot y) = \phi_{-1}(y)\cdot \phi_0(x),\qquad \phi_{-1}(y\cdot x) = \phi_0(x)\cdot \phi_{-1}(y) \nonumber
\end{eqnarray}
for each $x\in\cG_0,y\in\cG_{-1}$.

This allows us to define an {\it antipode} $S_0=(S_0^1, S_0^0):\cG\rightarrow \cG$ as an anti-2-algebra map on $\cG$ such that 
\begin{eqnarray}
    \cdot\circ(\operatorname{id}\otimes  S_0^1)\circ \Delta_{-1} &=& \cdot\circ( S_0^1\otimes \operatorname{id})\circ\Delta_{-1}=\eta_{-1}\epsilon_{-1},\nonumber\\
    \cdot\circ(\operatorname{id}\otimes  S_0^0)\circ \Delta_0^l &=& \cdot\circ(S_0^1\otimes \operatorname{id})\circ\Delta_0^l=\eta_{-1}\epsilon_0,\nonumber\\
    \cdot\circ(\operatorname{id}\otimes  S_0^1)\circ \Delta_0^r &=& \cdot\circ(S_0^0\otimes \operatorname{id})\circ\Delta_0^r=\eta_{-1}\epsilon_0,\label{antpodes}
\end{eqnarray}
where $\cdot,\Delta$ are the 2-bialgebra product/coproduct and $\eta,\epsilon$ are the unit/counit on $\cG$.

\begin{definition}\label{2hopf}
A (strict) {\bf 2-Hopf algebra} is a Hopf algebra object in $\mathsf{2Vect}^{BC}$. Equivalentlly, it is a (strict) counital 2-bialgebra  $(\cG,\cdot,\Delta,\epsilon)$, as in {\bf Definition \ref{assoc2bialg}}, equipped with an antipode $S_0$ satisfying  \eqref{antpodes}.
\end{definition}

Note  \eqref{antpodes} directly implies that $(\cG_{-1},\Delta_{-1},S_0^1)$ forms a Hopf algebra. Furthermore, recalling the definition  \eqref{deg0sweed} $\Delta_0'=\frac{1}{2}D_t^+\Delta_0$ of the coproduct in degree-0, as well as the property that $t\eta_{-1}=\eta_0$, we have
\begin{equation}
    \cdot \circ(\operatorname{id}\otimes S_0^0) \circ\Delta_0' = \cdot\circ(S_0^0\otimes \operatorname{id})\circ\Delta_0'=\eta_0\circ\epsilon_0\nonumber
\end{equation}
from  \eqref{antpodes}, as the $t$-map by definition intertwines between the antipodes $S_0^0 t = t S_0^1$. As such, $(\cG_0,\Delta_0',S_0^0)$ itself forms a Hopf algebra for which the $t$-map is a Hopf algebra map, given the conditions  \eqref{cohgrpd00}, \eqref{cohgrpd10}, \eqref{cohgrpd+}, \eqref{cohgrpd-} hold. Moreover, the 2-coalgebra compatibility conditions  \eqref{2algcoprod} implies that $H_{-1}=(\cG_{-1},\Delta_{-1},S_0^1)$ forms a Hopf bimodule algebra over $H_0 = (\cG_0,\Delta_0',S_0^0)$. 

In other words, our definition of a {\it strict} 2-Hopf algebra gives rise to a crossed-module of Hopf algebras $t:(\cG_{-1},\Delta_{-1},S_0^1)\rightarrow(\cG_0,\Delta'_0,S_0^0)$, which is precisely the definition of a "$\text{cat}^1$-Hopf algebra" of \cite{Wagemann+2021}. We are able to go even further, as we are able to introduce the 2-R-matrix, as well as study the monoidal weakening of $\cG$. We shall say a bit more about the former in the following.

\paragraph{Quasitriangular strict Hopf 2-algebras.} Now let $(\cG,\Delta,\cR)$ denote a quasitriangular 2-bialgebra, equipped with the 2-$R$-matrix $\cR$. Recall that an antipode $S:\cG\rightarrow \cG$ is a anti-2-algebra map on $\cG$ such that  \eqref{antpodes} are satisfied. Together with the 2-$R$-matrix and  \eqref{def2R1}, it follows that we must have 
\begin{gather}
    (S_0^1\otimes \id)\cR^l \cdot \cR^r = \eta_{-1}^{2\otimes},\qquad (S_0^0\otimes \id)\cR^r\cdot \cR^l = \eta_{-1}^{2\otimes},\nonumber\\
    (\id\otimes S_0^1)\cR^r \cdot \cR^l = \eta_{-1}^{2\otimes},\qquad (\id\otimes S_0^0)\cR^l\cdot \cR^r =\eta_{-1}^{2\otimes}.\label{rmatrixantpode}
\end{gather}
Note here that, as the skew-paring $\langle\cdot,\cdot\rangle_\text{sk}$ is non-degenerate (the "quasi" in quasitriangular), the 2-$R$-matrix components $\cR^{l,r}$ are square and hence admit uniquely defined inverses as square matrices. 

\begin{definition}
    Such a tuple $(\cG,\Delta,\cR,S)$ satisfying  \eqref{rmatrixantpode} defines a strict \textbf{quasitriangular strict 2-Hopf algebra}.
\end{definition}
\noindent If $\cG=kG$ were the 2-group algebra constructed from  \eqref{2grpalg2}, then the algebra action $\cdot$ is simply matrix product, and we recover the usual notion that $(S\otimes \id)\cR, (\id\otimes S)\cR = \cR^{-1}$ provided $\cR$ is quasitriangular.

To not bloat up this paper anymore than we already have, we conclude this section by noting that the embedding in {\it Remark \ref{2algembed}} extends to (quasitriangular) Hopf 2-algebras,
\begin{equation}
    \text{HopfAlg}\hookrightarrow \text{2HopfAlg},\qquad H\mapsto \cH = H\xrightarrow{\id}H.\nonumber
\end{equation}
The antipode on $\cH$ is simply two copies of that $S$ on $H$, and  \eqref{antpodes}, \eqref{rmatrixantpode} is automatically satisfied as the coproduct components $\Delta_{-1}=\Delta_0=\Delta_0'$ and the R-matrices $\cR = R$ all coincide; see Section \ref{2groupbialg}.

\medskip

In general, we should consider a weakened form of this 2-Hopf structure, based on weak 2-bialgebras that we have introduced in the main text. In particular, there is a "cochain homotopy" component $S_1:\cG_0\rightarrow \cG_{-1}$ to the antipode which "deforms" the relations  \eqref{antpodes}, \eqref{rmatrixantpode}. We shall study this structure in more detail in a followup paper.

\section{Classical limits of (weak) 2-bialgebras and 2-$R$-matrices}\label{claslim}
In this section, we prove that the notion of 2-quantum doubles we have defined in Section \ref{str2qd} in the main text reproduces the known notion of 2-Manin triples of Lie 2-bialgebras \cite{Bai_2013,chen:2022} in the classical limit. We shall also examine the quasi-weak case and show that a weak/quasi-2-bialgebra as defined in {\bf Definition \ref{wk2bialg}} reproduces a weak/quasi-Lie 2-bialgebra as defined in \cite{Chen:2012gz,Chen:2013}.

\paragraph{Classical limit and the Lie-ification functor.} Given an (associative) algebra $A\in\text{Alg}_\text{ass}$, it is well-known \cite{Majid:1996kd,Wagemann+2021} that there is a \textit{Lie-ification functor} $\mathcal{L}:\text{Alg}_\text{ass}\rightarrow \text{LieAlg}$ that assigns $A$ to its "classical" Lie algebra $\g(A)$. The Lie bracket is given by the commutator $[X,X'] = XX' - X'X$, where $X \in \g(A)$ is the image of an element $x\in A$ under $\mathcal{L}$. The associativity of $A$ implies the Jacobi identity of $[\cdot,\cdot]$; note $A$ only needs to be left-symmetric (not necessarily associative) in order for $\g(A)$ to enjoy the Jacobi identity \cite{Bai_2013}.

There is a left-adjoint to the Lie-ification functor given by the universal envelope $U: \g\mapsto U(\g)$, which can be understood as a "quantization" map \cite{Majid:1994nw}. There is an analogous result for associative 2-algebras \cite{Wagemann+2021}.
\begin{lemma}\label{classicalstr}
The Lie-ification functor $\mathcal{L}:\mathrm{2Alg}_\mathrm{ass}\rightarrow\mathrm{Lie2Alg}$ lifts to associative 2-algebras (see {\bf Definition \ref{assoc2alg}}), where $\g(\cG) = \mathcal{L}(\cG_{-1})\xrightarrow{t}\mathcal{L}(\cG_0)$ is a Lie 2-algebra with
\begin{equation}
    X\rhd Y = X\cdot Y - Y\cdot X,\qquad X=\mathcal{L}(x),~Y=\mathcal{L}(y),\nonumber
\end{equation}
where $x\in\cG_0,y\in\cG_{-1}$. Moreover, the universal enveloping functor $U$ also lifts to Lie 2-algebras $U(\g)=U(\g_{-1})\xrightarrow{t}U(\g_0)$, such that $U$ is left-adjoint to $\mathcal{L}$.
\end{lemma}
\noindent In the following, we shall write $[\cdot,\cdot]:\g^{2\wedge}\rightarrow\g$ as the binary $L_2$-bracket on $\g$, which consists of the Lie bracket in $\g_0$ as well as the action $\rhd$ of $\g_0$ on $\g_{-1}$ \cite{Bai_2013,chen:2022}.

Note Lie-ification $\mathcal{L}$ is a functor. This means that, in particular, it sends a 2-algebra representation $\rho:\cG\rightarrow\operatorname{End}(V)$ on 2-vector space $V$ to a Lie 2-algebra representation $\mathcal{L}(\rho):\g(\cG)\rightarrow \mathfrak{gl}(V)$ as defined in \cite{Bai_2013,Angulo:2018}.

\subsection{Lie 2-bialgebras and the 2-classical double}
We now extend the above lemma to associative 2-quantum doubles. Let $(\cG,\cdot,\Delta)$ denote a strict 2-bialgebra as defined in {\bf Definition \ref{2bialg}}, and let $(\cG^*,\cdot^*,\Delta^*)$ denote its dually-paired 2-algebra. We put $\g = \mathcal{L}(\cG)$ and $\g^* = \mathcal{L}(\cG^*)$ as the corresponding Lie-ification of these 2-bialgebras.

The Lie-ification procedure can be understood loosely as an "expansion", or linearization, $x\approx 1 + X$ near the identity. Indeed, we have
\begin{equation}
    xx'-x'x \approx (1+X)(1+X')-(1+X')(1+X) \approx [X,X']\nonumber
\end{equation}
modulo terms of higher order. We make use of this notion on the coproduct   \eqref{grpdcoprod}, and also perform a skew-symmetrization, in order to define a Lie 2-algebra 2-cochain $\mathcal{L}(\Delta)= \delta= \delta_{-1}+\delta_0$ on $\g$,
\begin{eqnarray}
    \delta_{-1}(Y) &=& Y_{(1)}\wedge 1 + 1\wedge Y_{(2)},\nonumber\\
    \delta_0(X) &=& \left[X_{(1)}^l - X_{(2)}^r\right]\wedge 1 + 1\wedge \left[X_{(2)}^l - X_{(1)}^r\right]\nonumber\\
    &\equiv& X_{(1)}\wedge 1 + 1\wedge X_{(2)},\label{class2algcocy} 
\end{eqnarray}
where we have made use of the Sweedler notation   \eqref{sweed}, and the conventional notation $\wedge$ to denote skew-symmetric tenor products. Note the skew-symmetrization $\cG_{-1}\wedge\cG_0$ lands as a subspace in $\cG_{-1}\otimes\cG_0\oplus \cG_0\otimes\cG_{-1}$. 

In degree-0, we have of course also the coproduct $\Delta_0'$ defined in   \eqref{deg0sweed}. It gives rise to a Lie algebra cochain on $\mathcal{L}(\cG_0) = \g_0$ by
\begin{equation}
    \delta_0'(X) = \bar X_{(1)} \wedge 1 + 1\wedge \bar X_{(2)} = tX_{(1)} \wedge 1 + 1\wedge X_{(2)},\nonumber
\end{equation}
where $X_{(1)},X_{(2)}$ have been given in   \eqref{class2algcocy}.
\begin{proposition}\label{class2bialge}
The Lie-ification functor $\mathcal{L}$ sends a strict 2-bialgebra $(\cG,\Delta)$ to a Lie 2-bialgebra $(\g,\delta)$.
\end{proposition}
\begin{proof}
Recall $(\g,\delta)$ is a Lie 2-bialgebra iff $\delta$ is a Lie 2-algebra 2-cocycle \cite{Bai_2013}. Therefore it suffices to show that the 2-cochain defined in   \eqref{class2algcocy} is a 2-cocycle. This shall follow from the fact that $(\cG,\cdot,\Delta)$ is a 2-bialgebra --- namely the coproduct map $\Delta$   \eqref{grpdcoprod} satisfies   \eqref{cohgrpd+}, \eqref{cohgrpd-} and \eqref{2algcoprod}. 

First note that   \eqref{cohgrpd+} and \eqref{cohgrpd-} for the coproduct $\Delta$ translates directly to the conditions
\begin{equation}
    (t\otimes 1 + 1\otimes t)\delta_{-1} = \delta_0\circ t,\qquad (t\otimes 1 - 1\otimes t)\delta_0 =0\nonumber
\end{equation}
for the 2-cochain $\delta=\delta_{-1}+\delta_0$. Now by a direct computation using   \eqref{class2algcocy}, the condition   \eqref{2algcoprod} implies
\begin{eqnarray}
    \delta_0[X,X'] &=& \delta_0(XX') - \delta_0(X'X) \nonumber\\
    &=& X_{(1)}X_{(1)}'\wedge 1  + 1\wedge X_{(2)}X_{(2)}'\nonumber\\
    &\qquad& -\left(X_{(1)}'X_{(1)}\wedge 1 + 1\wedge X_{(2)}'X_{(2)}\right)\nonumber\\
    &=& [X_{(1)},X_{(1)}']\wedge 1 + 1\wedge [X_{(2)},X'_{(2)}] \nonumber\\
    &=& tX_{(1)}\rhd X_{(1)}'\wedge 1 + 1\wedge [X_{(2)},X'_{(2)}]\nonumber\\
    &=& (\bar X_{(1)}\rhd\otimes 1 + 1 \otimes \operatorname{ad}_{X_{(2)}})\delta_0(X') - (\bar X'_{(1)}\rhd\otimes 1 + 1 \otimes \operatorname{ad}_{X'_{(2)}})\delta_0(X),\nonumber
\end{eqnarray}
where we have used the the Peiffer identity and the fact that $\bar X_{(1)} = tX_{(1)}$ inherited from the constraints \eqref{condcop}, and 
\begin{eqnarray}
    \delta_{-1}(X\rhd Y) &=& \delta_{-1}(X\cdot Y) - \delta_{-1}(Y\cdot X)\nonumber\\
    &=& \bar X_{(1)}\cdot Y_{(1)} \wedge 1 + 1\wedge \bar X_{(2)}\cdot Y_{(2)} \nonumber\\
    &\qquad& - \left(Y_{(1)}\cdot\bar X_{(1)}\wedge 1 - 1\wedge Y_{(2)}\cdot \bar X_{(1)}\right)\nonumber\\
    &=& (\bar X_{(1)}\rhd Y_{(1)})\wedge 1 + 1\wedge (\bar X_{(2)}\rhd Y_{(2)}) \nonumber\\
    &=& [X_{(1)},Y_{(1)}]\wedge 1 + 1\wedge (X_{(2)}\rhd Y_{(2)}) \nonumber\\
    &=& (\operatorname{ad}_{X_{(1)}}\otimes 1 + 1\otimes X_{(2)}\rhd) \delta_{-1}(Y) - (\operatorname{ad}_{Y_{(1)}}\otimes 1 - 1\otimes \Upsilon_{Y_{(2)}})\delta_0(X),\nonumber
\end{eqnarray}
where $\bar X_{(2)} = X_{(2)}$. These are precisely the Lie 2-algebra 2-cocycle conditions for $\delta$ \cite{Bai_2013,chen:2022}.
\end{proof}

Now the characterization result in \cite{Bai_2013} states that $(\g,\g^*)$ form a matched pair of Lie 2-bialgebras iff $\delta$ is a Lie 2-algebra 2-cobracket on $\g$, namely $\delta$ satisfies the 2-coJacobi identities. For the 2-cocycle $\delta=\mathcal{L}(\Delta)$ defined in   \eqref{class2algcocy}, this is guaranteed precisely by coassociativity   \eqref{cohgrpd+}, \eqref{cohgrpd-}. We have therefore the immediate corollary:
\begin{corollary}
Suppose $(\cG,\cG^*)$ form a matched pair of strict 2-bialgebras. The Lie-ification functor $\mathcal{L}$ sends a 2-quantum double $D(\cG)=\cG\bar{\bowtie}\cG^*$ to a 2-Manin triple $\d = \g\bowtie\g^*[1]$.
\end{corollary}
\noindent In other words, our construction of the 2-quantum double $D(\cG)$ admits the classical 2-Drinfel'd double as a classical limit, which directly categorifies an analogous statement between the general quantum double construction of Majid \cite{Majid:1994nw} and the classical Drinfel'd double \cite{Semenov1992}.

\subsection{The classical 2-$r$-matrix} 
Let us now turn to the classical limit of the 2-$R$-martrix as defined in Section \ref{quasitrihopf}. Prior to that, we first describe one of the key properties of the duality pairing on a 2-quantum double, namely its \textit{invariance}. This is expressed by, for instance,   \eqref{coadj} in the case of the coadjoint representation. For the sew-pairing $\langle\cdot,\cdot\rangle_\text{sk}$ forming the 2-quantum double $D(\cG,\cG) = \cG\bar\bowtie \cG^\text{opp}$, however, $\cG$ acts on $\cG^\text{opp}$ via its underlying (opposite) 2-algebra structure, which means that the skew-pairing satisfies the invariance property
\begin{equation}
    \langle xx',g\rangle_\text{sk} = -\langle x',g\cdot x\rangle_\text{sk},\qquad \langle x\cdot y,f\rangle_\text{sk} = -\langle y,fx\rangle_\text{sk},\qquad \langle ff',y\rangle_\text{sk} = -\langle f',f\cdot y\rangle_\text{sk}.\nonumber
\end{equation}
Given the adjoint action  $\bar\rhd = (\Upsilon,(\rhd_0,\rhd_{-1}))$ of $\cG$ on $\cG^\text{opp}$,
\begin{equation}
    x\rhd_0 g = g\cdot x,\qquad x\rhd_{-1} f = fx,\qquad \Upsilon_yf = f\cdot y,\nonumber
\end{equation}
this invariance property translates to the following conditions on the 2-$R$-matrix $\cR^{l,r}$,
\begin{equation}
    (x\cdot \otimes 1 + 1\otimes x\rhd_0) \cR^l = 0,\qquad (x\cdot \otimes 1)\cR^r + (1\otimes x\rhd_{-1})\cR^l=0,\qquad (f\cdot \otimes 1 + 1\otimes f\rhd_0)\cR^r =0.\nonumber
\end{equation}
Consider the first and last conditions with $x=f\in\cG_0$. They can be rewritten equivalently as the conditions 
\begin{equation}
    (x\cdot \otimes 1)\cR^l + (1\otimes x\rhd_0)\cR^r =0,\qquad (x\cdot \otimes 1)\cR^r + (1\otimes x\rhd_0)\cR^l = 0,\nonumber
\end{equation}
which together with the second condition may be compactly expressed as, using the graded sum,
\begin{equation}
    (x\bar\rhd \otimes 1 + 1\otimes x\bar\rhd)(\cR + \sigma(\cR)) = 0,\qquad \forall~x\in \cG_0,\label{rmatrixinvar}
\end{equation}
where $\sigma$ is a permutation of the $\cG_0,\cG_{-1}$ components.

Let us now finally recover  the  classical 2-$r$-matrix. This is once again accomplished by taking the Lie-ification functor on the quantum 2-R-matrix, $\r=\mathcal{L}(\cR)\in\g\otimes \g$, whence
\begin{equation}
    \g_{-1}\otimes\g_0\ni\r^r = \mathcal{L}(\cR^r),
    \qquad  \g_0\otimes\g_{-1}\ni\r^l=\mathcal{L}(\cR^l).
\end{equation}
The equivariance condition   \eqref{2ybequiv} clearly implies
\begin{equation}
    D_t^-\r = 0,\label{classicalequiv}
\end{equation}
while applying the Lie-ification functor $\mathcal{L}$ to   \eqref{rmatrixinvar} gives
\begin{equation}
    [X\otimes 1 + 1\otimes X, \r + \sigma(\r)] = 0,\qquad X= \mathcal{L}(x) \in \g_0.\nonumber
\end{equation}
Here, we have used the fact that the adjoint action $\rho$ of $\cG$ on itself gives rise to the adjoint representation ( using the graded Lie bracket) $\mathcal{L}(\bar\rhd)=[\cdot,\cdot]$ of $\g$ on itself \cite{Bai_2013}.

Finally, we consider the 2-Yang-Baxter equations \eqref{2yangbax2}. We sum each equation in \eqref{2yangbax2} in the total graded complex $\cG^{3\otimes}$, and rearragnge them to the form 
\begin{eqnarray}
0&=& \left( \cR^r_{23}(\cR^r_{13}\cdot_l\cR_{12}^l) - (\cR^l_{12}\cdot_r\cR^r_{13})\cR^r_{23}\right)\nonumber\\
&\qquad&+ ~    \left(\cR^l_{23}\cdot_l\cR^r_{13}) \cR^r_{12}-  \cR^r_{12} (\cR^r_{13}\cdot_r\cR^l_{23} )\right)  \nonumber\\
     &\qquad&+~ \left(\cR^l_{23}(\cR^l_{13}\cdot_r\cR_{12}^r) - (\cR^r_{12}\cdot_l\cR^l_{13})\cR^l_{23}\right)\nonumber\\
     &\qquad&+~ \left( (\cR^r_{23}\cdot_r\cR^l_{13}) \cR^l_{12}- \cR^l_{12} (\cR^l_{13}\cdot_l\cR^r_{23} )\right)    .\label{sum2yb}
\end{eqnarray}
Applying the Lie-ification functor $\mathcal{L}$ to this equation yields 
\begin{eqnarray}
0&=&\left( [\r^r_{13},\r_{12}^l] + 
[\r^r_{23},\r^r_{13}]+ 
[\r^r_{23},\r_{12}^l] \right)|_{rrl} + 
  \left({[\r^l_{23},\r^r_{13}]}+ 
  {[\r^l_{23},\r^r_{12}]}+{[\r^r_{13},\r^r_{12}]}\right)|_{lrr} \nonumber\\
&\qquad&+ ~\left( {[\r^l_{13},\r_{12}^r]} + { [\r^l_{23},\r^l_{13}]} + {[\r^l_{23},\r_{12}^r]} \right)|_{llr} + 
  \left( [\r^r_{23},\r^l_{13}]+ {[\r^r_{23},\r^l_{12}]} +{[\r^l_{13},\r^l_{12}]}\right)|_{rll} ,\nonumber
\end{eqnarray}
where the subscripts indicate where each term came from in \eqref{sum2yb}.

Consider the two places in which $\r_{23}^l\r^r_{12}$ occurs in the above. These terms take the form respectively in Sweedler notation
\begin{eqnarray}
    \r^l_{23}\r^r_{12}|_{lrr} &=& \r^r_{(1)} \eta_0 \otimes \r^l_{(1)}\r^r_{(2)}\otimes \r^l_{(2)}\cdot\eta_{-1},\nonumber\\
    \r^l_{23}\r^r_{12}|_{llr} &=& \r^r_{(1)} \cdot\eta_{-1} \otimes \r^l_{(1)}\r^r_{(2)}\otimes \r^l_{(2)}\eta_0,\nonumber
\end{eqnarray}
where $\eta_0,\eta_{-1}$ are the units in $\cG_0,\cG_{-1}$. By using the Peiffer identity and the equivariance condition \eqref{classicalequiv}
\begin{equation*}
    (t\r^l_{(1)})\otimes \r^l_{(2)} = (t\otimes 1)\r^l = (1\otimes t) \r^r = \r^r_{(1)}\otimes (t\r^r_{(2)}),
\end{equation*}
we can compute that
\begin{equation*}
    \begin{aligned}[c]
        \r^l_{23}\r^r_{12}|_{llr}&= \r^r_{(1)}\cdot\eta_{-1} \otimes \r^l_{(1)}\cdot (t\r^r_{(2)}) \otimes \r^l_{(2)}\eta_0 \\
        &= (t\r^l_{(1)})\cdot\eta_{-1} \otimes \r^l_{(1)} \cdot \r^l_{(2)}\otimes \r^l_{(2)}\eta_0 \\
        &= \r^l_{(1)}\eta_{-1} \otimes \r^l_{(1)}\cdot_r \r^l_{(2)}\otimes \r^l_{(2)} \eta_0 \\
        &= \r^l_{23}\cdot_r\r^l_{12}
    \end{aligned}\qquad\qquad
    \begin{aligned}[c]
        \r^l_{23}\r^r_{12}|_{lrr} &= \r^r_{(1)}\eta_0 \otimes (t\r^l_{(1)})\cdot \r^r_{(2)} \otimes \r^l_{(2)}\cdot\eta_{-1} \\ 
        &= \r^r_{(1)}\eta_0 \otimes \r^r_{(1)}\cdot \r^r_{(2)} \otimes (t\r^r_{(2)})\cdot\eta_{-1} \\
        &= \r^r_{(1)} \eta_0\otimes \r^r_{(1)}\cdot_l \r^r_{(2)} \otimes \r^r_{(2)}\eta_{-1} \\
        &= \r^r_{23}\cdot_l\r^r_{12}
    \end{aligned},
\end{equation*}
As such, we have
\begin{equation*}
    [\r^l_{23},\r^r_{12}] = [\r^l_{23},\r^l_{12}]= [\r^r_{23},\r^r_{12}],
\end{equation*}
and hence collecting all terms from the above gives
\begin{align*}
    [\r_{12},\r_{13}] ={ [\r^r_{12},\r^r_{13}]+[\r^r_{12},\r^l_{13}]+[\r^l_{12},\r^r_{13}]+[\r^l_{12},\r^l_{13}]}\\
    [\r_{13},\r_{23}]= {[\r^r_{13},\r^r_{23}]+[\r^r_{13},\r^l_{23}]+[\r^l_{13},\r^r_{23}]+[\r^l_{13},\r^l_{23}]}\\
    [\r_{12},\r_{23}]= {[\r^r_{12},\r^r_{23}]+[\r^r_{12},\r^l_{23}]+[\r^l_{12},\r^r_{23}]+[\r^l_{12},\r^l_{23}]}
\end{align*}
This is precisely the 2-graded classical Yang-Baxter equation of \cite{Bai_2013}
\begin{equation}
    \llbracket\r,\r\rrbracket = [\r_{12},\r_{13}] + [\r_{13},\r_{23}] + [\r_{12},\r_{23}] = 0\nonumber
\end{equation}
for the expansion $\r = \mathcal{L}(\cR) = \r^r + \r^l$.

\begin{theorem}
$\cR$ admits $\r$ as a classical limit: the Lie-ification functor sends the 2-$R$-matrix to a 2-graded classical $r$-matrix.
\end{theorem}
\noindent In other words, the "quantization" of the classical 2-$r$-matrix and the associated Lie 2-bialgebra $\g$ yields a  2-$R$-matrix with the associated quasitriangular 2-bialgebra $\cG$, as expected from \eqref{dimladder}.

\subsection{Weak Lie 2-bialgebras} 
We now prove the weak analogues of the classical limit for 2-bialgebras.
\begin{lemma}\label{weak2algtolie}
The Lie-ification functor $\mathcal{L}:\mathrm{Alg}\rightarrow \mathrm{Lie}$ extends to weak 2-algebras, assigning $(\cG,\cT)$ to a weak Lie 2-algebra $(\g(\cG),\mu_3)$ where the homotopy map $\mu_3$ is the total skew-symmetrization of $\cT$.
\end{lemma}
\begin{proof}
We construct the Lie 2-algebra structure as in {\bf Lemma \ref{classicalstr}}. Let $U_3=\mathcal{L}\circ \cT\circ\mathcal{L}$ denote the induced trilinear map on $\mathcal{L}(\cG)$. We apply $\mathcal{L}$ to the Jacobiator $J(X,X',X'')=[X,[X',X'']] + [X',[X'',X'] + [X'',[X,X']] $,
\begin{eqnarray}
    J(X,X',X'')
    &=& X(X'X'') - X(X''X') - (X'X'')X + (X''X')X \nonumber\\
    &\qquad& + X'(X''X) - X'(XX'') - (X''X)X' + (XX'')X'\nonumber\\
    &\qquad& + X''(XX') - X''(X'X) - (XX')X''+(X'X)X''\nonumber\\
    &=& tU_3(X,X',X'') - tU_3(X,X'',X') + tU_3(X',X'',X) \nonumber\\
    &\qquad&-~ tU_3(X',X,X'') + tU_3(X'',X',X) - tU_3(X'',X,X')\nonumber\\ 
    &=& t(U_3(X,X',X'') - U_3(X,X'',X') + U_3(X',X'',X) \nonumber\\
    &\qquad&-~ U_3(X',X,X'') + U_3(X'',X,X'))- U_3(X'',X',X) ,\nonumber
\end{eqnarray}
where we have used the weak 1-associativity condition for $\cG$. Similarly, for $J(X,X',Y)=X\rhd (X'\rhd Y) - X'\rhd (X\rhd Y) - [X,X']\rhd Y$ we have
\begin{eqnarray}
    J(X,X',Y)&=& t(U_3(X,X',tY) - U_3(X,tY,X') + U_3(X',tY,X) \nonumber\\
    &\qquad&-~ U_3(X',X,tY) +U_3(tY,X,X') - U_3(tY,X',X) ,\nonumber
\end{eqnarray}
hence if we define the total skew-symmetrization
\begin{eqnarray}
    \mu_3(X,X',X'') &\equiv& U_3(X,X',X'') - U_3(X,X'',X') + U_3(X',X'',X) \nonumber\\
    &\qquad&~- U_3(X',X,X'') +U_3(X'',X,X') - U_3(X'',X',X), \nonumber
\end{eqnarray}
then weak 1-associativity implies the 2-Jacobi identity on $\mathcal{L}(\cG)$. 

Using the Peiffer conditions on this fact, we see that the weak bimodularity condition also implies the 2-Jacobi identity, with two $tY$'s inserted in $U_3$ instead. Similar computations show that the Hochschild 3-cocycle condition for $\cT$ implies the Lie 3-cocycle condition for $\mu_3$. 

Finally, let $F:(\cG,\cT)\rightarrow (\cG',\cT')$ denote a weak 2-algebra homomorphism as defined in   \eqref{weak2hom}. By applying the Lie-ification functor and appropriately skew-symmetrizing $\cT,\cT'$ and the 2-algebra structure, we recover precisely the definition of a weak 2-algebra map $\mathcal{L}(F):(\g,\mu)\rightarrow (\g',\mu')$ \cite{Baez:2005sn}. Thus $\mathcal{L}$ is functorial.
\end{proof}
Similar to the Lie 2-algebra 2-cocycle   \eqref{class2algcocy} defined from the coproduct $\Delta$, we form the classical limit of the coassociator $\Delta_1$ by totally skew-symmetrizing and linearizing it, such that we have the Lie cochain
\begin{equation}
    \delta_1(X) = X_{(1)}\wedge 1 \wedge 1 - 1\wedge X_{(2)}\wedge 1 + 1\wedge 1\wedge X_{(3)},\qquad X\in\g_0=\mathcal{L}(\cG_0).\label{class2coassoc}
\end{equation}
It is not hard to see by, for instance, dualizing the computations in the proof of {\bf Lemma \ref{weak2algtolie}}, that the conditions   \eqref{weakcohgrpd1}, \eqref{weakcohgrpd3} reduce to
\begin{eqnarray}
    \delta_{-1}\circ\delta_{-1} &=& \delta_1\circ t,\qquad\text{cf.   (42) in \cite{Chen:2012gz}}\nonumber\\
    (\delta_{-1}+\delta_0)\circ\delta_0 &=& D_t\circ\delta_1,\qquad  \text{cf.   (43) in \cite{Chen:2012gz}},\nonumber\\
    \delta_1\circ\delta_0 &=& \delta_{-1}\circ\delta_1,\qquad \text{cf.   (44) in \cite{Chen:2012gz}}.\nonumber
\end{eqnarray}

Let $(\cG,\cT,\Delta_1)$ be a weak 2-bialgebra as given in {\bf Definition \ref{wk2bialg}}. The conditions   \eqref{coassocalg} translate directly to
\begin{eqnarray}
    \delta_{-1}(\mu_3(X,X',X'')) &=& \mu_3(\bar X_{(1)},\bar X_{(1)}',\bar X''_{(1)})\wedge \mu_3(\bar X_{(2)},\bar X_{(2)}',\bar X_{(2)}''),\nonumber\\
    \delta_1([X,X']) &=& [X_{(1)},X_{(1)}']\wedge 1 \wedge 1 - 1\wedge [X_{(2)},X_{(2)}']\wedge 1 + 1\wedge 1\wedge [X_{(3)},X_{(3)}'],\nonumber
\end{eqnarray}
which are precisely the conditions for a {\it weak-Lie 2-bialgebra} $(\g,\mu_3,\delta)$ \cite{Chen:2013}, expressed explicitly. In other words, we have the weak version of {\bf Proposition \ref{class2bialge}}:
\begin{proposition}
The Lie-ification functor takes a weak 2-bialgebra $(\cG,\cT,\Delta)$ to a weak Lie 2-bialgebra $(\g,\mu,\delta)$, with the 2-cocycle data given as in   \eqref{class2algcocy}, \eqref{class2coassoc}.
\end{proposition}
\noindent Note that this is a general result, which does not require the skeletality assumption on $\cG$. When $\cT=0$ and $\mu_3=0$, we recover the conditions for a quasi-Lie 2-bialgebra studied also in \cite{Chen:2012gz}.

\section{Module 2-categories and 2-representation theory}\label{weak2repthy}
As mentioned in {\it Remark \ref{weak2end}}, the {\it weak} 2-algebras live in the homotopy refinement $\mathsf{2Vect}^{hBC}$ of the 2-category $\mathsf{2Vect}^{BC}$ of Baez-Crans 2-vector spaces. We now show that the 2-representation theory developed here is nothing more than the 2-category of $\cG$-modules over the 2-category $\mathsf{2Vect}^{hBC}$.

Recall briefly some key aspects of modules in a 2-category $\textbf{D}$ of (certain types of) categories \cite{Delcamp:2021szr,Delcamp:2023kew}. 
\begin{definition}\label{2module2rep}
    Let $\cC\in\textbf{D}$ denote a monoidal category. The 2-category $\operatorname{Mod}_\textbf{D}(\cC)$ of {\bf $\cC$-modules in $\textbf{D}$} consist of objects $\cD\in \textbf{D}$ equipped with a $\cC$-action 1-morphism $\rhd:\cC\times\cD\rightarrow \cD$ and a set of pseudonatural transformations (the \textit{associators})
\begin{equation}
    \alpha_{XY|\cD}: (X \otimes Y)\rhd - \rightarrow X\rhd(Y\rhd -) \nonumber
\end{equation}
for each $X,Y\in \cC$, satisfying the module pentagon relations up to a possibly non-trivial {\it module pentagonator} 2-morphism $\pi_{XYZ}$. The 1-morphisms are $\cC$-module functors, and the 2-morphisms are $\cC$-module natural transformations.
\end{definition}
\noindent Crucially, the module pentagonators $\pi$ must satisfy {\it on the nose} a coherence condition, called the {\it associahedron condition}. The explicit expressions of these conditions can be found in \cite{Delcamp:2021szr,Delcamp:2023kew}.

\medskip

Consider a 2-bialgebra $\cG$ as an algebra object in $\textbf{D}= \mathsf{2Vect}^{hBC}$. Evaluating an action 2-functor $\rhd: \cG\times\mathsf{2Vect}^{hBC}\rightarrow\mathsf{2Vect}^{hBC}$ on the object $V$ gives precisely a weak 2-representation $\rho:\cG\rightarrow\operatorname{End}_{\mathsf{2Vect}^{hBC}}(V)= \End(V)$ of $\cG$ on $V\in\mathsf{2Vect}^{hBC}$, as we have defined in the main text. 
\begin{theorem}\label{pentag}
Weak 2-representations are $\cG$-modules in $\mathsf{2Vect}^{hBC}$: $$\operatorname{2Rep}^\cT(\cG) = \operatorname{Mod}_{\mathsf{2Vect}^{hBC}}(\cG).$$
\end{theorem}
\begin{proof}
As foretold, we reconstruct the {module associator} $\alpha$ and {pentagonator} $\pi$ of the $\cG$-modules $V\in \mathsf{2Vect}^{hBC}$ by taking 
\begin{equation}
\alpha_{x_1x_2|V} = \varrho(\rho_0(x_1),\rho_0(x_2))(V),\qquad \pi_{x_1x_2x_3|V} = \T(\rho_0(x_1),\rho_0(x_2),\rho_0(x_3))(V), \label{pentag2rep}
\end{equation}
where $\rho= (\varrho,\rho_0,\rho_1): \cG\rightarrow \End(V)$ is a weak 2-representation and $\T$ is the Hochschild 3-cocycle on $\End(V)$. We now proceed level by level.

\paragraph{Objects.} We identify the action 2-functor $\rhd$ as the weak 2-representation $\rho$ such that $x\rhd V = \rho_0(x)V$ for each $x\in \cG_0$. An arrow $x\rhd V\rightarrow x'\rhd V$ is therefore expressed as $\rho_1(y)V$, where $y\in \cG_{-1}$ is interpreted as a 2-morphism $x\xRightarrow{y}x'$ between $x,x'=x+ty$ \cite{Wagemann+2021,Baez:2003fs}, or simply by $\rho_1(y)$. What we need to prove is the pentagon relation between $\alpha,\pi$, as well as the associahedron condition for $\pi$. The pentagon relation can be written as
\begin{equation*}
\begin{tikzcd}
	{((x_1x_2)x_3)\rhd V} & & {(x_1(x_2x_3))\rhd V} & {x_1\rhd((x_2x_3)\rhd V)} \\
	& & {\xRightarrow{\pi_{x_1x_2x_3|V}}} \\
	{(x_1x_2)\rhd(x_3\rhd V)} & & & {x_1\rhd (x_2\rhd(x_3\rhd V)))}
	\arrow["{\rho_1(\cT(x_1,x_2,x_3))}", from=1-1, to=1-3]
	\arrow["{\varrho(x_1,x_2x_3)}", from=1-3, to=1-4]
	\arrow["{\varrho(x_1,x_2))\rho_0(x_3)}", from=3-1, to=3-4]
	\arrow["{\varrho(x_1x_2,x_3)}"', from=1-1, to=3-1]
	\arrow["{\rho_0(x_1)\varrho(x_2,x_3)}", from=1-4, to=3-4]
\end{tikzcd}
\end{equation*}
Rewriting $\pi$ in terms of the 3-cocycle $\T$, we have
\begin{eqnarray}
    \T(\rho_0(x_1),\rho_0(x_2),\rho_0(x_3)) &=& -\varrho(x_1x_2,x_3) - \varrho(x_1,x_2)\rho_0(x_3) \nonumber\\
    &\qquad& +~ \rho_1(\cT(x_1,x_2,x_3)) + \varrho(x_1,x_2x_3) + \rho_0(x_1)\varrho(x_2,x_3),\nonumber
\end{eqnarray}
which is nothing but the last equation of  \eqref{weak2hom}. It is then easy to see that the associahedron condition follows from the Hochschild 3-cocycle condition for $\T$.

\paragraph{2-intertwiners.} 
Recall the notion of weak 2-intertwiners that we have given in {\bf Definition \ref{weak2int}}. By treating $V$ as a $\cG$-module 2-category and taking $\rhd,\rhd'$ as the action 2-functors corresponding to the 2-representations $\rho,\rho'$, we equivalently characterize the cochain homotopy $I$ as a collection of invertible natural transformations $I_{\bullet,i}: i(\bullet \rhd V) \Rightarrow \bullet \rhd' i(V)$, such that the following pentagon relation
\[\begin{tikzcd}
	{i(\rho_0(xx’)V)}& && &{\rho’_0(xx’)\circ i(V)} \\
	\\
	{i(\rho_0(x)\rho_0(x’)V)}& & {\rho’_0(x)\circ i(\rho_0(x’)V)} && {\rho’_0(x)\rho’_0(x’)\circ i(V)}
	\arrow["{I_{xx’,i}}", from=1-1, to=1-5]
	\arrow["{i\circ \varrho(x,x’)}"', from=1-1, to=3-1]
	\arrow["{I_{x,i} \otimes\id_{\rho_0(x’)}}", from=3-1, to=3-3]
	\arrow["{\rho’_0(x)\otimes I_{x’,i}}", from=3-3, to=3-5]
	\arrow["{\varrho’(x,x’)\circ i}"', from=3-5, to=1-5]
\end{tikzcd}\]
follows directly from  \eqref{2inthomotopy} This recovers precisely the notion of a $\cG$-module functor \cite{Delcamp:2023kew}. Notice no pentagonator appears here, as this is a relation on the 2-morphisms in $\operatorname{2Rep}^\cT(\cG)$ and hence a pentagonator for it would have to be a 3-morphism.

\paragraph{Modifications.}
Now let us consider the notion of modifications in $\operatorname{2Rep}^\cT(\cG)$ we have defined in {\bf Definition \ref{weakmodif}}. The condition  \eqref{modifhomotopy} is equivalent to the composition of 2-morphisms $(\operatorname{id}_{\rho_0(x)}\mu)\ast I_{x,i} = I_{x,i'}\ast \mu$, which is exactly a module natural transformation \cite{Delcamp:2023kew}.
\end{proof}

We show in this Section that $\End(V)$ can be interpreted in the context of module 2-categories. 

\medskip

We first recall briefly some key aspects of a module 2-category \cite{Delcamp:2021szr,Delcamp:2023kew}. To be more concrete, let $\cC$ denote a semisimple (monoidal) 2-category. A {\bf $\cC$-module 2-category} is a $k$-linear semisimple 2-category $\cD$ with a $\cC$ action 2-functor $\rhd:\cC\times\cD\rightarrow \cD$ and a set of adjoint natural equivalences (the \textit{associators})
\begin{equation}
    \alpha_{XY|A}: (X \otimes Y)\rhd A \rightarrow X\rhd(Y\rhd A) \nonumber
\end{equation}
for each $X,Y\in \cC$ and $A\in \cD$, satisfying the module pentagon relations up to a possibly non-trivial {\it module pentagonator} 2-morphism $\pi_{XYZ|A}$. These pentagonators must satisfy {\it on the nose} an additional coherence condition, called the {\it associahedron condition}. The explicit expressions of these conditions can be found in \cite{Delcamp:2021szr,Delcamp:2023kew}.

\subsection{Higher-representations of 2-groups}
For simplicity, consider the special case of a \textit{finite} skeletal 2-group $G = G_{-1}\xrightarrow{1}G_0$, where $1=\eta_0$ denotes the group unit in $G_0$, and take the corresponding 2-group algebra $kG$ via part c) of {\bf Example \ref{sec:2gpalgb}}. We also assume $G$ {\it splits}, namely it has a trivial Postnikov class $\tau=0$ (and hence the Hochschild 3-cocycle $\cT=0$ is trivial on the associated 2-group algebra $kG$).  

Recall the definition of a 2-representation of $\cG$ in 2-category $\textbf{D}$ in \textbf{Definition \ref{2module2rep}}. 
As shown in \textbf{Theorem \ref{pentag}}, if we take $\textbf{D} = \mathsf{2Vect}^{hBC}$ then we recover the weak 2-representations as defined in {\bf Definition \ref{weak2rep}} --- cochain homotopies play the role of natural transformations/2-morphisms in this setting. 

On the other hand, one typically considers the 2-category $\cC = \mathsf{2Vect}^{KV}$ of Kapranov-Voevodsky 2-vector spaces \cite{Kapranov:1994} in the usual literature \cite{Baez:2012,Douglas:2018,Delcamp:2021szr,Delcamp:2023kew,Bartsch:2022mpm,Bartsch:2023wvv}. We shall denote the former 2-representation 2-category by $\operatorname{2Rep}(kG)$, and the latter by $\operatorname{2Rep}_G$. 

\medskip

It is clear that the weak 2-representation theory developed in this paper hosts non-trivial $k$-invariants, and does hence indeed bypass the results of \cite{heredia2016representations2groupsbaezcrans2vector}. We will now show moreover that, specifically for finite 2-groups, the $k$-invariants in $\operatorname{2Rep}(kG)$ are in bijection with those found in $\operatorname{2Rep}_G$ in recent literature. 

\medskip

Since $t=1$ and by definition $\varrho(x,1) = \varrho(1,x)=0$, the left-bimodule structure in particular is respected $\rho_1(x\cdot y) = \rho_0(x)\rho_1(y)$ from  \eqref{weak2repalg}. As the left-bimodule action $\cdot$ coincides with the group action $\rhd$ by construction  \eqref{2grpalg2}, this implies the condition
\begin{equation}
    \rho(x\rhd y)_v = \rho(y)_{\rho(x)v},\qquad \forall x\in G_0,~ v\in V\nonumber
\end{equation}
for each $y\in G_{-1}$, which has also appeared in the 2-representation theory $\operatorname{2Rep}_G$ based on the Kaparanov-Voevodsky setting \cite{Delcamp:2023kew}. However, in there we also have the following data 
\begin{enumerate}
    \item The composition of elements $x\in G_0$ is preserved only up to an invertible natural transformation
    \begin{equation}
        p_{x_1,x_2}: \rho(x_1)\circ \rho(x_2)\xRightarrow{\sim} \rho(x_1x_2),\nonumber
    \end{equation}
    satisfying 
    \begin{equation}
        p_{x_1,x_2x_3}\ast (\operatorname{id}_{\rho(x_1)}\circ p_{x_2,x_3}) = p_{x_1x_2,x_3}\ast (p_{x_1,x_2}\circ\operatorname{id}_{\rho(x_3)}),\label{2repcoh1}
    \end{equation}
    where $\ast$ is the composition of 2-morphisms.
    \item A 1-morphism $i:\rho\rightarrow \rho'$ between two 2-representations assigns an object $i_\text{pt} \in \operatorname{Fun}(\rho(\text{pt}),\rho'(\text{pt}))$ and an invertible natural transformation $i_x: i_\text{pt}\circ \rho(x)\xRightarrow{\sim}\rho'(x)\circ i_\text{pt}$ to each $x\in G_0$ satisfying
    \begin{equation}
        (p_{x_1,x_2}\circ\operatorname{id}_{i_\text{pt}})\ast (\operatorname{id}_{\rho(x_1)}\circ i_{x_2}) \ast (i_{x_1}\circ\operatorname{id}_{\rho(x_2)}) = i_{x_1x_2} \ast (\operatorname{id}_{i_ \text{pt}} \circ p_{x_1,x_2}),\label{2repcoh2}
    \end{equation}
    as well as the naturality condition 
    \begin{equation}
        i_x \ast (\operatorname{id}_{i_ \text{pt}}\circ\rho(y)) = (\rho'(y)\circ \operatorname{id}_{i_ \text{pt}})\ast i_x \label{2repcoh3}
    \end{equation}
    for each $x\in G_0,y\in G_{-1}$.
    \item A 2-morphism $\mu:i\Rightarrow i'$ between two 1-morphisms assigns a natural transformation $\mu_\ast \in \operatorname{Fun}(i_ \text{pt},i'_\text{pt})$ satisfying
    \begin{equation}
        (\operatorname{id}_{\rho(x)}\circ \mu_\ast)\ast i_x = i'_x\ast \mu_\ast.\label{2repcoh4}
    \end{equation}
\end{enumerate}

Now let $\rho=(\varrho,\rho_0,\rho_1)\in \operatorname{2Rep}(kG)$ denote a weak 2-representation as we have defined in the main text. We will demonstrate that an identical set of coherence witnesses are encoded in {\bf Definition \ref{weak2rep}}.
\begin{itemize}
    \item We identify the invertible natural transformation $p_{x_1,x_2}$ with $\varrho(\rho_0(x_1),\rho_0(x_2))$ for each $x_1,x_2\in G_0$; as $\tau =0 $ is trivial, the Hoschild 3-cocycles $\cT,\T$ are both trivial, whence last equation of  \eqref{weak2hom} implies  \eqref{2repcoh1}.
    \item Take $\rho(\ast) =V$ and $\rho'(\ast)=W$, a 1-morphisms $i_ \text{pt}$ clearly denote a cochain map $i:V\rightarrow W$. We identify the invertible natural transformation $i_x$ with the cochain homotopy $I_{x,i}$ defined in {\bf Definition \ref{weak2int}}, whence  \eqref{2repcoh2} is equivalent to  \eqref{2inthomotopy}. Moreover, as $G$ is skeletal $t=1$, the 2-morphisms $\operatorname{id}_i\circ\rho(y),\rho(y)\circ\operatorname{id}_i$ are self-modifications $\mu: i\Rightarrow i$ on $i=i_ \text{pt}$, whence  \eqref{2repcoh3} follows from  \eqref{modifhomotopy}.
    \item We identify the 2-morphism $\mu_\ast$ with a modification $\mu$ as defined in {\bf Definition \ref{weakmodif}}.  \eqref{2repcoh4} then clearly follows also from  \eqref{modifhomotopy}.
\end{itemize}
The set of $k$-invariants in $\operatorname{2Rep}^\cT(kG) = \operatorname{Mod}_{\mathsf{2Vect}^{hBC}}(kG)$ --- coming from $A_\infty$-algebras --- therefore coincide with that coming from $\operatorname{2Rep}_G = \operatorname{Mod}_{\mathsf{2Vect}^{KV}}(G)$. This fact is crucial for the main result in \cite{Chen2z:2023}.



\newpage

\printbibliography

@article{Budzik:2022mpd,
    author = "Budzik, Kasia and Gaiotto, Davide and Kulp, Justin and Wu, Jingxiang and Yu, Matthew",
    title = "{Feynman diagrams in four-dimensional holomorphic theories and the Operatope}",
    eprint = "2207.14321",
    archivePrefix = "arXiv",
    primaryClass = "hep-th",
    doi = "10.1007/JHEP07(2023)127",
    journal = "JHEP",
    volume = "07",
    pages = "127",
    year = "2023"
}

@journal{Fiore2004PseudoLB,
  title={Pseudo Limits, Biadjoints, and Pseudo Algebras: Categorical Foundations of Conformal Field Theory},
  author={Thomas M. Fiore},
  year={2006},
journal={Mem. Amer. Math. Soc.},
volume=182,
number=860,
  url={https://api.semanticscholar.org/CorpusID:119628594}
}

@misc{heredia2016representations2groupsbaezcrans2vector,
      title={On the representations of 2-groups in {Baez-Crans} 2-vector spaces}, 
      author={Benjamín Alarcón Heredia and Josep Elgueta},
      year={2016},
      eprint={1607.04986},
      archivePrefix={arXiv},
      primaryClass={math.CT},
      url={https://arxiv.org/abs/1607.04986}, 
}

@article{Bartsch:2022mpm,
    author = "Bartsch, Thomas and Bullimore, Mathew and Ferrari, Andrea E. V. and Pearson, Jamie",
    title = "{Non-invertible symmetries and higher representation theory I}",
    eprint = "2208.05993",
    archivePrefix = "arXiv",
    primaryClass = "hep-th",
    doi = "10.21468/SciPostPhys.17.1.015",
    journal = "SciPost Phys.",
    volume = "17",
    number = "1",
    pages = "015",
    year = "2024"
}

@article{Stasheff:1963,
 ISSN = {00029947},
 URL = {http://www.jstor.org/stable/1993608},
 author = {James Dillon Stasheff},
 journal = {Transactions of the American Mathematical Society},
 number = {2},
 pages = {275--292},
 publisher = {American Mathematical Society},
 title = {Homotopy Associativity of H-Spaces. I},
 urldate = {2023-07-29},
 volume = {108},
 year = {1963}
}

@book{costello_gwilliam_2016, place={Cambridge}, series={New Mathematical Monographs}, title={Factorization Algebras in Quantum Field Theory}, volume={1}, DOI={10.1017/9781316678626}, publisher={Cambridge University Press}, author={Costello, Kevin and Gwilliam, Owen}, year={2016}, collection={New Mathematical Monographs}}

@article{Lu:1996,
author = {LU, JIANG-HUA},
title = {HOPF ALGEBROIDS AND QUANTUM GROUPOIDS},
journal = {International Journal of Mathematics},
volume = {07},
number = {01},
pages = {47-70},
year = {1996},
doi = {10.1142/S0129167X96000050},

URL = { 
    
        https://doi.org/10.1142/S0129167X96000050
    
    

},
eprint = { 
    
        https://doi.org/10.1142/S0129167X96000050
    
    

}
}

@article{Kontsevich:1997vb,
    author = "Kontsevich, Maxim",
    title = "{Deformation quantization of Poisson manifolds. 1.}",
    eprint = "q-alg/9709040",
    archivePrefix = "arXiv",
    doi = "10.1023/B:MATH.0000027508.00421.bf",
    journal = "Lett. Math. Phys.",
    volume = "66",
    pages = "157--216",
    year = "2003"
}

@article{Bartsch:2023wvv,
    author = "Bartsch, Thomas and Bullimore, Mathew and Grigoletto, Andrea",
    title = "{Representation theory for categorical symmetries}",
    eprint = "2305.17165",
    archivePrefix = "arXiv",
    primaryClass = "hep-th",
    month = "5",
    year = "2023"
}

@book{neuchl1997representation,
  title={Representation Theory of Hopf Categories},
  author={Neuchl, M.},
  url={https://books.google.ca/books?id=gpgLHAAACAAJ},
  year={1997},
  publisher={Verlag nicht ermittelbar}
}

@book{gurski2006algebraic,
  title={An Algebraic Theory of Tricategories},
  author={Gurski, M.N.},
  url={https://books.google.ca/books?id=HTwCOAAACAAJ},
  year={2006},
  publisher={University of Chicago, Department of Mathematics}
}

@inbook{Kapranov:1994,
	title = {2-categories and Zamolodchikov tetrahedra equations},
	booktitle = {Algebraic groups and their generalizations: quantum and infinite-dimensional methods (University Park, PA, 1991)},
	series = {Proc. Sympos. Pure Math.},
	volume = {56},
	year = {1994},
	pages = {177{\textendash}259},
	publisher = {Amer. Math. Soc., Providence, RI},
	organization = {Amer. Math. Soc., Providence, RI},
	author = {Kapranov, M. M. and Voevodsky, V. A.}
}

@article{Else:2017yqj,
    author = "Else, Dominic V. and Nayak, Chetan",
    title = "{Cheshire charge in (3+1)-dimensional topological phases}",
    eprint = "1702.02148",
    archivePrefix = "arXiv",
    primaryClass = "cond-mat.str-el",
    doi = "10.1103/PhysRevB.96.045136",
    journal = "Phys. Rev. B",
    volume = "96",
    number = "4",
    pages = "045136",
    year = "2017"
}

@article{Kong:2020wmn,
    author = "Kong, Liang and Tian, Yin and Zhang, Zhi-Hao",
    title = "{Defects in the 3-dimensional toric code model form a braided fusion 2-category}",
    eprint = "2009.06564",
    archivePrefix = "arXiv",
    primaryClass = "cond-mat.str-el",
    doi = "10.1007/JHEP12(2020)078",
    journal = "JHEP",
    volume = "12",
    pages = "078",
    year = "2020"
}

@article{Hamma:2004ud,
    author = "Hamma, Alioscia and Zanardi, Paolo and Wen, Xiao Gang",
    title = "{String and membrane condensation on 3-D lattices}",
    eprint = "cond-mat/0411752",
    archivePrefix = "arXiv",
    doi = "10.1103/PhysRevB.72.035307",
    journal = "Phys. Rev. B",
    volume = "72",
    pages = "035307",
    year = "2005"
}

@article{Kong:2014qka,
    author = "Kong, Liang and Wen, Xiao-Gang",
    title = "{Braided fusion categories, gravitational anomalies, and the mathematical framework for topological orders in any dimensions}",
    eprint = "1405.5858",
    archivePrefix = "arXiv",
    primaryClass = "cond-mat.str-el",
    month = "5",
    year = "2014"
}

@article{Chen2z:2023,
    author = "Chen, Hank",
    title = "{Drinfeld double symmetry of the 4d Kitaev model}",
    eprint = "2305.04729",
    archivePrefix = "arXiv",
    primaryClass = "cond-mat.str-el",
    doi = "10.1007/JHEP09(2023)141",
    journal = "JHEP",
    volume = "09",
    pages = "141",
    year = "2023"
}

@article{GURSKI20114225,
title = {Loop spaces, and coherence for monoidal and braided monoidal bicategories},
journal = {Advances in Mathematics},
volume = {226},
number = {5},
pages = {4225-4265},
year = {2011},
issn = {0001-8708},
doi = {https://doi.org/10.1016/j.aim.2010.12.007},
url = {https://www.sciencedirect.com/science/article/pii/S0001870810004408},
author = {Nick Gurski}
}

@article{Hopf:1941,
 ISSN = {0003486X},
 URL = {http://www.jstor.org/stable/1968985},
 author = {Heinz Hopf},
 journal = {Annals of Mathematics},
 number = {1},
 pages = {22--52},
 publisher = {Annals of Mathematics},
 title = {Uber Die Topologie der Gruppen-Mannigfaltigkeiten und Ihre Verallgemeinerungen},
 urldate = {2023-01-20},
 volume = {42},
 year = {1941}
}

@article{Borel:1953,
 ISSN = {0003486X},
 URL = {http://www.jstor.org/stable/1969728},
 author = {Armand Borel},
 journal = {Annals of Mathematics},
 number = {1},
 pages = {115--207},
 publisher = {Annals of Mathematics},
 title = {Sur La Cohomologie des Espaces Fibres Principaux et des Espaces Homogenes de Groupes de Lie Compacts},
 urldate = {2023-01-20},
 volume = {57},
 year = {1953}
}

@article{WITTEN1990285,
title = {Gauge theories, vertex models, and quantum groups},
journal = {Nuclear Physics B},
volume = {330},
number = {2},
pages = {285-346},
year = {1990},
issn = {0550-3213},
doi = {https://doi.org/10.1016/0550-3213(90)90115-T},
url = {https://www.sciencedirect.com/science/article/pii/055032139090115T},
author = {Edward Witten}
}

@article{Reshetikhin:1991tc,
    author = "Reshetikhin, N. and Turaev, V. G.",
    title = "{Invariants of three manifolds via link polynomials and quantum groups}",
    doi = "10.1007/BF01239527",
    journal = "Invent. Math.",
    volume = "103",
    pages = "547--597",
    year = "1991"
}

@article{Wen:2010gda,
    author = "Chen, Xie and Gu, Zheng Cheng and Wen, Xiao Gang",
    title = "{Local unitary transformation, long-range quantum entanglement, wave function renormalization, and topological order}",
    eprint = "1004.3835",
    archivePrefix = "arXiv",
    primaryClass = "cond-mat.str-el",
    doi = "10.1103/PhysRevB.82.155138",
    journal = "Phys. Rev. B",
    volume = "82",
    pages = "155138",
    year = "2010"
}

@article{Delcamp:2021szr,
    author = "Delcamp, Clement",
    title = "{Tensor network approach to electromagnetic duality in (3+1)d topological gauge models}",
    eprint = "2112.08324",
    archivePrefix = "arXiv",
    primaryClass = "cond-mat.str-el",
    doi = "10.1007/JHEP08(2022)149",
    journal = "JHEP",
    volume = "08",
    pages = "149",
    year = "2022"
}

@article{Delcamp:2023kew,
    author = "Delcamp, Clement and Tiwari, Apoorv",
    title = "{Higher categorical symmetries and gauging in two-dimensional spin systems}",
    eprint = "2301.01259",
    archivePrefix = "arXiv",
    primaryClass = "hep-th",
    doi = "10.21468/SciPostPhys.16.4.110",
    journal = "SciPost Phys.",
    volume = "16",
    number = "4",
    pages = "110",
    year = "2024"
}

@article{Nguyen2014CROSSEDMA,
  title={CROSSED MODULES AND STRICT GR-CATEGORIES},
  author={Tien Quang Nguyen and Thi Cuc Pham and Thu Thuy Nguyen},
  journal={Communications of The Korean Mathematical Society},
  year={2014},
  volume={29},
  pages={9-22},
  url={https://api.semanticscholar.org/CorpusID:55316397}
}

@book{book-quasihopf,
title = "Quasi-Hopf Algebras: A Categorical Approach",
author = "Stefaan Caenepeel and Daniel Bulacu and Florin Panaite and {Van Oystaeyen}, Freddy",
year = "2019",
language = "English",
isbn = "978-1-108-42701-2",
series = "Encyclopedia of Mathematics and its Applications",
publisher = "Cambridge University Press",
}

@article{KongTianZhou:2020,
title = {The center of monoidal 2-categories in 3+1D Dijkgraaf-Witten theory},
journal = {Advances in Mathematics},
volume = {360},
pages = {106928},
year = {2020},
issn = {0001-8708},
doi = {https://doi.org/10.1016/j.aim.2019.106928},
url = {https://www.sciencedirect.com/science/article/pii/S0001870819305432},
author = {Liang Kong and Yin Tian and Shan Zhou}
}

@article{Joyal:1993,
title = "Braided tensor categories",
author = "A. Joyal and R. Street",
year = "1993",
month = nov,
doi = "10.1006/aima.1993.1055",
language = "English",
volume = "102",
pages = "20--78",
journal = "Advances in Mathematics",
issn = "0001-8708",
publisher = "Elsevier",
number = "1",
}

@article{Wen:2019,
  title = {Classification of $3+1\mathrm{D}$ Bosonic Topological Orders (II): The Case When Some Pointlike Excitations Are Fermions},
  author = {Lan, Tian and Wen, Xiao-Gang},
  journal = {Phys. Rev. X},
  volume = {9},
  issue = {2},
  pages = {021005},
  numpages = {37},
  year = {2019},
  month = {Apr},
  publisher = {American Physical Society},
  doi = {10.1103/PhysRevX.9.021005},
  url = {https://link.aps.org/doi/10.1103/PhysRevX.9.021005}
}

@article{Schafer:1955,
  title={Structure and representation of nonassociative algebras},
  author={R. D. Schafer},
  journal={Bulletin of the American Mathematical Society},
  year={1955},
  volume={61},
  pages={469-484}
}

@article{Siegel:1999,
author = {Siegel, Stephen F. and Witherspoon, Sarah J.},
title = {The Hochschild Cohomology Ring of a Group Algebra},
journal = {Proceedings of the London Mathematical Society},
volume = {79},
number = {1},
pages = {131-157},
doi = {https://doi.org/10.1112/S0024611599011958},
url = {https://londmathsoc.onlinelibrary.wiley.com/doi/abs/10.1112/S0024611599011958},
eprint = {https://londmathsoc.onlinelibrary.wiley.com/doi/pdf/10.1112/S0024611599011958},
year = {1999}
}

@book{Wagemann+2021,
author = {Friedrich Wagemann},
doi = {doi:10.1515/9783110750959},
url = {https://doi.org/10.1515/9783110750959},
title = {Crossed Modules},
year = {2021},
publisher = {De Gruyter},
ISBN = {9783110750959},
lastchecked = {2022-06-28}
}

@article{Douglas:2018,
  title={Fusion 2-categories and a state-sum invariant for 4-manifolds},
  author={Christopher L. Douglas and David J. Reutter},
  journal={arXiv: Quantum Algebra},
  year={2018},
  url={https://api.semanticscholar.org/CorpusID:119305837}
}

@article{Baez:2012,
	doi = {10.1090/s0065-9266-2012-00652-6},
  
	url = {https://doi.org/10.1090%2Fs0065-9266-2012-00652-6},
  
	year = 2012,
	publisher = {American Mathematical Society ({AMS})},
  
	volume = {219},
  
	number = {1032},
  
	author = {John Baez and Aristide Baratin and Laurent Freidel and Derek Wise},
  
	title = {Infinite-Dimensional Representations of 2-Groups},
  
	journal = {Memoirs of the American Mathematical Society}
}

@article{Barrett1993,
    author = "Barrett, John W. and Westbury, Bruce W.",
    title = "{Invariants of piecewise linear three manifolds}",
    eprint = "hep-th/9311155",
    archivePrefix = "arXiv",
    doi = "10.1090/S0002-9947-96-01660-1",
    journal = "Trans. Am. Math. Soc.",
    volume = "348",
    pages = "3997--4022",
    year = "1996"
}

@article{Majid:1990bt,
    author = "Majid, Shahn",
    title = "{Tannaka-Krein theorem for quasiHopf algebras and other results}",
    reportNumber = "DAMTP-90-34",
    month = "9",
    year = "1990"
}

@article{Majid:1994nw,
    author = "Majid, S.",
    title = "{Some remarks on the quantum double}",
    eprint = "hep-th/9409056",
    archivePrefix = "arXiv",
    reportNumber = "DAMTP-94-70",
    doi = "10.1007/BF01690458",
    journal = "Czech. J. Phys.",
    volume = "44",
    pages = "1059",
    year = "1994"
}

@misc{Angulo:2018,
  doi = {10.48550/arxiv.1810.05740},
  
  url = {https://arxiv.org/abs/1810.05740},
  
  author = {Angulo, Camilo},
  
  title = {A cohomology theory for Lie 2-algebras and Lie 2-groups},
  
  publisher = {arXiv},
  
  year = {2018},
  
  copyright = {arXiv.org perpetual, non-exclusive license}
}

@article{Nikolaus_2014,
	doi = {10.1007/s40062-014-0077-4},
  
	url = {https://doi.org/10.1007%2Fs40062-014-0077-4},
  
	year = 2014,
	month = {mar},
  
	publisher = {Springer Science and Business Media {LLC}
},
  
	volume = {10},
  
	number = {3},
  
	pages = {565--622},
  
	author = {Thomas Nikolaus and Urs Schreiber and Danny Stevenson},
  
	title = {Principal $\infty$-bundles: presentations},
  
	journal = {Journal of Homotopy and Related Structures}
}

@article{Johnson-Freyd:2020,
    author = "Johnson-Freyd, Theo",
    title = "{(3+1)D topological orders with only a $\mathbb{Z}_2$-charged particle}",
    eprint = "2011.11165",
    archivePrefix = "arXiv",
    primaryClass = "math.QA",
    month = "11",
    year = "2020"
}

@article{KitaevKong_2012,
	doi = {10.1007/s00220-012-1500-5},
  
	url = {https://doi.org/10.1007%2Fs00220-012-1500-5},
  
	year = 2012,
	month = {jun},
  
	publisher = {Springer Science and Business Media {LLC}
},
  
	volume = {313},
  
	number = {2},
  
	pages = {351--373},
  
	author = {Alexei Kitaev and Liang Kong},
  
	title = {Models for Gapped Boundaries and Domain Walls},
  
	journal = {Communications in Mathematical Physics}
}

@article{Kong:2020,
  title = {Algebraic higher symmetry and categorical symmetry: A holographic and entanglement view of symmetry},
  author = {Kong, Liang and Lan, Tian and Wen, Xiao-Gang and Zhang, Zhi-Hao and Zheng, Hao},
  journal = {Phys. Rev. Research},
  volume = {2},
  issue = {4},
  pages = {043086},
  numpages = {53},
  year = {2020},
  month = {Oct},
  publisher = {American Physical Society},
  doi = {10.1103/PhysRevResearch.2.043086},
  url = {https://link.aps.org/doi/10.1103/PhysRevResearch.2.043086}
}

@article{Cui_2017,
	doi = {10.1142/s0218216517501048},
  
	url = {https://doi.org/10.1142%2Fs0218216517501048},
  
	year = 2017,
	month = {dec},
  
	publisher = {World Scientific Pub Co Pte Lt},
  
	volume = {26},
  
	number = {14},
  
	pages = {1750104},
  
	author = {Shawn X. Cui and Zhenghan Wang},
  
	title = {State sum invariants of three manifolds from spherical multi-fusion categories},
  
	journal = {Journal of Knot Theory and Its Ramifications}
}

@article{Bai:2007,
   title={A unified algebraic approach to the classical Yang–Baxter equation},
   volume={40},
   ISSN={1751-8121},
   url={http://dx.doi.org/10.1088/1751-8113/40/36/007},
   DOI={10.1088/1751-8113/40/36/007},
   number={36},
   journal={Journal of Physics A: Mathematical and Theoretical},
   publisher={IOP Publishing},
   author={Bai, Chengming},
   year={2007},
   month={Aug},
   pages={11073–11082} }

@ARTICLE{Adler:1978,
       author = {{Adler}, M.},
        title = "{On a Trace Functional for Formal Pseudo-Differential Operators and the Symplectic Structure of the Korteweg-Devries Type Equation.}",
      journal = {Inventiones Mathematicae},
         year = 1978,
        month = jan,
       volume = {50},
        pages = {219},
          doi = {10.1007/BF01410079},
       adsurl = {https://ui.adsabs.harvard.edu/abs/1978InMat..50..219A}
}

@article{Zhang:1991,
author = {M. Q. Zhang},
title = {{How to find the Lax pair from the Yang-Baxter equation}},
volume = {141},
journal = {Communications in Mathematical Physics},
number = {3},
publisher = {Springer},
pages = {523 -- 531},
year = {1991},
doi = {cmp/1104248391},
URL = {https://doi.org/}
}

@article{Tomei:2013,
title = {The Toda lattice, old and new},
journal = {Journal of Geometric Mechanics},
volume = {5},
number = {4},
pages = {511-530},
year = {2013},
author = {Carlos Tomei},
}

@book{book-integrable, place={Cambridge}, series={Cambridge Monographs on Mathematical Physics}, title={Introduction to Classical Integrable Systems}, DOI={10.1017/CBO9780511535024}, publisher={Cambridge University Press}, author={Olivier Babelon, Denis Bernard and Michel Talon}, year={2003}, collection={Cambridge Monographs on Mathematical Physics}}

@article{grosse2001suggestion,
    author = "Grosse, Harald and Schlesinger, Karl-Georg",
    title = "{A Suggestion for an integrability notion for two-dimensional spin systems}",
    eprint = "hep-th/0103176",
    archivePrefix = "arXiv",
    reportNumber = "UWTHPH-2001-12",
    doi = "10.1023/A:1010988421217",
    journal = "Lett. Math. Phys.",
    volume = "55",
    pages = "161--167",
    year = "2001"
}

@article{Zhu:2019,
   title={Topological nonlinear $\sigma$-model, higher gauge theory, and a systematic construction of (3+1)D topological orders for boson systems},
   volume={100},
   ISSN={2469-9969},
   url={http://dx.doi.org/10.1103/PhysRevB.100.045105},
   doi={10.1103/physrevb.100.045105},
   number={4},
   journal={Physical Review B},
   publisher={American Physical Society (APS)},
   author={Zhu, Chenchang and Lan, Tian and Wen, Xiao-Gang},
   year={2019},
   month={Jul} }

@misc{Ang2018,
      title={Higher categorical groups and the classification of topological defects and textures}, 
      author={J.P. Ang and Abhishodh Prakash},
      year={2018},
      eprint={1810.12965},
      archivePrefix={arXiv},
      primaryClass={math-ph}
}

@article{Pfeiffer2007,
title = {2-Groups, trialgebras and their Hopf categories of representations},
journal = {Advances in Mathematics},
volume = {212},
number = {1},
pages = {62-108},
year = {2007},
issn = {0001-8708},
doi = {https://doi.org/10.1016/j.aim.2006.09.014},
url = {https://www.sciencedirect.com/science/article/pii/S0001870806003343},
author = {Hendryk Pfeiffer},
eprint = "0411468",
    archivePrefix = "arXiv",
    primaryClass = "math-ph",
}

@article{chen:2022,
  title = {Categorified Drinfel'd double and $BF$ theory: 2-groups in 4D},
  author = {Chen, Hank and Girelli, Florian},
  journal = {Phys. Rev. D},
  volume = {106},
  issue = {10},
  pages = {105017},
  numpages = {35},
  year = {2022},
  month = {Nov},
  publisher = {American Physical Society},
  doi = {10.1103/PhysRevD.106.105017},
  url = {https://link.aps.org/doi/10.1103/PhysRevD.106.105017}
}

@article{Mackaay:hc,
	author = {Marco Mackaay and Roger Picken},
	eprint = {math/0104285},
	title = {2-Categories, 4d state-sum models and gerbes},
	url = {https://arxiv.org/pdf/math/0104285.pdf}}

@article{Mackaay:ek,
	author = {Marco Mackaay},
	eprint = {math/9903003},
	journal = {Adv. Math. 153(2) (2000), 353-390},
	title = {Finite groups, spherical 2-categories, and 4-manifold invariants},
	url = {https://arxiv.org/pdf/math/9903003.pdf},
	year={2000}}

@article{Mackaay:uo,
	author = {Marco Mackaay},
	eprint = {math/9805030},
	journal = {Adv. Math. 143 (1999), 288-348},
	title = {Spherical 2-categories and 4-manifold invariants},
	url = {https://arxiv.org/pdf/math/9805030.pdf},
	year={1999}}

@article{Wag:2006,
	author = {Friedrich Wagemann},
	doi = {10.1080/00927870500542705},
	eprint = {https://doi.org/10.1080/00927870500542705},
	journal = {Communications in Algebra},
	number = {5},
	pages = {1699-1722},
	publisher = {Taylor & Francis},
	title = {On Lie Algebra Crossed Modules},
	url = {https://doi.org/10.1080/00927870500542705},
	volume = {34},
	year = {2006}}

@article{Baez:1995xq,
	archiveprefix = {arXiv},
	author = {Baez, J. C. and Dolan, J.},
	doi = {10.1063/1.531236},
	eprint = {q-alg/9503002},
	journal = {J. Math. Phys.},
	pages = {6073--6105},
	title = {{Higher dimensional algebra and topological quantum field theory}},
	volume = {36},
	year = {1995}}

@article{Baez:2004in,
	archiveprefix = {arXiv},
	author = {Baez, John and Schreiber, Urs},
	eprint = {hep-th/0412325},
	month = {12},
	title = {{Higher gauge theory: 2-connections on 2-bundles}},
	year = {2004}}

@article{Baez:2003fs,
	archiveprefix = {arXiv},
	author = {Baez, John C. and Crans, Alissa S.},
	eprint = {math/0307263},
	journal = {Theor. Appl. Categor.},
	pages = {492--528},
	title = {{Higher-Dimensional Algebra VI: Lie 2-Algebras}},
	volume = {12},
	year = {2004}}

@article{Baez:2012bn,
	archiveprefix = {arXiv},
	author = {Baez, John C. and Wise, Derek K.},
	doi = {10.1007/s00220-014-2178-7},
	eprint = {1204.4339},
	journal = {Commun. Math. Phys.},
	number = {1},
	pages = {153--186},
	primaryclass = {gr-qc},
	title = {{Teleparallel Gravity as a Higher Gauge Theory}},
	volume = {333},
	year = {2015}}

@article{Baez:2010ya,
	archiveprefix = {arXiv},
	author = {Baez, John C. and Huerta, John},
	doi = {10.1007/s10714-010-1070-9},
	eprint = {1003.4485},
	journal = {Gen. Rel. Grav.},
	pages = {2335--2392},
	primaryclass = {hep-th},
	title = {{An Invitation to Higher Gauge Theory}},
	volume = {43},
	year = {2011}}

@article{Drinfeld:1986in,
	author = {Drinfeld, V. G.},
	doi = {10.1007/BF01247086},
	journal = {Zap. Nauchn. Semin.},
	pages = {18--49},
	title = {{Quantum groups}},
	volume = {155},
	year = {1986}}

@article{Crane:1994ty,
	archiveprefix = {arXiv},
	author = {Crane, Louis and Frenkel, Igor},
	doi = {10.1063/1.530746},
	eprint = {hep-th/9405183},
	journal = {J. Math. Phys.},
	pages = {5136--5154},
	title = {{Four-dimensional topological field theory, Hopf categories, and the canonical bases}},
	volume = {35},
	year = {1994}}

@article{Majid:2012gy,
	archiveprefix = {arXiv},
	author = {Majid, Shahn},
	eprint = {1208.6265},
	month = {8},
	primaryclass = {math.QA},
	title = {{Strict quantum 2-groups}},
	year = {2012}}

@article{Meusburger:2021cxe,
	author = {Meusburger, Catherine},
	doi = {10.3390/sym13081324},
	journal = {Symmetry},
	number = {8},
	pages = {1324},
	title = {{Poisson\textendash{}Lie Groups and Gauge Theory}},
	volume = {13},
	year = {2021}}

@article{Chen:2013,
	author = {Chen, Zhuo and Sti{\'e}non, Mathieu and Xu, Ping},
	doi = {10.1016/j.geomphys.2013.01.006},
	issn = {0393-0440},
	journal = {Journal of Geometry and Physics},
	month = {Jun},
	pages = {59--68},
	publisher = {Elsevier BV},
	title = {Weak Lie 2-bialgebras},
	url = {http://dx.doi.org/10.1016/j.geomphys.2013.01.006},
	volume = {68},
	year = {2013}}

@article{Chen:2012gz,
	archiveprefix = {arXiv},
	author = {Chen, Zhuo and Sti{\'e}non, Mathieu and Xu, Ping},
	doi = {10.4310/jdg/1367438648},
	eprint = {1202.0079},
	journal = {J. Diff. Geom.},
	number = {2},
	pages = {209--240},
	primaryclass = {math.DG},
	title = {{Poisson 2-groups}},
	volume = {94},
	year = {2013}}

@article{Osei:2013xra,
	archiveprefix = {arXiv},
	author = {Osei, Prince K and Schroers, Bernd J},
	doi = {10.1063/1.4824704},
	eprint = {1307.6485},
	journal = {J. Math. Phys.},
	pages = {101702},
	primaryclass = {math-ph},
	reportnumber = {EMPG-13-12},
	title = {{Classical r-matrices via semidualisation}},
	volume = {54},
	year = {2013}}

@Inbook{Kapustin:2013uxa,
author="Kapustin, Anton
and Thorngren, Ryan",
editor="Auroux, Denis
and Katzarkov, Ludmil
and Pantev, Tony
and Soibelman, Yan
and Tschinkel, Yuri",
title="Higher Symmetry and Gapped Phases of Gauge Theories",
bookTitle="Algebra, Geometry, and Physics in the 21st Century: Kontsevich Festschrift",
year="2017",
publisher="Springer International Publishing",
address="Cham",
pages="177--202",
isbn="978-3-319-59939-7",
doi="10.1007/978-3-319-59939-7_5",
url="https://doi.org/10.1007/978-3-319-59939-7_5"
}

@article{Baez:2005sn,
author = {John C. Baez and Danny Stevenson and Alissa S. Crans and Urs Schreiber},
title = {{From loop groups to 2-groups}},
volume = {9},
journal = {Homology, Homotopy and Applications},
number = {2},
publisher = {International Press of Boston},
pages = {101 -- 135},
year = {2007},
}

@article{Delcamp:2016yix,
	archiveprefix = {arXiv},
	author = {Delcamp, Clement and Dittrich, Bianca and Riello, Aldo},
	doi = {10.1007/JHEP02(2017)061},
	eprint = {1607.08881},
	journal = {JHEP},
	pages = {061},
	primaryclass = {hep-th},
	title = {{Fusion basis for lattice gauge theory and loop quantum gravity}},
	volume = {02},
	year = {2017}}

@article{Bai_2013,
	author = {Chengming Bai and Yunhe Sheng and Chenchang Zhu},
	doi = {10.1007/s00220-013-1712-3},
	journal = {Communications in Mathematical Physics},
	month = {apr},
	number = {1},
	pages = {149--172},
	publisher = {Springer Science and Business Media {LLC}},
	title = {Lie 2-Bialgebras},
	url = {https://doi.org/10.1007%2Fs00220-013-1712-3},
	volume = {320},
	year = 2013}

@article{Dupuis:2020ndx,
    author = {Dupuis, Ma{\" i}t{\'e} and Freidel, Laurent and Girelli, Florian and Osumanu, Abdulmajid and Rennert, Julian},
    title = "{Origin of the quantum group symmetry in 3D quantum gravity}",
    eprint = "2006.10105",
    archivePrefix = "arXiv",
    primaryClass = "gr-qc",
    doi = "10.1103/fr9h-3w7t",
    journal = "Phys. Rev. D",
    volume = "112",
    number = "8",
    pages = "084071",
    year = "2025"
}

@book{Baez:2008hz,
	archiveprefix = {arXiv},
	author = {Baez, John C. and Baratin, Aristide and Freidel, Laurent and Wise, Derek K.},
	doi = {10.1090/S0065-9266-2012-00652-6},
	eprint = {0812.4969},
	isbn = {978-0-8218-7284-0, 978-0-8218-9116-2},
	primaryclass = {math.QA},
	title = {{Infinite-Dimensional Representations of 2-Groups}},
	volume = {1032},
	year = {2012},
	publisher={American Mathematical Society}}

@article{Witten:1988hc,
	author = {Witten, Edward},
	doi = {10.1016/0550-3213(88)90143-5},
	journal = {Nucl. Phys.},
	pages = {46},
	reportnumber = {IASSNS-HEP-88-32},
	slaccitation = {%%CITATION = NUPHA,B311,46;%%},
	title = {{(2+1)-Dimensional Gravity as an Exactly Soluble System}},
	volume = {B311},
	year = {1988}}

@article{Turaev:1992hq,
	author = {Turaev, V. G. and Viro, O. Y.},
	doi = {10.1016/0040-9383(92)90015-A},
	journal = {Topology},
	pages = {865-902},
	slaccitation = {%%CITATION = TPLGA,31,865;%%},
	title = {{State sum invariants of 3 manifolds and quantum 6j symbols}},
	volume = {31},
	year = {1992}}

@book{Majid:1996kd,
	author = {Majid, S.},
	isbn = {9780511834530, 9780521648684},
	publisher = {Cambridge University Press},
	slaccitation = {%%CITATION = INSPIRE-427381;%%},
	title = {{Foundations of quantum group theory}},
	year = {2011}}

@article{Kitaev1997,
	archiveprefix = {arXiv},
	author = {Kitaev, A. {\relax Yu}.},
	doi = {10.1016/S0003-4916(02)00018-0},
	eprint = {quant-ph/9707021},
	journal = {Annals Phys.},
	pages = {2-30},
	primaryclass = {quant-ph},
	slaccitation = {%%CITATION = QUANT-PH/9707021;%%},
	title = {{Fault tolerant quantum computation by anyons}},
	volume = {303},
	year = {2003}}

@article{Bullivant:2016clk,
	archiveprefix = {arXiv},
	author = {Bullivant, A. and Calcada, M. and Kadar, Z. and Martin, P. and Martins, J.},
	doi = {10.1103/PhysRevB.95.155118},
	eprint = {1606.06639},
	journal = {Phys. Rev.},
	number = {15},
	pages = {155118},
	primaryclass = {cond-mat.str-el},
	slaccitation = {%%CITATION = ARXIV:1606.06639;%%},
	title = {Topological phases from higher gauge symmetry in 3+1 dimensions},
	volume = {B95},
	year = {2017}}

@article{Semenov1992,
	author = {Semenov-Tyan-Shanskii, M. A.},
	day = {01},
	doi = {10.1007/BF01083527},
	issn = {1573-9333},
	journal = {Theoretical and Mathematical Physics},
	month = {Nov},
	note = {\url{https://doi.org/10.1007/BF01083527}},
	number = {2},
	pages = {1292--1307},
	title = {{Poisson-Lie groups. The quantum duality principle and the twisted quantum double}},
	volume = {93},
	year = {1992}}
\end{document}